\documentclass[a4paper,11pt]{article}
\usepackage[english]{babel}
\usepackage[top=3cm, bottom=3cm, left=2.3cm, right=2.3cm]{geometry}
\usepackage{microtype}
\usepackage{url}
\usepackage{verbatim}
\usepackage{amsmath,amssymb,amsthm,graphicx,stmaryrd,enumerate,bbm,hyperref,latexsym,calc,capt-of,xcolor,ifthen,mathrsfs}
\usepackage{color}
\usepackage{hyperref}
\hypersetup{
    colorlinks=true,
    linkcolor=blue,
    citecolor=magenta,   
    urlcolor=blue,
    }
\newcommand{\assign}{:=}
\newcommand{\backassign}{=:}
\newcommand{\cdummy}{\cdot}
\newcommand{\longhookrightarrow}{{\lhook\joinrel\relbar\joinrel\rightarrow}}
\newcommand{\mathLaplace}{\Delta}
\newcommand{\mathd}{\mathrm{d}}
\newcommand{\nin}{\not\in}
\newcommand{\nobracket}{}
\newcommand{\nocomma}{}
\newcommand{\tmcolor}[2]{{\color{#1}{#2}}}
\newcommand{\tmdummy}{$\mbox{}$}
\newcommand{\tmem}[1]{{\em #1\/}}
\newcommand{\tmop}[1]{\ensuremath{\operatorname{#1}}}
\newcommand{\tmrsub}[1]{\ensuremath{_{\textrm{#1}}}}
\newcommand{\tmtextit}[1]{\text{{\itshape{#1}}}}

\newenvironment{enumeratealpha}{\begin{enumerate}[a{\textup{)}}] }{\end{enumerate}}
\newenvironment{enumerateroman}{\begin{enumerate}[i.] }{\end{enumerate}}

\theoremstyle{plain}
\newtheorem{theorem}{Theorem}
\newtheorem{lemma}[theorem]{Lemma}
\newtheorem{proposition}[theorem]{Proposition}
\newtheorem{corollary}[theorem]{Corollary}

\theoremstyle{definition}
\newtheorem{definition}[theorem]{Definition} 
 
\theoremstyle{remark}
\newtheorem{remark}[theorem]{Remark}
\newtheorem{example}[theorem]{Example}

\newcommand{\tmfloatcontents}{}
\newlength{\tmfloatwidth}
\newcommand{\tmfloat}[5]{
  \renewcommand{\tmfloatcontents}{#4}
  \setlength{\tmfloatwidth}{\widthof{\tmfloatcontents}+1in}
  \ifthenelse{\equal{#2}{small}}
    {\setlength{\tmfloatwidth}{0.45\linewidth}}
    {\setlength{\tmfloatwidth}{\linewidth}}
  \begin{minipage}[#1]{\tmfloatwidth}
    \begin{center}
      \tmfloatcontents
      \captionof{#3}{#5}
    \end{center}
  \end{minipage}}

\newcommand{\1}{\mathbbm{1}}

\begin{document}

\title{Weak well-posedness of energy solutions to singular SDEs with
supercritical distributional drift}
\author{
  Lukas Gräfner\\
  Freie Universität Berlin\\
  Institut f\"ur Mathematik \\
  \texttt{l.graefner@fu-berlin.de}
\and
  Nicolas Perkowski\\
  Freie Universität Berlin\\
  Institut f\"ur Mathematik \\
  \texttt{perkowski@math.fu-berlin.de}
}
\maketitle

\begin{abstract}
  We study stochastic differential equations with additive noise and
  distributional drift on $\mathbb{T}^d$ or $\mathbb{R}^d$ and $d \geqslant
  2$. We work in a scaling-supercritical regime using energy solutions and
  recent ideas for generators of singular stochastic partial differential
  equations. We mainly focus on divergence-free drift, but allow for
  scaling-critical non-divergence free perturbations. In the time-dependent
  divergence-free case we roughly speaking prove weak well-posedness of energy
  solutions with initial law $\mu \ll \tmop{Leb}$ for drift $b \in L^p_T B^{-
  \gamma}_{p, 1}$ with $p \in (2, \infty]$ and $p \geqslant \frac{2}{1 -
  \gamma}$. For time-independent $b$ we show weak well-posedness of energy
  solutions with initial law $\mu \ll \tmop{Leb}$ under certain structural
  assumptions on $b$ which allow local singularities such that $b \nin B^{-
  1}_{2 d/(d - 2), 2}$, meaning that for any $p > 2$ in sufficiently high dimension there exists $b \nin B^{- 1}_{p, 2}$ such
  that weak well-posedness holds for energy solutions with drift $b$.
  
\end{abstract}

{\tableofcontents}

\section{Introduction}

We study the weak well-posedness and Markovianity of solutions to the SDE
\begin{equation}
  \mathd X_t = b (t, X_t) \mathd t + \sqrt{2} \mathd B_t, \qquad X_0 \sim \mu,
  \label{MainSDE}
\end{equation}
on $\mathbb{R}^d$ or $\mathbb{T}^d$, where $b$ is a singular drift
coefficient. Motivated by recent progress in scaling (sub-)critical singular
stochastic partial differential
equations~{\cite{Hairer2014,Gubinelli2015Paracontrolled,GubinelliPerkowski20}}
and by the modeling of diffusion in random media by SDEs with distributional
drift~{\cite{DelarueDiel16,CannizzaroChouk18,KrempPerkowski22,CannizzaroHaunschmidSibitzToninelli21}},
we are interested in scaling critical and supercritical distributional drift
$b$. Based on Hairer's regularity structures there is now a well-developed
pathwise solution and renormalization theory for singular SPDEs under a
scaling subcriticality condition. Subcriticality is crucial in this pathwise
approach, because it allows to interpret nonlinear equations as perturbations
of linearized equations. Using the probabilistic notion of energy solutions,
it is possible to prove weak existence for some relevant scaling supercritical
SPDEs~{\cite{GoncalvesJara14,GubinelliJara13}} and weak uniqueness in the
critical case~{\cite{GraefnerPerkowski23}}, but the weak uniqueness of
supercritical energy solutions is still open. Here we prove weak uniqueness of
energy solutions to certain supercritical finite-dimensional SDEs.

The criticality of equation \eqref{MainSDE} for a generic $b \in L^q_T B^s_{p,
r}$ is determined by rescaling the equation diffusively as $X^{\varepsilon}
(t) = \varepsilon^{- 1} X_{\varepsilon^2 t}$, i.e. in such a way that the law
of the Brownian motion is preserved, and then computing the behaviour of $\|
b^{\varepsilon} \|_{L^q_T B^s_{p, r}}$, where $b^{\varepsilon}$ is the
rescaled drift. Intuitively, if this quantity blows up in the limit
$\varepsilon \rightarrow 0$, \ the regularizing effects of the Brownian noise
cannot compensate the singularity of the drift. In that case, equation
\eqref{MainSDE} is considered {\tmem{supercritical}}. This leads to the
relation
\begin{equation}
  \frac{d}{p} + \frac{2}{q} > 1 + s, \label{criticalityrelation}
\end{equation}
for supercriticality (see e.g.
{\cite{BeckFlandoliGubinelliMaurelli14,HaoZhang23}}). Likewise, the equation
is considered {\tmem{critical}} or {\tmem{subcritical}} if $>$ is replaced by
$=$ or $<$, respectively.

Our intuition is that while scaling subcriticality is closely related to
well-posedness of stochastic differential equations, it is not the decisive
criterion. There are scaling subcritical problems that cannot be given direct
sense~{\cite{Coutin2002,Hairer2024}}, and equations that are supercritical at
least from a regularity counting perspective but still possess weakly unique
solutions~{\cite{Perkins2002,Konarovskyi2019}}. This also aligns with the
theory of renormalized solutions to ODEs~{\cite{DiPerna1989}}, where
well-posedness of the ODE $\dot{x} = b (x)$ is shown (among others) for $b \in
W^{1, 1} \supset B^1_{1, 1}$, i.e. $p = 1$ and $s = 1$, despite the scaling
subcriticality relation $\frac{d}{p} < 1 + s$ being violated in $d \geqslant
2$. Yet, even without noise the equation is well-posed, and this can be
extended to include noise~{\cite{Figalli2008,LeBris2019}}, further
demonstrating that scaling criticality and well-posedness do not need to
coincide.

There is extensive work on SDE \eqref{MainSDE} under low regularity
assumptions on $b$, in particular for functions $b$ in Lebesgue spaces.
Influential works include
{\cite{Zvonkin1974,Veretennikov1981,Krylov2005,Flandoli2010}}, which prove
strong well-posedness and the existence of stochastic flows for $b$ in scaling
subcritical Lebesgue spaces, see also the monographs
{\cite{Flandoli2011,Lee2022}} or the recent work {\cite{AnzelettiLeLing23}}.
For distributional $b$ there are typically only weak well-posedness results,
e.g. {\cite{Flandoli2003,FlandoliIssoglioRusso17}}, with the notable exception
of the one-dimensional case~{\cite{Bass2001}}, and even in the subcritical
case it may be necessary to enrich $b$ with higher order ``rough path type''
data~{\cite{DelarueDiel16,CannizzaroChouk18,KrempPerkowski22,Kremp2023}}.
Based on Nash theory or Dirichlet forms, very strong results exist for
subcritical and critical divergence-free or gradient drift, for which it is
not necessary to enrich $b$ {\cite{Osada1987,Mathieu1994,Stannat1999}}. For
functions $b$ there have been recent breakthroughs in the critical and
supercritical regime (e.g.
{\cite{BeckFlandoliGubinelliMaurelli14,RoecknerZhao23,ChenZhangZhao21,ZhangZhao21,Krylov2023}}),
but for distributional $b$ the supercritical case is worse understood. The
work most related to ours is {\cite{HaoZhang23}}, and we comment on
similarities and differences after stating our main results. To the best of
our knowledge, our work and {\cite{HaoZhang23}} contain the first
well-posedness results for equation (\ref{MainSDE}) in the distributional
supercritical regime.

The solution concept we focus on in this work is that of {\tmem{energy
solutions}}. These were initially introduced in the infinite-dimensional
setting in {\cite{GoncalvesJara14}} to describe mesoscopic scaling limits of
particle systems as solutions to singular SPDEs, but their uniqueness remained
open. A refined notion of energy solutions, introduced in
{\cite{GubinelliJara13}}, was shown to be unique for the KPZ equation in
{\cite{GubinelliPerkowski18}}, providing the first probabilistic
well-posedness result for a singular SPDE outside of the Da Prato-Debussche
regime. In {\cite{GubinelliPerkowski20}} a new approach to uniqueness via
construction of the infinitesimal generator was introduced and subsequentially
applied in {\cite{GubinelliTurra20,LuoZhu21}}. This approach has since been
further developed further and will be contained in the upcoming work
{\cite{GraefnerPerkowski23plus}} on (sub-)critical singular SPDEs, parts of
which have already been outlined in the lecture notes
{\cite{GraefnerPerkowski23}}. In the context of SDEs, energy solutions are
essentially weak solutions for which the drift is defined by an approximation
procedure and which satisfy an {\tmem{energy estimate}}. This energy estimate
can be seen as a selection principle which guarantees uniqueness, see
Remark~\ref{rmk:modena} below.

Let us state our main results. Let $M \in \{ \mathbb{T}^d, \mathbb{R}^d \}$.
Throughout the paper we consider $d$-dimensional Schwartz distributions $b$ on
$M$ of the form
\[ b = b_1 + b_2, \]
where {\tmem{always}}
\[ \nabla \cdummy b_1 \equiv 0. \]
For $b \in \mathcal{S}' (M ; \mathbb{R}^d)$ (regular enough at the origin in
Fourier space) we can find a canonical decomposition, the {\tmem{Helmholtz
decomposition $(A (b), V (b))$}} of $b$, via
\begin{equation}
  A_{i j} (b) = \mathLaplace^{- 1} (\partial_i b^j - \partial_j b^i)
  \label{Helmholtzdivfree}
\end{equation}
and
\begin{equation}
  V (b) = \mathLaplace^{- 1} \nabla \cdummy b, \label{Helmholtzdiv}
\end{equation}
so that
\begin{equation}
  b^i = b_1^i + b_2^i \assign [\nabla \cdummy A_i] + [\partial_i V + \mathbbm{1}_{M
  =\mathbb{T}^d} \hat{b} (0)^i], \label{Helmholtz}
\end{equation}
where $A_i$ is the $i$th column of the antisymmetric (distributional) matrix
field $A$. If the Fourier transform of $b$ is not sufficiently regular near
$0$ to make sense of (\ref{Helmholtzdivfree}), it may be possible to split off
the small frequencies manually. On $\mathbb{T}^d$ this does
not cause any problems and the identity (\ref{Helmholtz}) makes always sense.
Conversely, given an antisymmetric $A \in \mathcal{S}' (M ; \mathbb{R}^{d
\times d})$ it follows directly that
\begin{equation}
  b^i (A) = \nabla \cdummy A_i, \label{bofA}
\end{equation}
is such that $\nabla \cdummy b (A) = 0$.

We consider approximations of equation (\ref{MainSDE}) given by
\begin{equation}
  \mathd X_t^n = b^n (t, X_t^n) \mathd t + \sqrt{2} \mathd B_t, \label{approx}
\end{equation}
where $b^n = \rho^n \ast b$, with $\rho^n = n^{d + 1} \rho (n \cdummy)$ and
$\rho (t, x) = \rho_t (t) \rho_x (x)$, where $\rho_t, \rho_x$ are positive
mollifiers in space and time, respectively. We obtain the following
well-posedness results for {\tmem{energy solutions}}. (See Definitions
\ref{defenergysolutions} and \ref{defenergysolutionsEuclid} for the
definitions of energy solutions and (\ref{sqrtspace}) for the definition of
$B^0_{2 r, 1, 2}$).

\begin{theorem}
  \label{WellposednessBesov}Let $\mu \ll \tmop{Leb}_M$ be a probability
  measure on $M$. Let $b_i : \mathbb{R}_+ \rightarrow \mathcal{S}' (M,
  \mathbb{R}^d), i \in \{ 1, 2 \}$ such that $\nabla \cdummy b_1 \equiv 0$ and
  all of the following conditions i.-iii. are satisfied:
  \begin{enumerateroman}
    \item \label{supercriticalbesov}$b_1 = b (A)$, $A \in L^q_{\tmop{loc}}
    L^p$, or $b_1 \in L^q_{\tmop{loc}} B^{- 1}_{p, 2}$, where $p \in [2,
    \infty), q \in [2, \infty]$ and $p \left( \frac{1}{2} - \frac{1}{q}
    \right) > 1$;
    
    \item \label{supercriticalbesovunique}$b_1 = b (A)$ with $A \in
    L^{\infty}_{\tmop{loc}} L^{\infty}$ or $b_1 \in L^2_{\tmop{loc}} L^2$;
    
    \item \label{gradientassupmtionbesovuniqueness}$b_2 \in
    L^{\infty}_{\tmop{loc}} L^{\hat{r}} \cap L^4_{\tmop{loc}} B^0_{2 r, 1, 2}$
    where $r \in [d, \infty]$ and $\hat{r} \in (d, \infty]$ or $\hat{r} = d$
    and $d \geqslant 3$.
  \end{enumerateroman}
  Let $b = b_1 + b_2$. Then there exists a unique energy solution $X$ to
  equation (\ref{MainSDE}) with initial law $\mu$. The process $X$ is the
  limit in law of the sequence of solutions to (\ref{approx}), it is a Markov
  process and, if $b_2 = 0$, $\tmop{Leb}_M$ is an invariant measure for $X$.
\end{theorem}

The main ideas of a proof of Theorem \ref{WellposednessBesov} for
time-independent, divergence-free $b$ are sketched in the summer school
lecture notes {\cite{GraefnerPerkowski23}}. But the analysis there is more
similar to the one we use here for Theorem \ref{Wellposednesslocaldiverging}
below. The main ingredients for the uniqueness proof of Theorem
\ref{WellposednessBesov} are the energy estimate of energy solutions and the
following simple bounds for $b \cdummy \nabla u = \nabla \cdummy (A \cdummy
\nabla u) = \nabla \cdummy (b u)$
\begin{equation}
  \| \nabla \cdummy (A \cdummy \nabla u) \|_{H^{- 1}} \lesssim \| A 
  \|_{L^{\infty}} \| u \|_{H^1}, \qquad \| \nabla \cdummy (b u) \|_{H^{- 1}}
  \lesssim \| b  \|_{L^2} \| u \|_{L^{\infty}} . \label{easyestimate}
\end{equation}
Together with simple a priori estimates and the maximum principle for the PDE,
this estimate \ allow to link energy solutions with the corresponding
Kolmogorov backward equation. The gradient field $b_2$ is treated by
perturbative arguments: On the SDE side this is done by adding a drift using
Girsanov's theorem. On the PDE side we exploit Lemma~\ref{Laplacebound} which
shows that $b_2 \cdummy \nabla u$ is relatively bounded with respect to the
Laplacian.

\begin{remark}
  \label{remarkBesovtheorem}{\tmdummy}
  
  \begin{enumerateroman}
    \item It follows directly from the additive nature of the proofs leading
    to Theorem \ref{WellposednessBesov} that its statement is still true if $b
    = \sum_{i = 1}^n b^{(i)}$ such that each $b^{(i)}$ fulfills
    \ref{supercriticalbesov}. - \ref{gradientassupmtionbesovuniqueness}. for
    possibly different parameters.
    
    \item In particular, we can take $b_1 \in L^p_T B^{- \gamma}_{p, 1}$ for
    $p \geqslant \frac{2}{1 - \gamma}$ and $p \in [2, \infty], \gamma \in [0,
    1]$ in condition~\ref{supercriticalbesovunique}.: In the boundary case
    $\gamma = 0$ we have $b_1 \in L^2_T B^0_{2, 1} \subset L^2_T L^2$ and for
    $\gamma = 1$ we have $b_1 \in L^{\infty}_T B^{- 1}_{\infty, 1}$, so that
    with subtracting the first Littlewood-Paley block we have $A (b_1 -
    \Delta_{- 1} b_1) \in L^{\infty}_T B^0_{\infty, 1} \subset L^{\infty}_T
    L^{\infty}$, while we can add $\Delta_1 b_1 \in L^{\infty}_T C^{\infty}_b$
    to $b_2$. For $\gamma \in (0, 1)$ we can very likely use general
    interpolation space theory to obtain $L^p_T B^{- \gamma}_{p, 1} \subset
    L^2_T B^0_{2, 1} + L^{\infty}_T B^{- 1}_{\infty, 1}$, but we were unable
    to find a reference for this. However, by Theorem~1.64 in
    {\cite{Triebel2006}} we know that for a suitable wavelet basis
    $(\Psi^G_{j, m})$:
    \[ \| b_1 (t) \|_{B^s_{r, q}}^q \simeq \sum_{j = 0}^{\infty} 2^{j \left( s
       - \frac{d}{r} \right) q} \sum_{G \in G^j} \left( \sum_{m \in
       \mathbb{Z}^d} | \lambda^{j, G}_m (t) |^r \right)^{q / r}, \]
    where
    \[ b_1 (t) = \sum_{j = 0}^{\infty} \sum_{G \in G^j} \sum_{m \in
       \mathbb{Z}^d} \lambda^{j, G}_m (t) 2^{- j d / 2} \Psi^{j, G}_m . \]
    Then we define (with $\frac{0}{0} \assign 0$)
    \begin{eqnarray*}
      \mathcal{I}_1 (t) & = & \left\{ (j, G, m) : | \lambda^{j, G}_m (t) |
      \leqslant 2^{j (1 - \gamma)} \frac{M_{j, G} (t)}{\| b_1 (t) \|_{B^{-
      \gamma}_{p, 1}}} \| b_1 \|_{L^p_T B^{- \gamma}_{p, 1}} \right\},\\
      \mathcal{I}_2 (t) & = & \left\{ (j, G, m) : | \lambda^{j, G}_m (t) | >
      2^{j (1 - \gamma)} \frac{M_{j, G} (t)}{\| b_1 (t) \|_{B^{- \gamma}_{p,
      1}}} \| b_1 \|_{L^p_T B^{- \gamma}_{p, 1}} \right\},
    \end{eqnarray*}
    where $M_{j, G} (t) = (2^{- j d} \sum_{m \in \mathbb{Z}^d} | \lambda^{j,
    G}_m (t) |^p)^{\frac{1}{p}}$, and $b_1 = f_1 + f_2$ with
    \[ f_i (t) = \sum_{(j, G, m) \in \mathcal{I}_i (t)} \lambda^{j, G}_m (t)
       2^{- j d / 2} \Psi^{j, G}_m . \]
    It follows from a direct computation that $f_1 \in L^{\infty}_T B^{-
    1}_{\infty, 1}$ and $f_2 \in L^2_T B^0_{2, 1}$. For $p > \frac{2}{1 -
    \gamma}$ we simply use $L^p_T B_{p, 1}^{- \gamma} \subset L^p_T B_{p,
    1}^{- \hat{\gamma}}$, where $p = \frac{2}{1 - \hat{\gamma}}$ so that
    $\hat{\gamma} > \gamma$.
    
    \item Revisiting condition (\ref{criticalityrelation}), we see that
    \ref{supercriticalbesov}., \ref{supercriticalbesovunique}. and the
    preceding arguments allow the drift $b$ to be supercritical in $d
    \geqslant 2$. If we choose $q = \infty$, we need $d \geqslant 3$. The
    regularity for $b_2$ is critical so that in particular Theorem
    \ref{WellposednessBesov} implies a well-posedness result for a class of
    critical drift without restrictions on the divergence.
    
    \item \label{localizationforinfty}For $M =\mathbb{R}^d$ the case $b_1 = b
    (A)$ with $A \in L^{\infty}_T L^{\infty}$ is excluded in condition
    \ref{supercriticalbesov}., which requires some decay at infinity from $A$.
    This is for technical convenience, and we expect that it is possible to
    allow $p = \infty$ in condition \ref{supercriticalbesov}. by a
    localization argument: Roughly speaking, if $\eta_n$ is a smooth cutoff
    function supported in $B (0, 2 n)$ and which is constantly equal $1$ in $B
    (0, n)$, then we get from the It{\^o} trick (considering time-independent
    $A$ for simplicity)
    \begin{eqnarray*}
      &  & \mathbb{P} \left( \sup_{t \leqslant T} \left| \int_0^t b (A) (X_s)
      \mathd s \right| \geqslant \frac{n}{2}, | X_0 | + \sup_{t \leqslant T} |
      B_t | \leqslant \frac{n}{2} \right)\\
      &  & \leqslant \mathbb{P} \left( \sup_{t \leqslant T} \left| \int_0^t
      (b (\eta_n A)) (X_s) \mathd s \right| \geqslant \frac{n}{2} \right)\\
      &  & \leqslant \frac{\mathbb{E} [\sup_{t \leqslant T} | \int_0^t (b
      (\eta_n A)) (s, X_s) \mathd s|^p]}{\left( \frac{n}{2} \right)^p}\\
      &  & \lesssim_T \frac{\| A|_{B (0, 2 n)} \|_{{L^p} }^p}{n^p} \lesssim
      \| A \|_{L^{\infty}}^p n^{d - p},
    \end{eqnarray*}
    so we can localize the diffusion by taking $p > d$. To not distract from
    the main ideas of the work, we prefer to not give the details.
  \end{enumerateroman}
\end{remark}

\begin{remark}
  \label{rmk:modena}In~{\cite{Modena2018}}, Theorem~1.9, the authors use
  convex integration techniques to construct for every $d \geqslant 3$ and $p
  \in (1, \infty)$, $\tilde{p} \in [1, \infty)$ with ($p'$ is the H{\"o}lder
  conjugate exponent of $p$)
  \[ \frac{1}{p} + \frac{1}{\tilde{p}} > 1 + \frac{1}{d - 1}, \qquad p' < d -
     1, \]
  a divergence-free drift $b \in C ([0, T] ; W^{1, \tilde{p}} (\mathbb{T}^d)
  \cap L^{p'} (\mathbb{T}^d))$ and an initial density $\rho_0 \in L^p
  (\mathbb{T}^d)$ such that there are multiple weak solutions $\rho$ to the
  Fokker-Planck equation
  \[ \partial_t \rho = \Delta \rho - \nabla \cdummy (b \rho), \qquad \rho (0)
     = \rho_0 . \]
  Morally, this corresponds to weak non-uniqueness for the SDE. For $d
  \geqslant 4$ we can take $p = p' = 2$ and $\tilde{p} = 1$, and for $b \in C
  ([0, T] ; L^2 (\mathbb{T}^d))$ we obtain from
  Theorem~\ref{WellposednessBesov} the weak well-posedness of energy solutions
  with initial density $\rho \in L^1 (\mathbb{T}^d)$ (even less integrability
  than $\rho \in L^2 (\mathbb{T}^d)$) and even without assuming $b \in C ([0,
  T] ; W^{1, 2} (\mathbb{T}^d))$. Therefore, the energy estimate in the
  definition of energy solutions can be interpreted as a selection principle,
  an additional assumptions which guarantees uniqueness of weak solutions. See
  Definition~1 on p.54 of~{\cite{LeBris2019}} for a different type of
  selection principle for this problem. 
\end{remark}

For time-independent $b$ and corresponding to the approach from
(\ref{MainKBEb}), we obtain the following result (see Definition
\ref{defmartingalesol} for the definition of the martingale problem).

\begin{theorem}
  \label{Wellposednesslocaldiverging}Let $\mu \ll \tmop{Leb}$ be a probability
  measure on $M$. Let $b_i \in \mathcal{S}' (M, \mathbb{R}^d), i \in \{ 1, 2
  \}$ with $b_1 = b_1 (A)$ fulfill the following properties.
  \begin{enumerateroman}
    \item $A \in L^p$, $p \in (2, \infty]$, and there exists a sequence $(g_n)
    \subset L^{\infty}$ such that $\nabla g_n \in L^2 \cap L^{\frac{2 p}{p -
    2}}$ and
    \begin{enumeratealpha}
      \item $\nabla g_n \rightarrow 0$ weakly in $L^2$ and $\int g_n h
      \rightarrow \int h$ for all $h \in L^1$.\label{lemmaassumptionimain}
      
      \item $g_n A \in L^{\infty}$, $n \in
      \mathbb{N}$.\label{lemmaassumptioniimain}
      
      \item $A \cdummy \nabla g_n \rightarrow 0$ weakly in
      $L^2$.\label{lemmaassumptioniiimain}
    \end{enumeratealpha}
    \item $b_2 \in L^{\hat{r}} \cap B^0_{2 r, 1, 2} $with $r, \hat{r} \in [d,
    \infty]$, where for $\hat{r} = d$ we assume $d \geqslant 3$ and $\| b
    \|_{L^d} \leqslant K_d$ for a certain constant $K_d$ depending only on
    $d$.
  \end{enumerateroman}
  Let $b = b_1 + b_2$. Then there exists a unique energy solution $X$ to
  equation (\ref{MainSDE}) with initial law $\mu$. The process $X$ is the weak
  limit of the sequence of solutions to (\ref{approx}), it is a Markov
  process, and if $b_2 = 0$, $\tmop{Leb}$ is an invariant measure for $X$.
  Furthermore, $X$ is the unique solution to the martingale problem for the
  generator $(\mathcal{D}_{\max} (\mathcal{L}), \mathcal{L})$, where
  \begin{equation}
    \mathcal{D}_{\max} (\mathcal{L}) = \{ u \in H^1 : \mathcal{L} u \in L^2
    \}, \qquad \mathcal{L} u = \mathLaplace u + \nabla \cdummy (A \cdummy
    \nabla u) + b_2 \cdummy \nabla u, \qquad u \in H^1 . \label{maxdomaindef}
  \end{equation}
\end{theorem}

In the following lemma we give a more explicit sufficient condition for i.:

\begin{lemma}
  \label{uniquenesscompactset}Let $A \in L^p (M ; \mathbb{R}^{d \times d})$
  for $p \in [2, \infty]$ be antisymmetric $\tmop{Leb} - a.e.$ and suppose
  there exists a compact set $K$ such that for $\varepsilon > 0$ and
  $B^{\varepsilon} = \{ x : d (x, K) \leqslant \varepsilon \}$
  \begin{equation}
    \sup_{\varepsilon > 0} \varepsilon^{- 2} \tmop{Leb} (B^{\varepsilon}),
    \qquad \sup_{\varepsilon > 0} \varepsilon^{- 2} \int_{B^{\varepsilon}} | A
    |^2 < \infty, \label{divcond}
  \end{equation}
  and for all $\varepsilon > 0$,
  \begin{equation}
    A \1_{B_{\varepsilon}^c} \in L^{\infty} (M) . \label{condfinite}
  \end{equation}
  Then, there exists a sequence $g_n$ satisfying the assumptions from Theorem
  \ref{Wellposednesslocaldiverging}.
\end{lemma}

\begin{example}
  \label{Hausdorff}The assumptions (\ref{divcond}) and (\ref{condfinite}) are
  satisfied if there exist $s, \alpha, \delta > 0$ with
  \begin{equation}
    d - s \geqslant 2 (1 + \alpha), \label{conddimdiv}
  \end{equation}
  and a compact set $K$ such that with $B^{\varepsilon} = \{ x : d (x, K)
  \leqslant \varepsilon \}$
  \[ V (B^{\varepsilon}) \lesssim \varepsilon^{d - s}, \qquad A
     \1_{(B^{\varepsilon})^c} \in L^{\infty}, \qquad \left| A \1_{B^1} \right|
     \lesssim_{\tmop{Leb}} d (K, \cdummy)^{- \alpha} . \]
  Indeed, then
  \[ \varepsilon^{- 2} V (B^{\varepsilon}) \lesssim \varepsilon^{- 2 + d - s}
     \leqslant 1 \]
  and
  \[ \varepsilon^{- 2} \int_{B^{\varepsilon}} | A |^2 \lesssim \varepsilon^{-
     2 + d - s - 2 \alpha} \leqslant 1 . \]
  We can think of $K$ as a set of dimension at most $s$; the correct notion of
  (fractal) dimension is not important here.
  
  Let us consider the case where $K$ consists of a single point, i.e. $s = 0$
  and then $\alpha = (d - 2) / 2$ which satisfies (\ref{conddimdiv}). In terms
  of integrability, $A$ may behave like
  \[ | \mathbf{x} |^{- \alpha} \nin L^{\frac{d}{\alpha}}_{\tmop{loc}} =
     L^{\frac{2 d}{d - 2}}_{\tmop{loc}}, \qquad | \mathbf{x} |^{- \alpha} \in
     L^{\frac{2 d}{d - 2} -}_{\tmop{loc}} . \]
  For $d \rightarrow \infty$, the exponent $\frac{2 d}{d - 2}$ gets
  arbitrarily close to $p = 2$. Hence, for any $p > 2$ (in sufficiently high
  dimensions) we find a proper element $b \in W^{- 1, p -} = \bigcap_{\delta >
  0} W^{- 1, p - \delta}$, where $W^{- 1, \bar{p}} = \left( W^{1, \left( 1 -
  \frac{1}{\bar{p}} \right)^{- 1}} \right)^{\ast}$, such that existence and
  uniqueness in law holds for the drift $b$. This goes far beyond the regime
  treated in Theorem \ref{WellposednessBesov}, since at regularity $\gamma = -
  1$ we would need $p = \infty$.
  
  The second condition in \eqref{divcond} resembles a local form of $L^{2, 2}$
  Morrey regularity but ultimately the condition is different from those
  imposed for Morrey regularity. See the discussion in
  Example~\ref{counterexampleMorrey} in the appendix for a compactly supported
  counterexample that satisfies the assumptions in Lemma
  \ref{uniquenesscompactset} but is not in $L^{2, 2}$.
\end{example}

To prove uniqueness in Theorem \ref{Wellposednesslocaldiverging} we construct
a (a priori non-unique) domain $\mathcal{D}$ for the generator $\mathcal{L} u
= \mathLaplace u + \nabla \cdummy (A \cdummy \nabla u) + b_2 \cdummy \nabla u$
by abstract arguments. Then, we show that smooth functions are in a certain
sense a core for this operator (see Proposition \ref{uniquenessabstract}),
which allows us to use the energy estimate of energy solutions to show that
any energy solution is a solution to the martingale problem for $(\mathcal{D},
\mathcal{L})$. Uniqueness then follows from uniqueness of the martingale
problem for a fixed domain $\mathcal{D}$. An alternative approach is to
directly show that all candidate domains $\mathcal{D}$ are equal to the
maximal domain (\ref{maxdomaindef}). Eventually both approaches agree and
reduce to showing for all sufficiently large $\lambda > 0$ the simple
statements
\begin{equation}
  \begin{array}{lll}
    (\lambda - \mathcal{L}^{\ast}) u = 0 \qquad & \Leftrightarrow & \qquad u =
    0,\\
    (\lambda - \mathcal{L}) u = 0 \qquad & \Leftrightarrow & \qquad u = 0,
  \end{array} \label{reducedstatement}
\end{equation}
for all $u \in H^1$, see Proposition \ref{uniquenessabstract} and Lemma
\ref{lemmainjective}. We use the same strategy for proving uniqueness in the
infinite-dimensional setting in {\cite{GraefnerPerkowski23plus}} (see also
{\cite{GraefnerPerkowski23}}) to study solutions to critical singular SPDEs
with Gaussian reference measure. To prove (\ref{reducedstatement}) we would
like to test the equations against $u$ and use the antisymmetry of $\nabla
\cdummy (A \cdummy \nabla u)$ to show that $\langle \nabla \cdummy (A \cdummy
\nabla u), u \rangle = 0$. However, for $p \in (2, \infty)$ the integrability
of $A \in L^p$ \ and $\nabla u \in L^2$ is not sufficient to directly justify
this step, and we approximate the integral by multiplying with a sequence $g_n
\rightarrow 1$ that removes the singularities of $A$ in a way which allows us
to pass to the limit (see Figure \ref{picture}). This is the content of Lemmas
\ref{uniquenessabstract}, \ref{lemmainjective} and \ref{uniquenesscompactset}.

\
\begin{figure}
\centering
\includegraphics[width=6.65174471992654cm,height=7.24550701823429cm]{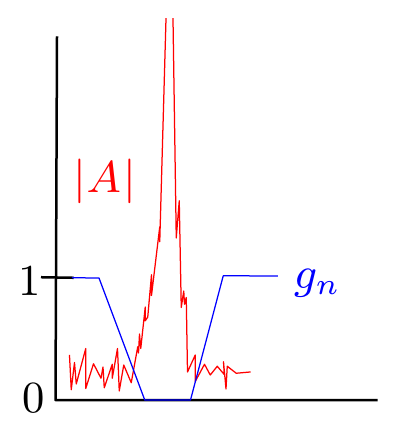}\caption{\label{picture}Illustration
of the sequence $g_n \rightarrow 1$ from Theorem
\ref{Wellposednesslocaldiverging} which removes the singular set of $A$.}
\end{figure}

\begin{remark}
  Since Lemma \ref{uniquenessabstract} reduces the problem of proving
  uniqueness to the simple sufficient condition (\ref{reducedstatement}), we
  expect that it is possible to find more general assumptions on $b$ than the
  one from Theorem \ref{Wellposednesslocaldiverging} which still yield
  uniqueness. Another interesting direction is to consider supercritical
  singular SPDEs with Gaussian invariant measure with a similar structure to
  the finite-dimensional equation (\ref{MainSDE}). An example would be the
  {\tmem{fractional stochastic Burgers equation}} $u : \mathbb{R}_+ \times
  \mathbb{T} \rightarrow \mathbb{R}$,
  \[ \partial_t u = - (- \Delta)^{\theta} u + \partial_x u^2 + \sqrt{2} (-
     \Delta)^{\theta / 2} \xi, \]
  with $\theta > 1 / 2$ and $\xi$ being space-time white noise, which has the
  law $\mu$ of spatial white noise as an invariant measure. Existence of
  energy solutions for this equation is shown in {\cite{GubinelliJara13}}.
  Uniqueness in the subcritical regime $\theta > 3 / 4$ is shown in
  {\cite{GubinelliPerkowski20}}, and a proof of uniqueness in the critical
  case $\theta = 3 / 4$ is contained in the upcoming work
  {\cite{GraefnerPerkowski23plus}}, based on a formulation of
  (\ref{reducedstatement}) in the infinite-dimensional setting. In future work
  we will consider an infinite-dimensional operator-valued equivalent of the
  matrix field $A$ corresponding to the ``divergence-free'' (w.r.t. to $\mu$)
  drift $b (u) = \partial_x u^2$ and to explore whether a similar condition as
  in (\ref{divcond}) can be used to show (\ref{reducedstatement}) in the
  supercritical regime $\theta \in (1 / 2, 3 / 4)$.
\end{remark}

\begin{example}[Brownian particle in the curl of the log-regularized Gaussian
free field]
  \hphantom{Another} Another motivation for our work are the recent results on
  $\sqrt{\log t}$ superdiffusivity of a Brownian particle in the curl of the
  Gaussian free field. Let $\xi$ be a massless two-dimensional Gaussian free field, independent of
  $B$, and consider the (smooth, random) SDE on $\mathbb{R}^2$
  \[ \mathd Y_t = \nabla^{\perp} (\xi \ast \rho) (Y_t) \mathd t + \mathd B_t,
  \]
  where $\nabla^{\perp} = (\partial_2, - \partial_1)$. This model was studied
  in {\cite{Toth2012}} and more recently in
  {\cite{CannizzaroHaunschmidSibitzToninelli21}} and
  {\cite{ChatzigeorgiouMorfeOttoWang23}}. The authors show that the variance
  grows superdiffusively, essentially $\mathbb{E} [X_t^2] \simeq t \sqrt{\log
  t}$, as $t \rightarrow \infty$. By the scaling invariance of the Gaussian
  free field and of $B$, this translates to an equivalent statement at fixed
  scales as we remove the ultraviolet cutoff: The SDE
  \[ \mathd Y_t^n = \nabla^{\perp} (\xi (\omega) \ast \rho_n) (Y^n_t) \mathd t
     + \mathd B_t  \]
  satisfies $\mathbb{E} [| Y_t^n |^2] \simeq t \sqrt{\log n}$, suggesting that
  there are no subsequential limits for $n \rightarrow \infty$. Note that
  locally a.s. $\nabla^{\perp} \xi \in B^{- 1 - \varepsilon}_{\infty, \infty}$
  for any $\varepsilon > 0$ and that $\nabla \cdummy \nabla^{\perp} \xi
  (\omega) \equiv 0$.
  
  In the recent work~{\cite{Yang2024}}, inspired
  by~{\cite{CannizzaroErhardToninelli21,CannizzaroGubinelliToninelli23}}, the authors show that in the weak disorder limit
  \[ \mathd X^n_t = \frac{1}{\sqrt{\log \frac{1}{n}}} \nabla^{\perp} (\xi \ast
     \rho_n) (X^n_t) \mathd t + \sqrt{2} \mathd B_t, \]
  $X^n$ converges to a two-dimensional Brownian motion $W$ with increased
  variance, $\mathbb{E} [| W_t^i |^2] > 2 t$ for $i = 1, 2$. On the other hand,~\cite{Armstrong2024} derive a quenched central limit theorem for $|\log \varepsilon^2|^{-1/4}\varepsilon X_{t/\varepsilon^2}$.

  Here we take a
  different logarithmic regularization and we show below that slightly
  regularizing $\nabla^{\perp} \xi$ by considering $\nabla^{\perp} \log ((1 -
  \mathLaplace)^{- \alpha}) \xi$ with $\alpha > 1$ (this can be improved to
  $\alpha > \frac{1}{2}$) leads to a well-posed equation. It would be
  interesting to understand the limit (if existent) of
  \[ \mathd X^n_t = \nabla^{\perp} \log ((1 - \mathLaplace)^{- 1 / 2}) (\xi
     \ast \rho_n) (X^n_t) \mathd t + \sqrt{2} \mathd B_t . \]
  To see the claimed well-posedness for $\alpha > 1$, we consider for
  simplicity a {\tmem{periodic}} Gaussian free field $\xi$ on $\mathbb{T}^2$
  with $\mathcal{F} \xi (0) = 0$ (without loss of generality) and the quenched
  SDE
  \begin{equation}
    \mathd X_t = \nabla^{\perp} \log ((1 - \mathLaplace)^{- \alpha}) \xi
    (\omega) (X_t) \mathd t + \mathd B_t, \label{logregularizedGFfield}
  \end{equation}
  where $\alpha > 0$, and $\log ((1 - \mathLaplace)^{- \alpha})$ is a Fourier
  multiplier. We have to compute the regularity of realizations of $b =
  \nabla^{\perp} \log (1 - \mathLaplace)^{- \alpha} \xi$. Note that $A_{1 1} =
  A_{2 2} = 0$ and
  \[ A_{12} = - A_{2 1} = - \log (1 - \mathLaplace)^{- \alpha} \xi . \]
  Hence, using the embedding of $B^0_{\infty, 1}$ into the continuous bounded
  functions, it suffices to show that $\log (1 - \mathLaplace)^{- \alpha} \xi
  \in B^0_{\infty, 1}$ for all $\alpha > 1$. This is shown in Lemma~\ref{lem:GFF-reg} in the appendix. As a consequence of Theorem
  \ref{Wellposednesslocaldiverging} we thus obtain the following {\tmem{quenched}}
  result.
\end{example}

\begin{corollary}
  In the context of the previous example let $\alpha > 1$ and let $\mu \ll
  \tmop{Leb}_{\mathbb{T}^d}$ be a probability measure on $\mathbb{T}^d$. Then
  there exists a unique energy solution $X$ to equation
  (\ref{logregularizedGFfield}) with initial law $\mu$. The process $X$ is the
  weak limit of the (respective) sequence of solutions to (\ref{approx}), it
  is a Markov process and $\tmop{Leb}_{\mathbb{T}^d}$ is an invariant measure
  for $X$. Furthermore $X$ is the unique solution to the martingale problem
  for the generator $(\mathcal{D}_{\max} (\mathcal{L}), \mathcal{L})$, where
  $\mathcal{D}_{\max} (\mathcal{L}) = \{ u \in H^1 : \mathcal{L} u \in L^2 \}$
  and
  \begin{equation}
    \mathcal{L} u = \frac{1}{2} \mathLaplace u + \nabla \cdummy (- \log (1 -
    \mathLaplace)^{- \alpha} \xi (\omega) \nabla^{\perp} u), \qquad u \in H^1
    .
  \end{equation}
\end{corollary}

In {\cite{GraefnerPerkowski23plus}} a proof of {\tmem{annealed}}
well-posedness of equation (\ref{logregularizedGFfield}) is given for $\alpha
> 1 / 2$. In fact it is even shown that almost surely $(1 - \mathcal{L})$ is
injective on $H^1$ which implies {\tmem{quenched}} well-posedness for all
$\alpha > 1 / 2$ by Proposition \ref{uniquenessabstract}.

\begin{example}[$N$-particle system]
  Consider a system of $N$ $d$-dimensional particles $(X^k)_{k=1,\dots,N}$ which are driven by
  independent Brownian motions $(B^k)_{k=1,\dots,N}$ and which interact through some
  field $b : \mathbb{R}_+ \rightarrow \mathcal{S}' (M ;
  \mathbb{R}^d)$. This corresponds to the SDE
  \begin{equation}
    \mathd X^k_t = \sum_{k \neq l} b (t, X_t^k - X_t^l) \mathd t + \mathd B^k_t .
    \label{Nparticle}
  \end{equation}
  In {\cite{HaoZhang23}} this example is treated for divergence-free $b \in H^{-1}_\infty(\mathbb{R}^d)$. If we define $b_1$ in terms of some
  antisymmetric $A$, we can obtain $A^N: \mathbb{R}^{dN} \to \mathbb{R}^{dN \times dN}$ with the Helmholtz decomposition  as $A_{(k - 1) d
  + i, (l - 1) d + j}^N(x) =\mathLaplace_{M}^{-1}\nabla(\nabla\cdot A_j)^i(x^l - x^k)-\mathLaplace_{M}^{-1}\nabla(\nabla\cdot A_i)^j(x^k - x^l)$ where $k, l \in \{ 1, \ldots, N
  \}$ and $ i,j \in \{ 1, \ldots, d \}$.
  
  Consider the case $M =\mathbb{T}^d$ and assume that $b$ satisfies the conditions of Theorem
  \ref{WellposednessBesov}, where we replace \ref{supercriticalbesovunique} by the condition $b_1 \in L^p_T B^{- \gamma}_{p, 1}$ for
    $p \geqslant \frac{2}{1 - \gamma}$ and some $p \in [2, \infty], \gamma \in [0,
    1]$ from Remark \ref{remarkBesovtheorem}. Then the drift
  $b^N$ corresponding to (\ref{Nparticle}) satisfies the conditions from Theorem 
  \ref{WellposednessBesov} as well and thus energy solutions are well-posed. This follows from the fact that via a coordinate transformation it suffices to analyze the regularities of $b
  \circ \pi_l$ and  $\mathLaplace_{\mathbb{T}^d}^{-1}\nabla(\nabla\cdot A_j) \circ \pi_l$, $j=1,...,d$, where $\pi_l : \mathbb{T}^{N d}
  \rightarrow \mathbb{T}^d$ with $x \mapsto x^l$ is the projection on the $l$-th
  particle. Then we can use that $\| u \circ \pi_l \|_{L^r (\mathbb{T}^{d N})} \lesssim \|
  u \|_{L^r (\mathbb{T}^d)}$ and for any (polynomially growing) function $D :
  \mathbb{Z}^{N d} \rightarrow \mathbb{C}$ and the corresponding Fourier
  multiplier $\sigma (D)$ it holds $\sigma (D) u = [\sigma (D |_l \nobracket)
  u] \circ \pi_l$, where $D |_l \nobracket : \mathbb{Z}^d \rightarrow
  \mathbb{C}$ is defined by $D |_l \nobracket (k) = D (k_l)$. Moreover, $\mathLaplace_{\mathbb{T}^d}^{-1}\nabla(\nabla\cdot A_j)$ is a bounded operators on $L^r$ for $r\in(1,\infty)$ (see the proof of Theorem \ref{tightnesstorusBesov}). Thus all
  regularities of $b$ from Theorem \ref{WellposednessBesov} translate directly
  to the drift $b^N : \mathbb{R}_+ \rightarrow \mathcal{S}' (\mathbb{T}^{N d}
  ; \mathbb{R}^{N d})$.
  
  If $A$ is an even function satisfying the assumptions from Lemma~\ref{uniquenesscompactset}, we can take\\ $A_{(k - 1) d
  + i, (l - 1) d + j}^N(x) =A_{i,j}(x_l-x_k)$ and $A^N$ satisfies the same conditions with $K^N = \{ x \in \mathbb{T}^{d N} : \exists k \neq l : x^k - x^l \in K
  \},$ so that energy solutions are well-posed due to Theorem \ref{Wellposednesslocaldiverging}.
  
  For $M =\mathbb{R}^d$ the only problem arises from the fact that $u
  \circ \pi_l$ does not inherit global integrability from $u$, which might be addressed through a localization argument similarly to the one
  from \ref{localizationforinfty} in Remark \ref{remarkBesovtheorem}.
\end{example}

\paragraph*{Relation with {\cite{HaoZhang23}}.}

In {\cite{HaoZhang23}} the authors show uniqueness of the limit
of the solutions to (\ref{approx}) if $\nabla \cdummy b \equiv 0$ and for example (taking $L^\infty$ time regularity for simplicity) $b \in
(L^{\infty}_T B^{- 1}_{\infty, 2} + L^2_T L^2) \cap L^{\infty}_T H^{- 1, p}$ with $p > d$, and $\frac{\mathd\mu}{\mathd \operatorname{Leb}_M} \in L^2$. This is outside of the scope of Theorem \ref{WellposednessBesov} (and the
succeeding arguments) since we require the Besov parameter $q = 1$, that is $(b \in
L^{\infty}_T B^{- 1}_{\infty, 1} + L^2_T L^2) \cap L^\infty_T B^{-1}_{p,2}$ with $p>2$. Conversely, our assumptions are dimension-independent and we do not require $b \in L^\infty_T H^{-1,p}$ with $p>d$, which e.g. in $d \ge 4$ rules out generic elements of $L^\infty_T L^2 \subset L^2_T L^2 \cap L^\infty_T B^{-1}_{2+\frac{4}{d-2},2}$, and also we do not need $L^2$ integrability of the initial density. Furthermore, \cite{HaoZhang23} do not consider perturbations of the divergence-free case.

Conceptually, our work is related to~\cite{HaoZhang23} as both rely on PDE methods for proving uniqueness. However, while~\cite{HaoZhang23} derive stronger results on the PDE under the assumption $L^\infty_T B^{-1}_{p,2}$ with $p>d$ or similar, our approach only relies on  generic a priori estimates for the PDE. We obtain  uniqueness from properties of the solutions to the SDE such as incompressibility and energy estimates. Moreover, with the latter we identify a selection
principle for weak solutions which guarantees uniqueness even if they are not
given as limits of approximations.

\paragraph{Plan of the paper.}We write $M$ instead of $\mathbb{T}^d$ or
$\mathbb{R}^d$ if a statement (or proof) is valid for both cases. In Section
\ref{existenceperiodic} we show existence of energy solutions in
$\mathbb{T}^d$, and in Section \ref{Euclideanset} the existence of energy
solutions in $\mathbb{R}^d$. In Section \ref{Besovcase} we prove uniqueness of
energy solutions in $M$ under the integrability and regularity assumptions on
$b$ from Theorem \ref{WellposednessBesov} and in Section
\ref{assumptionsonAcase} we prove uniqueness of energy solutions in $M$ under
the structural assumptions on $A$ in Theorem
\ref{Wellposednesslocaldiverging}.

\paragraph{Notation.}We denote by $\mathbb{T}^d =\mathbb{R}^d /\mathbb{Z}^d$
the $d$-dimensional torus and identify functions on $\mathbb{T}^d$ with
periodic functions on $\mathbb{R}^d$. As usual we write $\mathcal{F} u (k) =
\hat{u} (k) = \langle u, e^{- 2 \pi i k \cdummy} \rangle$ for the Fourier
transform at $k \in \mathbb{Z}^d$ of a distribution $u \in \mathcal{S}'
(\mathbb{T}^d)$. For $\alpha \in \mathbb{R}$, we define the action of the
fractional Laplacian $\mathcal{F} (- \mathLaplace)^{\alpha} u (k) = 1_{| k | >
0} (2 \pi | k |)^{2 \alpha} \hat{u} (k)$. For numbers $a, b$ we write $a
\lesssim b$ if $a \leqslant C b$ for some constant $C > 0$ which is irrelevant
for the discussion. To emphasize the dependence of $C$ on some parameter $T$,
we write $a \lesssim_T b$. For equality up to such a constant we write $a
\simeq b$. When discussing stochastic processes, we will always implicitly
assume that all objects are defined on some complete probability space
$(\Omega, \mathcal{F}, \mathbb{P})$ which is rich enough to allow the relevant
objects, like e.g. Brownian motion.

\section{Construction of energy solutions}

In our finite-dimensional setting we obtain existence of energy solutions on
$\mathbb{T}^d$ in Theorem \ref{tightnesstorusBesov} by the usual
forward-backward martingale argument under the invariant measure that allows
to cancel the antisymmetric part of the generator corresponding to the
divergence-free part of the drift. This yields the central It\^{o} trick bound
(\ref{Itotrick}) for additive functionals. We generalize this argument by
allowing non-divergence-free perturbations (which are not antisymmetric on the
generator level) of the drift via an application of Girsanov's theorem. On
$\mathbb{R}^d$ we obtain existence in Theorem \ref{tightnesstorusBesovEuclid}
by an approximation with tori $\mathbb{T}^{d, L}$ of length $L \rightarrow
\infty$.

\subsection{Existence of energy solutions in
$\mathbb{T}^d$}\label{existenceperiodic}

We work with nonhomogeneous Besov spaces
\[ B^s_{p, q} (M) = \left\{ u \in \mathcal{S}' (M) : \sum_j^{\infty} 2^{s j q}
   \| \mathLaplace_j u \|_{L^p}^q < \infty \right\}, \]
where $p, q \in [1, \infty], s \in \mathbb{R}$ and $\mathLaplace_j u = \phi_j
\ast u, \phi_j = \mathcal{F}^{- 1} \varphi_j$ is the $j$th Littlewood-Paley
block associated to $u$. The corresponding norms are denoted by $\| \cdummy
\|_{B^s_{p, q}}$. Standard references on Besov spaces are
{\cite{BahouriCheminDanchin11}} on $\mathbb{R}^d$ and {\cite{Triebel83}} on
$\mathbb{T}^d$. However, we will often refer to the nice lecture notes
{\cite{vanZuijlen22}} which contain all necessary results on $\mathbb{R}^d$
and which carry over to the periodic setting in a canonical way. We will
repeatedly use the fact that $\rho^n \ast b \rightarrow b$ in $L^q_T L^p$, or
in $L^q_T B^s_{p, r}$, respectively, if $b \in L^q_T L^p$, or $b \in L^q_T
B^s_{p, r}$, respectively, where $p, q, r \in [1, \infty)$. For later use we
define for $p \in [1, \infty]$
\begin{equation}
  B^0_{p, 1, 2} \assign \left\{ u \in B^0_{p, 1} : \| u \|_{B^0_{p, 1, 2}}^2
  \assign \sum_j \| \Pi_{\geqslant j} u \|_{B^0_{p, 1}}^2 < \infty \right\},
  \label{sqrtspace}
\end{equation}
where $\Pi_{\geqslant j} u \assign \sum_{i \geqslant j} \mathLaplace_j u$. We
note that $B^0_{p, 1, 2} \longhookrightarrow B^0_{p, 1} \longhookrightarrow
L^p$ (see e.g. {\cite{vanZuijlen22}} for the second embedding). See Appendix
$\ref{appendixA}$ for some auxiliary results concerning $B^0_{p, 1, 2}$. For
$b \in L^{\infty} (\mathbb{R}_+ \times \mathbb{T}^d)$, we say that a process
$X$ on $\mathbb{T}^d$ solves
\[ \mathd X_t = b (t, X_t) \mathd t + \sqrt{2} \mathd B_t \]
if $X = Y \tmop{mod} \mathbb{Z}^d$, where $Y$ is an $\mathbb{R}^d$-valued
process solving
\[ \mathd Y_t = b (t, Y_t) \mathd t + \sqrt{2} \mathd B_t, \]
with a Brownian motion $B$ on $\mathbb{R}^d$.

The following It{\^o} trick estimate is one of the standard tools for energy
solutions.

\begin{lemma}[It{\^o} trick]
  \label{Itotricklemma}For $x_0 \in \mathbb{T}^d$, let $X : \Omega \times
  \mathbb{R}_+ \rightarrow \mathbb{T}^d$ solve
  \[ \mathd X_t = b (t, X_t) \mathd t + \sqrt{2} \mathd B_t, \qquad X_0 = x_0
  \]
  where $B$ is a Brownian motion and $b \in C (\mathbb{R}_+ ; C^{\infty}
  (\mathbb{T}^d))$ is such that $\nabla \cdummy b \equiv 0$. Then, $X$ is an
  ergodic $\mathbb{T}^d$-valued Markov process with unique invariant measure
  $\tmop{Leb}_{\mathbb{T}^d}$. For $X_0 \sim \tmop{Leb}_{\mathbb{T}^d}, f \in
  C^{\infty} ([0, T] \times \mathbb{T}^d, \mathbb{R})$ and for all $p \in [2,
  \infty), q \in [2, \infty]$ we have
  \begin{equation}
    \mathbb{E} \left[ \sup_{t \leqslant T} \left| \int_0^t \mathLaplace f (s,
    X_s) \mathd s \right|^p \right] \lesssim T^{p \left( \frac{1}{2} -
    \frac{1}{q} \right)} \| \nabla f \|^p_{L^q_T L^p (\mathbb{T}^d)},
    \label{Itotrick}
  \end{equation}
  where the implicit constant is independent of $b$.
\end{lemma}

\begin{proof}
  Markov property, ergodicity and uniqueness of the invariant measure are well
  known. For smooth $f \in C^{\infty} ([0, T] \times \mathbb{T}^d,
  \mathbb{R})$ we have that
  \[ f (T, X_T) = f (0, X_0) + \int_0^T  (\partial_s f + b \cdummy
     \nabla f + \mathLaplace f) (s, X_s) \mathd s + \int_0^T \sqrt{2} \nabla f
     (s, X_s) \cdummy \mathd B_s . \]
  Since $X$ is a Markov process with invariant measure
  $\tmop{Leb}_{\mathbb{T}^d}$, for $X_0 \sim \tmop{Leb}_{\mathbb{T}^d}$ the
  time-reversed process $\hat{X}_t = X_{T - t}$ has the generator
  $\mathcal{L}_t^{\ast}$, the adjoint of $\mathcal{L}_t = \mathLaplace + b (t)
  \cdummy \nabla$ (see e.g. {\cite{Kolmogorov37}},
  {\cite{HaussmannPardoux86}}, {\cite{GubinelliJara13}},
  {\cite{GoncalvesJara14}}). Therefore,
  \[ f (0, \hat{X}_T) = f (T, \hat{X}_0) + \int_0^T (- \partial_s f
     - b \cdummy \nabla f + \mathLaplace f) (T - s, \hat{X}_s) \mathd s +
     \int_0^T \sqrt{2} \nabla f (T - s, \hat{X}_s) \cdummy \mathd \hat{B}_s,
  \]
  where $\hat{B}$ is a Brownian motion in the filtration generated by
  $\hat{X}$. Adding both equations gives
  \[ - 2 \int_0^T \mathLaplace f (s, X_s) \mathd s = \int_0^T \sqrt{2} \nabla
     f (s, X_s) \cdummy \mathd B_s + \int_0^T \sqrt{2} \nabla f (T - s,
     \hat{X}_s) \cdummy \mathd \hat{B}_s, \]
  and thus Burkholder-Davis-Gundy-inequality, Minkowski's inequality and
  stationarity yield
  \begin{eqnarray}
    \mathbb{E} \left[ \sup_{t \leqslant T} \left| \int_0^t \mathLaplace f (s,
    X_s) \mathd s \right|^p \right] & \lesssim & \mathbb{E} \left[ \left(
    \int_0^T | \nabla f (s, X_s) |^2 \mathd s \right)^{p / 2} \right]
    \nonumber\\
    & &+\mathbb{E} \left[ \left( \int_0^T | \nabla f (T - s, X_{T - s}) |^2
    \mathd s \right)^{p / 2} \right] \nonumber\\
    & \lesssim &\left( \int_0^T \mathbb{E} [| \nabla f (s, X_s) |^p]^{2 / p}
    \mathd s \right)^{p / 2}  \label{herethemassappears}\\
    & \leqslant& T^{\left( 1 - \frac{2}{q} \right) \frac{p}{2}} \left(
    \int_0^T \| \nabla f (s, \cdummy) \|^q_{L^p} \mathd s
    \right)^{\frac{1}{q}} \nonumber\\
    & =& T^{\left( \frac{p}{2} - \frac{p}{q} \right)} \| \nabla f \|_{L^q_T
    L^p}^p . \nonumber
  \end{eqnarray}
  
\end{proof}

We handle the non-divergence-free part of $b$ as a perturbation with
Girsanov's theorem, for which we need the following result.

\begin{lemma}[Novikov type bound]
  \label{stochexponential}Let $b$ and $X$ be as in the previous lemma. For all
  $p \in [1, \infty)$ there exists $C > 0$, independent of $b$, such that for
  all $a \in L^4_T B^0_{2 r, 1, 2}$, $r \in [d, \infty]$, the stochastic
  exponential
  \[ \mathcal{E} \left( \int_0^{\cdummy} a (s, X_s) \cdummy \mathd B_s
     \right)_t = \exp \left( \int_0^t a (s, X_s) \cdummy \mathd B_s -
     \frac{1}{2} \int_0^t | a |^2 (s, X_s) \mathd s \right) \]
  satisfies
  \[ \mathbb{E} \left[ \mathcal{E} \left( \int_0^{\cdummy} a (s, X_s) \cdummy
     \mathd B_s \right)_t^p \right] \lesssim \exp (C \| a \|^4_{L^4_T B^0_{2
     r, 1, 2}}) . \]
\end{lemma}

\begin{proof}
  Since $\mathbb{E} [\mathcal{E} (\ldots)_T] \leqslant 1$, we obtain from the
  Cauchy-Schwarz inequality
  \begin{eqnarray*}
    \mathbb{E} \left[ \left| \mathcal{E} \left( \int_0^{\cdummy} a (s, X_s)
    \cdummy \mathd B_s \right)_T \right|^p \right] & =&\mathbb{E} \left[ \exp
    \left( p \int_0^T a (s, X_s) \cdummy \mathd B_s - \frac{p}{2} \int_0^T | a
    (s, X_s) |^2 \mathd s \right) \right]\\
    & =&\mathbb{E} \left[ \exp \left( 2 p \int_0^T a (s, X_s) \cdummy \mathd
    B_s - \frac{(2 p)^2}{2} \int_0^T | a (s, X_s) |^2 \mathd s \right)^{1 / 2}
    \right.\\
    & &\qquad \cdot \left. \exp \left( \left( p^2 - \frac{p}{2} \right)
    \int_0^T | a (s, X_s) |^2 \mathd s \right) \right]\\
    & \leqslant & \mathbb{E} \left[ \exp \left( c \int_0^T | a (s, X_s) |^2
    \mathd s \right) \right]^{1 / 2},
  \end{eqnarray*}
  where $c = (2 p^2 - p) > 0$ and in what follows $c > 0$ may change in every
  step. Now we bound
  \[ \overline{| a |^2} (s) = \int_{\mathbb{T}^d} | a |^2 (s, x) \mathd x =
     \|a (s) \|_{L^2}^2 \lesssim \| a (s) \|_{B^0_{2 r, 1, 2}}^2, \]
  and using $e^{| x |} \lesssim e^{x^2}$ together with the forward-backward
  representation from the proof of Lemma \ref{Itotricklemma} we obtain
  \begin{eqnarray}
    \mathbb{E} \left[ \exp \left( c \int_0^t | a |^2 (s, X_s) \mathd s
    \right) \right]
    & =& \exp \left( c \int_0^t \overline{| a |^2} (s) \mathd s \right)\nonumber
   \\ & &\cdot\mathbb{E}\left[ \exp \left( c \int_0^t \mathLaplace \mathLaplace^{- 1}
    \left( | a |^2 - \overline{| a |^2} \right) (s, X_s) \mathd s \right)
    \right]  \label{Laplaceinverseexp}\\
    & \lesssim & \exp (c \| a  \|^4_{L^2_T B^0_{2 r, 1, 2}}) \mathbb{E} \left[
    \exp \left( - \frac{c}{\sqrt{2}} \int_0^t \mathLaplace^{- 1} \nabla | a
    |^2 (s, X_s) \mathd B_s \right) \right. \nonumber\\
    & & \cdot \exp \left( - \frac{c}{\sqrt{2}} \int_0^t
    \mathLaplace^{- 1} \nabla | a |^2 (t - s, \hat{X}_s) \mathd \hat{B}_s 
    \bigg)\right], \nonumber
  \end{eqnarray}
  using $\nabla \mathLaplace^{- 1} \left( | a |^2 - \overline{| a |^2} \right)
  = \Delta^{- 1} \nabla | a |^2$ in the last step. Thus, applying the
  Cauchy-Schwarz inequality (and changing $c$) it remains to estimate
  \begin{eqnarray*}
    \mathbb{E} \left[ \exp \left( c \int_0^t \mathLaplace^{- 1} \nabla | a |^2
    (s, X_s) \mathd B_s \right) \right]^2 & \lesssim & \mathbb{E} \left[ \exp
    \left( c \int_0^t (\mathLaplace^{- 1} \nabla | a |^2)^2 (s, X_s) \mathd s
    \right) \right]\\
    & \leqslant & \exp (c \| \mathLaplace^{- 1} \nabla | a |^2 \|^2_{L^2
    L^{\infty}}),
  \end{eqnarray*}
  where we used the exponential martingale inequality. We get, using the Besov
  embedding theorem with $0 \leqslant 1 - d \left( \frac{1}{r} -
  \frac{1}{\infty} \right)$,
  \begin{equation}
    \| \mathLaplace^{- 1} \nabla | a |^2 \|_{L^2 L^{\infty}} \lesssim \|
    \mathLaplace^{- 1} \nabla | a |^2 \|_{L^2 B^0_{\infty, 1}} \lesssim \|
    \mathLaplace^{- 1} \nabla | a |^2 \|_{L^2 B^1_{r, 1}} \lesssim \| | a |^2
    \|_{L^2 B^0_{r, 1}}, \label{criticalinfestimate}
  \end{equation}
  which yields the claim because $\| | a |^2 \|_{L^2_T B^0_{r, 1}} \lesssim \|
  a \|^2_{L^4_T B^0_{2 r, 1, 2}}$ by Lemma \ref{estimateforsquarenorm}.
\end{proof}

Exponential integrability was also obtained from the forward-backward
decomposition in the infinite-dimensional case of {\cite{GubinelliJara13}}
(without an application to change of drift), which was pointed out to the
authors by M. Gubinelli and which significantly simplifies the proof in the
periodic case compared to the longer argument in Lemma \ref{NoviovboundEuclid}
below.\tmcolor{magenta}{ }

The It{\^o} trick together with the previous Novikov bound gives us uniform a
priori bounds for solutions to $\mathd X_t^n = b_n (t, X_t^n) \mathd t +
\sqrt{2} \mathd W_t$. Any limit point will be an {\tmem{energy solution}} to
$\mathd X_t = b (t, X_t) \mathd t + \sqrt{2} \mathd W_t$, analogously to
{\cite{GubinelliPerkowski20}}, {\cite{GraefnerPerkowski23}}.

\begin{definition}[Energy solution]
  \label{defenergysolutions}Let $b : \mathbb{R}_+ \rightarrow \mathcal{S}'
  (\mathbb{T}^d, \mathbb{R}^d)$ be such that for all $T > 0$ and $f \in
  C^{\infty} ([0, T] \times \mathbb{T}^d)$ we have $b \cdummy \nabla f \in
  L^2_T H^{- 1} (\mathbb{T}^d)$. Let $X$ be a stochastic process with values
  in $C (\mathbb{R}_+ ; \mathbb{T}^d)$ such that for all $T > 0$:
  \begin{enumerateroman}
    \item \label{incompressible}$X$ is {\tmem{incompressible in probability}},
    i.e. for all $\varepsilon > 0$ and $T > 0$ there exists $M = M
    (\varepsilon, T) > 0$ such that for all $A \in \mathcal{B} (\mathbb{R}^d)$
    with the Lebesgue measure $\tmop{Leb} (A)$
    \[ \mathbb{P} (X_t \in A) \leqslant \varepsilon + M \cdot \tmop{Leb} (A),
       \qquad t \in [0, T], \]
    as well as
    \[ \mathbb{P} \left( \sup_{t \leqslant T} \left| \int_0^t f (s, X_s)
       \mathd s \right| > \delta \right) \leqslant \varepsilon +
       \frac{M}{\delta} \| f \|_{L^1_T L^1 (\mathbb{T}^d)}, \]
    for all $f \in C^{\infty} ([0, T] \times \mathbb{T}^d)$.
    
    \item \label{admissible}X is {\tmem{admissible in probability}}/satisfies
    an {\tmem{energy estimate in probability}}, i.e. for all $\varepsilon > 0$
    and $T > 0$ there exists $M = M (\varepsilon, T) > 0$ such that for all $f
    \in C^{\infty} ([0, T] \times \mathbb{T}^d)$
    \begin{equation}
      \mathbb{P} \left( \sup_{t \leqslant T} \left| \int_0^t f (s, X_s) \mathd
      s \right| > \delta \right) \leqslant \varepsilon + \frac{M}{\delta} \| f
      \|_{L^2_T H^{- 1} (\mathbb{T}^d)} . \label{Itotrickbound}
    \end{equation}
    \item \label{martingaleproperty}For any $f \in C^{\infty} ([0, T] \times
    \mathbb{T}^d)$, the process
    \[ M^f_t = f (t, X_t) - f (0, X_0) - \int_0^t (\partial_s + \mathLaplace +
       b \cdummy \nabla) f (s, X_s) \mathd s, \qquad t \in [0, T], \]
    is a (continuous) local martingale in the filtration generated by $X$,
    where the integral is defined as
    \[ \int_0^t (\partial_s + \mathLaplace + b \cdummy \nabla) f (s, X_s)
       \mathd s \assign I ((\partial_s + \mathLaplace + b \cdummy \nabla)
       f)_t, \]
    with $I$ the unique continuous extension from $C^{\infty} ([0, T] \times
    \mathbb{T}^d)$ to $L^2_T H^{- 1} (\mathbb{T}^d)$ of the map $g \mapsto
    \int_0^{\cdummy} g (s, X_s) \mathd s$ taking values in the continuous
    adapted stochastic processes equipped with the topology of uniform
    convergence on compacts (ucp-topology), and the extension $I$ exists by
    (\ref{Itotrickbound}).
    
    \item \label{quadraticvariation}The local martingale $M^f$ from
    \ref{martingaleproperty} has quadratic variation
    \[ \langle M^f \rangle_t = \int_0^t | \nabla f (s, X_s) |^2 \mathd s . \]
  \end{enumerateroman}
  Then $X$ is called an {\tmem{energy solution}} to the SDE $\mathd X_t = b
  (t, X_t) \mathd t + \sqrt{2} \mathd W_t$.
\end{definition}

Note that \ref{quadraticvariation} implies that the local martingale from
\ref{martingaleproperty} is a proper martingale.

\begin{remark}
  These conditions are tailored to handle initial distributions with densities
  in $L^1 (\mathbb{T}^d)$. If we instead assume densities in $L^2
  (\mathbb{T}^d)$, then we can replace the incompressibility condition by the
  simpler $L^1$-bound
  \[ \mathbb{E} [| f (X_t) |] \lesssim_T \| f \|_{L^2 (\mathbb{T}^d)}, \qquad
     t \in [0, T], \]
  and the energy estimate by
  \[ \mathbb{E} \left[ \sup_{t \leqslant T} \left| \int_0^t f (s, X_s) \mathd
     s \right| \right] \lesssim_T \| f \|_{L^2_T H^{- 1} (\mathbb{T}^d)}, \]
  and in that case the property iv would not be necessary to obtain the
  uniqueness of energy solutions, because we could replace convergence in
  probability below by $L^1$-convergence.
  
  In the theorem below we would need $\eta \in L^{2 + \kappa} (\mathbb{T}^d)$
  if $b_2 \neq 0$ to get these $L^1$-bounds, due to an application of
  H{\"o}lder's inequality to split off the stochastic exponential.
\end{remark}

\begin{theorem}[Existence of energy solutions in $\mathbb{T}^d$]
  \label{tightnesstorusBesov}Let $\mu \ll \tmop{Leb}$ be a probability measure
  on $\mathbb{T}^d$ and let $b_i : \mathbb{R}_+ \rightarrow \mathcal{S}'
  (\mathbb{T}^d, \mathbb{R}^d), i \in \{ 1, 2 \}$ with $\nabla \cdummy b_1
  \equiv 0$. Assume that
  \begin{enumerateroman}
    \item \label{exis2}$\hat{b_1} (0) \in L^2_{\tmop{loc}}$ and $A (b_1)
    \in L^q_{\tmop{loc}} L^p$ for some $p \in [2, \infty), q \in (2, \infty]$
    with $p \left( \frac{1}{2} - \frac{1}{q} \right) > 1$ (both of these
    conditions are satisfied if $b_1 \in L^q_{\tmop{loc}} B^{- 1}_{p, 2}$);
    
    \item \label{exis3}$b_2 \in L^4_{\tmop{loc}} B^0_{2 r, 1, 2}$ for some $r
    \in [d, \infty]$.
  \end{enumerateroman}
  Let $b = b_1 + b_2$ and let $X^n : \mathbb{R}_+ \rightarrow \mathbb{T}^d$
  solve
  \[ \mathd X_t^n = b^n (t, X_t^n) \mathd t + \sqrt{2} \mathd B_t, \]
  with $X^n_0 \sim \mu$, where we recall that $b^n = b \ast \rho^n$. Then,
  $(X^n)_{n \in \mathbb{N}}$ is tight on $C (\mathbb{R}_+, \mathbb{T}^d)$. Any
  limit point is an energy solution to $\mathd X_t = b (t, X_t) \mathd t +
  \sqrt{2} \mathd W_t$.
\end{theorem}

\begin{proof}
  To prove tightness of $(X^n)_{n \in \mathbb{N}}$, we first assume that
  $b_2^n \equiv 0$, i.e. the drift is divergence-free and $\eta \assign
  \frac{\mathd \mu}{ \tmop{dLeb}_{\mathbb{T}^d}} \equiv 1$, meaning that we
  consider the stationary initial condition. Then obviously $(X^n_0)_{n \in
  \mathbb{N}}$ and $B$ are tight. To prove tightness of the drift term $\left(
  \int_0^{\cdummy} b^n (s, X_s^n) \mathd s \right)_{n \in \mathbb{N}}$ we
  estimate for $0 \leqslant s \leqslant t \leqslant T$
  \begin{eqnarray}
    \mathbb{E} \left[ \left| \int_s^t b^{n, i} (r, X_r^n) \mathd r
    \right|^p \right]
    & \lesssim &  \mathbb{E} \left[ \left| \int_s^t \mathLaplace
    \mathLaplace^{- 1} b^{n, i} (r, X_r^n) \mathd r \right|^p \right]
     +\mathbb{E} \left[ \left| \int_s^t \mathcal{F} b^{n, i} (r)  (0) \mathd r
    \right|^p \right] \nonumber\\
    & \lesssim & | t - s |^{p \left( \frac{1}{2} - \frac{1}{q} \right)} \|
    \mathLaplace^{- 1} \nabla (\nabla \cdummy A_i (b^n)) \|^p_{L^q L^p
    (\mathbb{T}^d)}\nonumber \\ & & + | t - s |^{p / 2} \| \mathcal{F} b^n  (0) \|_{L^2_T}^p 
    \label{b2termtorus}\\
    & \lesssim_T & | t - s |^{p \left( \frac{1}{2} - \frac{1}{q} \right)} \{ \|
    A (b^n) \|^p_{L^q L^p (\mathbb{T}^d)} + \| \mathcal{F} b^n  (0)
    \|_{L^2_T}^p \} .  \label{tightnessbound}
  \end{eqnarray}
  Here we used that $\mathLaplace^{- 1} \nabla (\nabla \cdummy)$ as an
  operator on functions on $\mathbb{R}^d$ is of Calderon-Zygmund type (it
  behaves like the Leray projector) and hence it restricts to a bounded
  operator on $L^s (\mathbb{R}^d)$ for any $s \in (1, \infty)$ and thus the
  same is true on $L^s (\mathbb{T}^d)$ by de Leeuw's theorem {\cite{Leeuw65}}.
  Since $p (\frac{1}{2} - \frac{1}{q}) > 1$, this proves tightness by
  Kolmogorov's continuity criterion.
  
  If now $b^n_2 \neq 0$ and $\eta \neq 1$, we bound for any $A \in \mathcal{B}
  (C ([0, T], \mathbb{T}^d))$ and any $M > 0$
  \begin{eqnarray*}
    \mathbb{P} (X^n \in A) & = & \int_{\mathbb{T}^d} \eta (y_0)
    \mathbb{P}_{y_0} (X^n \in A) \mathd y_0\\
    & \leqslant & \int_{\mathbb{T}^d} \eta (y_0) \1_{\{ \eta (y_0) > M \}}
    \mathd y_0 + M\mathbb{P}_{\tmop{Leb}} (X^n \in A)\\
    & = & \int_{\mathbb{T}^d} \eta (y_0) \1_{\{ \eta (y_0) > M \}} \mathd y_0
    + M\mathbb{E}_{\tmop{Leb}} \left[ \1_{\{ \tilde{X}_n \in A \}} \mathcal{E}
    \left( \int_0^{\cdummy} \sqrt{2} b_2^n (r, \tilde{X}_r^n) \cdot \mathd B_r
    \right)_T \right]\\
    & \leqslant & \int_{\mathbb{T}^d} \eta (y_0) \1_{\{ \eta (y_0) > M \}}
    \mathd y_0\\
    &  & + M\mathbb{P}_{\tmop{Leb}} (\tilde{X}^n \in A)^{\frac{1}{2}}
    \mathbb{E}_{\tmop{Leb}} \left[ \mathcal{E} \left( \int_0^{\cdummy}
    \sqrt{2} b_2^n (r, \tilde{X}_r^n) \cdot \mathd B_r \right)_T^2
    \right]^{\frac{1}{2}},
  \end{eqnarray*}
  where $\tilde{X}^n$ is the process with drift $b^n_2 \equiv 0$ and
  $\mathbb{P}_{\tmop{Leb}}$ and $\mathbb{E}_{\tmop{Leb}}$ are the probability
  and expectation with respect to the stationary initial condition. By Lemmas
  \ref{stochexponential} and \ref{uniformboundinapproxnew} the second
  expectation on the right hand side is uniformly bounded in $n$, and given
  $\varepsilon > 0$ we can find $M > 0$ large enough so that
  $\int_{\mathbb{T}^d} \eta (y_0) \1_{\{ \eta (y_0) > M \}} \mathd y_0
  \leqslant \varepsilon / 2$. Therefore, up to changing the value of $M = M
  (\varepsilon, T) > 0$, we get from the weighted Young inequality for
  products
  \begin{equation}
    \mathbb{P} (X^n \in A) \leqslant \varepsilon + M\mathbb{P}_{\tmop{Leb}}
    (\tilde{X}^n \in A), \label{eq:ExpXn-bound}
  \end{equation}
  uniformly in $n$. Since $(\tilde{X}^n)_{n \in \mathbb{N}}$ is tight under
  $\mathbb{P}_{\tmop{Leb}}$, for any $T > 0$ and $\varepsilon > 0$ there
  exists a compact set $K \subset C ([0, T], \mathbb{T}^d)$ with $\sup_n
  \mathbb{P}_{\tmop{Leb}} (\tilde{X}^n |_{[0, T]} \in K^c) \leqslant
  \frac{\varepsilon}{M}$, which yields
  \[ \sup_n \mathbb{P} (X^n |_{[0, T]} \in K^c) \leqslant 2 \varepsilon, \]
  and then tightness of $(X^n)_{n \in \mathbb{N}}$.
  
  Let $X$ be a limit point. From~\eqref{eq:ExpXn-bound} we obtain with the
  Portmanteau theorem for open $U \subset \mathbb{R}^d$ by stationarity of the
  Lebesgue measure for $\tilde{X}^n$:
  \[ \mathbb{P} (X_t \in U) \leqslant \liminf_{n \rightarrow \infty}
     \mathbb{P} (X^n_t \in U) \leqslant \varepsilon + M \cdot \tmop{Leb} (U),
     \qquad t \in [0, T] . \]
  Since any probability measure on $\mathbb{R}^d$ is outer regular, we then
  obtain for $A \in \mathcal{B} (\mathbb{R}^d)$ and $t \in [0, T]$
  \begin{eqnarray*}
    \mathbb{P} (X_t \in A) & = & \inf \left\{ \mathbb{P} (X_t \in U) : U
    \supset A, U \text{ open} \right\}\\
    & \leqslant & \inf \left\{ \varepsilon + M \cdot \tmop{Leb} (U) : U
    \supset A, U \text{ open} \right\}\\
    & = & \varepsilon + M \cdot \tmop{Leb} (A),
  \end{eqnarray*}
  i.e. $X$ satisfies the first part of the incompressibility condition
  \ref{incompressible} of energy solutions.
  
  To see that \ref{admissible} and the second part of \ref{incompressible} is
  satisfied, we use~\eqref{eq:ExpXn-bound} with the open set $A = \left\{
  \sup_{t \leqslant T} \left| \int_0^t f (s, X_s) \mathd s \right| > \delta
  \right\}$, and we bound the right hand side with the It{\^o} trick:
  \begin{eqnarray*}
    \mathbb{P}_{\tmop{Leb}} \left( \left| \sup_{t \nocomma \leqslant T}
    \int_0^t f (r, \tilde{X}_r^n) \mathd r \right| > \delta \right) &
    \leqslant & \frac{1}{\delta} \mathbb{E}_{\tmop{Leb}} \left[ \left| \sup_{t
    \nocomma \leqslant T} \int_0^t \mathLaplace \mathLaplace^{- 1} f (r,
    \tilde{X}_r^n) \mathd r \right|^2 \right]^{1 / 2}\\
    &  & + \frac{1}{\delta} \mathbb{E}_{\tmop{Leb}} \left[ \left| \int_0^T |
    \hat{f} (r, 0) | \mathd r \right|^2 \right]^{1 / 2}\\
    & \lesssim & \frac{1}{\delta}  \left( \| \mathLaplace^{- 1} \nabla f
    \|_{L^2 L^2 (\mathbb{T}^d)} + \int_0^T | \hat{f} (r, 0) | \mathd r
    \right)\\
    & \lesssim_T & \frac{1}{\delta} \| f \|_{L^2_T H^{- 1} (\mathbb{T}^d)},
  \end{eqnarray*}
  uniformly in $n$. Hence, we can pass to the limit with another application
  of the Portmanteau theorem. Analogously, we obtain
  \[ \mathbb{P}_{\tmop{Leb}} \left( \left| \sup_{t \nocomma \leqslant T}
     \int_0^t f (r, \tilde{X}_r^n) \mathd r \right| > \delta \right) \leqslant
     \frac{1}{\delta} \mathbb{E}_{\tmop{Leb}} \left[ \int_0^T | f (r,
     \tilde{X}_r^n) | \mathd r \right] = \frac{1}{\delta} \| f \|_{L^1_T L^1}
     . \]
  In order to prove \ref{martingaleproperty}, we consider $f \in C^{\infty}
  ([0, T] \times \mathbb{T}^d)$ and the sequence of martingales
  \[ M^{f, n}_t = f (t, X^n_t) - f (0, X^n_0) - \int_0^t (\partial_s +
     \mathLaplace + b^n \cdummy \nabla) f (s, X^n_s) \mathd s, \qquad t \in
     [0, T] . \]
  By the same arguments as for $(X^n)_{n \in \mathbb{N}}$, we get joint
  tightness of $(X^n, M^{f, n})_{n \in \mathbb{N}}$. Any limit point $(X,
  M^f)$ is such that $M^f$ is a local martingale in the filtration generated
  by $(X, M^f)$, see Proposition~IX.1.17 in {\cite{Jacod2003}}, and, by
  continuity of $f (t, \cdummy)$ and $f (0, \cdummy)$, we know that $M^f_t = f
  (t, X_t) - f (0, X_0) - I_t$, for a process $I$ to be identified. Note that
  \begin{eqnarray}
    \| (b^m - b^n) (t, \cdummy) \cdummy \nabla f (t, \cdummy) \|_{H^{- 1}} &
    \simeq &\| (b^m - b^n) (t, \cdummy) \cdummy \nabla f (t, \cdummy) \|_{B^{-
    1}_{2, 2}} \nonumber\\
    & \lesssim & \| (b^m - b^n) (t, \cdummy) \|_{B^{- 1}_{2, 2}} \| \nabla f
    (t, \cdummy) \|_{B^{1 + \kappa}_{\infty, 2}} \nonumber\\
    & \simeq & \| (b^m - b^n) (t, \cdummy) \|_{H^{- 1}} \| \nabla f (t,
    \cdummy) \|_{B^{1 + \kappa}_{\infty, 2}},  \label{driftestimateforsmoothf}
  \end{eqnarray}
  see e.g. Theorem 27.11 in {\cite{vanZuijlen22}}, and thus we obtain from the
  It{\^o} trick together with \eqref{eq:ExpXn-bound} for all $m \leqslant n$
  \[ \mathbb{P} \left( \sup_{t \leqslant T} \left| \int_0^t (b^m - b^n)
     \cdummy \nabla f (s, X^n_s) \mathd s \right| > \delta \right) \leqslant
     \varepsilon + \frac{M}{\delta} \| (b^m - b^n) (t, \cdummy) \|_{H^{- 1}}
     \| \nabla f (t, \cdummy) \|_{B^{1 + \varepsilon}_{\infty, 2}} . \]
  Passing to the limit $n \rightarrow \infty$, we deduce that
  \begin{eqnarray*}
    && \mathbb{P} \left( \sup_{t \leqslant T} \left| I_t - \int_0^t 
    (\partial_s + \mathLaplace + b^m \cdummy \nabla) f (s, X_s) \mathd s
    \right| > \delta \right)\\
    &\leqslant & \varepsilon + \frac{M}{\delta} \| (b^m - b) (t, \cdummy)
    \|_{H^{- 1}} \| \nabla f (t, \cdummy) \|_{B^{1 + \varepsilon}_{\infty,
    2}},
  \end{eqnarray*}
  from where we obtain by sending $m \rightarrow \infty$ that
  \[ I_t = I ((\partial_s + \mathLaplace + b \cdummy \nabla) f)_t \]
  as claimed.
  
  Finally, we have to show that $\langle M^f \rangle_t = \int_0^t | \nabla f
  (s, X_s) |^2 \mathd s$. But since for all $n \in \mathbb{N}$ the process
  \[ (M^{f, n}_t)^2 - \int_0^t | \nabla f (s, X^n_s) |^2 \mathd s, \qquad t
     \in [0, T], \]
  is a local martingale, we get again by Proposition~IX.1.17 in
  {\cite{Jacod2003}} that
  \[ (M^f_t)^2 - \int_0^t | \nabla f (s, X_s) |^2 \mathd s, \qquad t \in [0,
     T], \]
  is a local martingale. This concludes the proof.
\end{proof}

\subsection{Existence of energy solutions in
$\mathbb{R}^d$}\label{Euclideanset}

Formally, the Lebesgue measure is invariant for $\mathd X_t = b (t, X_t)
\mathd t + \sqrt{2} \mathd B_t$ with divergence-free measurable $b :
\mathbb{R}_+ \times \mathbb{R}^d \rightarrow \mathbb{R}^d$. Since
$\tmop{Leb}_{\mathbb{R}^d}$ is not a probability measure we can not use it as
initial distribution for a stochastic process in order to apply the
forward-backward martingale argument from Lemma \ref{Itotricklemma}, which was
essential for proving the existence of energy solutions on $\mathbb{T}^d$. In
the following we prove an It{\^o} trick type bound despite this non-existence
of a stationary probability measure. The main idea is to consider a sequence
of tori $\mathbb{T}^d_m$ of length $m \rightarrow \infty$.

\begin{lemma}[It{\^o} trick]
  \label{ItotricklemmaEuclid}Let $b \in C (\mathbb{R}_+ ; C^{\infty}_b
  (\mathbb{R}^d))$ and let $X : \mathbb{R}_+ \rightarrow \mathbb{R}^d$ solve
  \[ \mathd X_t = b (t, X_t) \mathd t + \sqrt{2} \mathd B_t, \]
  where $B$ is a Brownian motion, $\nabla \cdummy b \equiv 0$ and $b, \nabla b
  \in L^{\infty}_{\tmop{loc}} L^s$ for some $s \in [1, \infty)$. Assume that
  $X (0) \sim \mu$ with $\eta \assign \frac{\mathd \mu}{
  \tmop{dLeb}_{\mathbb{R}^d}} \in L^{\infty} (\mathbb{R}^d)$. Then, for all $T
  > 0$ and $f \in C^{\infty}_c ([0, T] \times \mathbb{R}^d, \mathbb{R})$
  \begin{equation}
    \mathbb{E} \left[ \sup_{0 \leqslant t \leqslant T} \left| \int_0^t
    \mathLaplace f (s, X_s) \mathd s \right|^p \right] \lesssim \| \eta
    \|_{L^{\infty} (\mathbb{R}^d)} T^{p \left( \frac{1}{2} - \frac{1}{q}
    \right)} \| \nabla f \|^p_{L^q L^p (\mathbb{R}^d)}, \label{ItotrickEuclid}
  \end{equation}
  for all $p \in [2, \infty)$ and $q \in [2, \infty]$. The implicit constant
  on the right hand side is independent of $b$.
\end{lemma}

\begin{proof}
  We first assume $\eta \in C^{\infty}_c$ and $b \in C (\mathbb{R}_+ ;
  C_c^{\infty} (\mathbb{R}^d))$. We define
  \[ g^m (t, x) \assign \sum_{k \in \mathbb{Z}^d} g (t, x + k m), \]
  for $g \in \{ b, \eta, f \}$. Since $b^m$ and $\eta^m$ are smooth and
  $m$-periodic, we interpret them as elements of $C^{\infty}
  (\mathbb{T}^d_m)$, where $\mathbb{T}^d_m$ is the torus of length $m$. Let
  $X^m : \mathbb{R}_+ \rightarrow \mathbb{R}^d$ be a solution to
  \begin{equation}
    \mathd X^m_t = b^m (t, X^m_t) \mathd t + \sqrt{2} \mathd B_t,
    \label{periodizedSDE}
  \end{equation}
  with $X^m (0) \sim \eta  \tmop{dLeb}_{\mathbb{R}^d}$. Then, the projection
  $Y^m (0) = X^m (0) / (m\mathbb{Z}^d)$ has density $\eta^m$ w.r.t.
  $\tmop{Leb}_{\mathbb{T}^d_m}$. Switching to the invariant measure, we obtain
  as in the proof of Theorem~\ref{tightnesstorusBesov} from H{\"o}lder's
  inequality and Jensen's inequality
  \begin{eqnarray*}
    \mathbb{E} [| \Phi (Y^m) |] & \leqslant & \| \eta^m \|_{L^{\infty}
    (\mathbb{T}^d_m)} \int_{\mathbb{T}^d_m} \mathbb{E}_{y_0} [| \Phi (Y^m) |]
    \mathd y_0\\
    & = &\| \eta^m \|_{L^{\infty} (\mathbb{T}^d_m)} m^d
    \mathbb{E}_{\tmop{stat}} [| \Phi (Y^m) |],
  \end{eqnarray*}
  where $\mathbb{E}_{\tmop{stat}}$ is the expectation under the stationary
  initial distribution $m^{- d} \tmop{Leb} (\mathbb{T}^d_m)$. In particular,
  we get from the same arguments as in Lemma~\ref{Itotricklemma}
  \begin{eqnarray*}
    \mathbb{E} \left[ \sup_{0 \leqslant t \leqslant T} \left| \int_0^t
    \mathLaplace f^m (s, X^m_s) \mathd s \right|^p \right] &=&\mathbb{E} \left[
    \sup_{0 \leqslant t \leqslant T} \left| \int_0^t \mathLaplace f^m (s,
    Y^m_s) \mathd s \right|^p \right]\\
    &\leqslant & \| \eta^m \|_{L^{\infty} (\mathbb{T}^d_m)} m^d
    \mathbb{E}_{\tmop{stat}} \left[ \sup_{0 \leqslant t \leqslant T} \left|
    \int_0^t \mathLaplace f^m (s, Y^m_s) \mathd s \right|^p \right]\\
    &\lesssim & \| \eta^m \|_{L^{\infty} (\mathbb{T}^d_m)} m^d
    T^{p \left( \frac{1}{2} - \frac{1}{q} \right)} \| \nabla f^m \|^p_{L^q_T
    L^p (m^{- d} \mathbb{T}^d_m)}\\
    & = & \| \eta^m \|_{L^{\infty} (\mathbb{T}^d_m)} T^{p \left(
    \frac{1}{2} - \frac{1}{q} \right)} \| \nabla f^m \|^p_{L^q_T L^p
    (\mathbb{T}^d_m)} .
  \end{eqnarray*}
  Now it remains to pass to the limit $m \rightarrow \infty$ on both sides of
  this inequality. The convergence $\| \eta^m \|_{L^{\infty} (\mathbb{T}^d_m)}
  \rightarrow \| \eta \|_{L^{\infty} (\mathbb{R}^d)}$, $\| \nabla f^m
  \|^p_{L^q_T L^p (\mathbb{T}^d_m)} \rightarrow \| \nabla f \|^p_{L^q_T L^p
  (\mathbb{R}^d)}$ is clear since $\eta$ and $f$ are compactly supported.
  Hence it remains to show convergence of the left hand side. To justify this,
  we can use that
  \begin{equation}
    \mathbb{E} [\sup_{0 \leqslant t \leqslant T} | X_t - X^m_t |^p]
    \rightarrow 0, \label{verygoodconvergence}
  \end{equation}
  which by Gronwall's inequality reduces to showing
  \[ \int_0^T \mathbb{E} [\sup_{0 \leqslant s \leqslant t} | b (s, X_s) - b^m
     (s, X_s) |^p] \mathd t \rightarrow 0. \]
  This convergence holds by the dominated convergence theorem since $(b -
  b^m)$ is uniformly bounded in $s, x$, converges locally uniformly to $0$,
  and since $X$ has locally bounded trajectories.
  Using~\eqref{verygoodconvergence}, we obtain the following convergence by
  another application of the dominated convergence theorem together with the
  locally uniform convergence of $\Delta f^m$ to $\Delta f$
  \[ \lim_{m \rightarrow \infty} \mathbb{E} \left[ \sup_{0 \leqslant t
     \leqslant T} \left| \int_0^t (\mathLaplace f (s, X_s) - \mathLaplace f^m
     (s, X^m_s)) \mathd s \right|^p \right] = 0. \]
  Therefore,
  \begin{equation}
    \mathbb{E} \left[ \sup_{0 \leqslant t \leqslant T} \left| \int_0^t
    \mathLaplace f (s, X_s) \mathd s \right|^p  \right] \lesssim \| \eta
    \|_{L^{\infty}} T^{p \left( \frac{1}{2} - \frac{1}{q} \right)} \| \nabla f
    \|^p_{L^q L^p} . \label{compactlysupportedb}
  \end{equation}
  It remains to remove the assumption $\eta \in C^{\infty}_c$ and $b \in C
  (\mathbb{R}_+ ; C_c^{\infty} (\mathbb{R}^d))$. For $\eta$ we can simply use
  the Markov property. To approximate $b \in C (\mathbb{R}_+, C^{\infty}_b
  (\mathbb{R}^d))$ with $(b^{(\ell)})_{\ell} \subset C (\mathbb{R}_+ ;
  C_c^{\infty} (\mathbb{R}^d))$, we consider $\varphi \in C^{\infty}_c$ such
  that $\varphi \equiv 1$ on a ball around $0$ and we set $\varphi_a = \varphi
  (a \cdummy)$ for $a > 0$ and $\phi_a = \mathcal{F}^{- 1} (\varphi_a)$. Then,
  we define
  \[ b^{(\ell)} = b (\varphi_{\ell^{- 1}} A (b - \phi_{\ell} \ast b)) \in C
     (\mathbb{R}_+ ; C_c^{\infty} (\mathbb{R}^d)), \]
  for $\ell \in \mathbb{N}$, where $b (A) = \nabla \cdummy A$ is the map from
  (\ref{bofA}), which is local. We claim that $b^{(\ell)} \rightarrow b$ and
  $\nabla b^{(\ell)} \rightarrow \nabla b$ locally uniformly in $(t, x)$ as
  $\ell \rightarrow \infty$. Indeed, for compact $K$ and large enough $\ell$
  we have
  \[ b^{(\ell)} |_K = b (A (b - \phi_{\ell} \ast b_1)) |_K = b|_K -
     \phi_{\ell} \ast b|_K . \]
  Next, by Young's convolution inequality $\| \phi_{\ell} \ast b
  \|_{L^{\infty}_T L^{\infty}} \lesssim \| b \|_{L^{\infty}_T L^s
  (\mathbb{R}^d)} \| \phi_{\ell} \|_{L^r}$, where $r \in (1, \infty]$ is such
  that $\frac{1}{s} + \frac{1}{r} = 1$, and $\| \phi_{\ell} \|_{L^r} = \ell^{-
  d + d / r} \| \phi \|_{L^r}$ converges to $0$. Therefore, $b^{(\ell)}
  \rightarrow b$ uniformly in $[0, T] \times K$ and by the same argument
  $\nabla (\phi_{\ell} \ast b_1) \rightarrow 0$ uniformly in $[0, T] \times
  K$. We also observe that $b^{(\ell)}$ is uniformly bounded in $C_T C^1_b
  (\mathbb{R}^d)$ since
  \[ b^{(\ell)} = b (\varphi_{\ell^{- 1}} A (b - \phi_{\ell} \ast b)) =
     \varphi_{\ell^{- 1}} (b - \phi_{\ell} \ast b) + \nabla \varphi_{\ell^{-
     1}}^T A (b - \phi_{\ell} \ast b), \]
  and the first term is uniformly bounded in $C_T C^1_b$ by the same arguments
  as above, while for the second term
  \begin{eqnarray*}
    \| \nabla \varphi_{\ell^{- 1}}^T A (b_1 - \phi_{\ell} \ast b_1)
    \|_{L^{\infty}_T L^{\infty}} & \lesssim & \| \nabla \varphi_{\ell^{- 1}}
    \|_{L^{\infty}} \| A (b_1 - \phi_{\ell} \ast b) \|_{B^0_{\infty, 1}}\\
    & \lesssim & \ell^{- 1} \| \nabla \varphi \|_{L^{\infty}} \ell \| b
    \|_{B^0_{\infty, 1}} \\
    &=& \| \nabla \varphi \|_{L^{\infty}} \| b
    \|_{B^0_{\infty, 1}} .
  \end{eqnarray*}
  We argue similarly for the derivatives of $b^{(\ell)}$. From here it follows
  as above that the solution $X^{(\ell)}$ to the SDE with drift $b^{(\ell)}$
  and initial condition $\eta$ satisfies
  \[ \lim_{\ell \rightarrow \infty} \mathbb{E} \left[ \sup_{0 \leqslant t
     \leqslant T} \left| \int_0^t (\mathLaplace f (s, X_s) - \Delta f (s,
     X^{(\ell)}_s)) \mathd s \right|^p \right] = 0, \]
  so that our claim follows.
\end{proof}

If the drift is not divergence free, $b = b_1 + b_2$ with $b_2 \neq 0$, we
could use the periodic approximation as in the previous proof to bound
\[ \mathbb{E} \left[ \mathcal{E} \left( \int_0^{\cdummy} \sqrt{2} b_2^m (s,
   X^m_s) \cdummy \mathd B_s \right)_T^p \right] \leqslant \| \eta^m
   \|_{L^{\infty} (\mathbb{T}^d_m)} m^d \mathbb{E}_{\tmop{stat}} \left[
   \mathcal{E} \left( \int_0^{\cdummy} \sqrt{2} b_2^m (s, X^m_s) \cdummy
   \mathd B_s \right)_T^p \right] . \]
Unfortunately, the estimate that we used for the stochastic exponential in the
periodic case does not produce a factor $m^{- d}$ and therefore we need to
argue differently.

\begin{lemma}[Novikov type bound]
  \label{NoviovboundEuclid}Let $b$, $X$ and $\eta$ be as in the previous
  lemma. Then, for all $T > 0$, $p > 0$ and $r \in [d, \infty]$ there exists
  $c > 0$ such that for all $a \in L^4_T B^0_{2 r, 1, 2}$
  \begin{equation}
    \mathbb{E} \left[ \mathcal{E} \left( \int_0^{\cdummy} a (s, X_s) \cdummy
    \mathd B_s \right)_T^p \right] \lesssim_T (1 + \| \eta \|_{L^{\infty}
    (\mathbb{R}^d)}) e^{c \| a \|^4_{L^4_T B^0_{2 r, 1, 2}}},
    \label{NovikovEuclid}
  \end{equation}
  where the implicit constant on the right hand side and $c > 0$ are
  independent of $b$.
\end{lemma}

\begin{proof}
  As in the periodic case it suffices to bound $\mathbb{E} \left[ \exp \left(
  c \int_0^T | a (s, X_s) |^2 \mathd s \right) \right]$ for all $c > 0$, and
  again we allow $c > 0$ to change in every step. Since the case $r = \infty$
  can be treated directly with
  \[ \mathbb{E} \left[ \exp \left( c \int_0^T | a (s, X_s) |^2 \mathd s
     \right) \right] \leqslant \exp (c \| a \|^2_{L^2_T L^{\infty}
     (\mathbb{R}^d)}), \]
  we may assume $r < \infty$. We first argue for $a \in C^{\infty}_c ([0, T]
  \times \mathbb{R}^d)$. Expanding the exponential, we obtain
  \begin{eqnarray}
    \mathbb{E} \left[ \exp \left( c \int_0^T | a (s, X_s) |^2 \mathd s \right)
    \right] & = & \sum_N \frac{1}{N!} c^N \mathbb{E} \left[ \left| \int_0^T |
    a (s, X_s) |^2 \mathd s \right|^N \right] .  \label{taylorsum}
  \end{eqnarray}
  For even $N \geqslant r$ we estimate with the It{\^o} trick from
  Lemma~\ref{ItotricklemmaEuclid} above
  \begin{eqnarray*}
    \mathbb{E} \left[ \left| \int_0^T | a (s, X_s) |^2 \mathd s \right|^N
    \right] & \leqslant & 2^N \mathbb{E} \left[ \left| \int_0^T \mathLaplace
    (1 - \mathLaplace)^{- 1} | a (s, X_s) |^2 \mathd s \right|^N \right]\\
    &  & + 2^N \mathbb{E} \left[ \left| \int_0^T (1 - \mathLaplace)^{- 1} | a
    (s, X_s) |^2 \mathd s \right|^N \right]\\
    & \lesssim & 2^N \| \eta \|_{L^{\infty}} C_{\tmop{BDG}} (N)^N \| \nabla
    (1 - \mathLaplace)^{- 1} | a |^2 \|_{L^2_T L^N}^N\\
    &  & + 2^N \| (1 - \mathLaplace)^{- 1} | a |^2 \|_{L^1_T L^{\infty}}^N\\
    & \leqslant & 2^N \| \eta \|_{L^{\infty}} C_{\tmop{BDG}} (N)^N \| \nabla
    (1 - \mathLaplace)^{- 1} | a |^2 \|_{L^2_T L^N}^N\\
    &  & + 2^N T^{1 / 2} \| (1 - \mathLaplace)^{- 1} | a |^2 \|_{L^2_T
    L^{\infty}}^N,
  \end{eqnarray*}
  where $C_{\tmop{BDG}} (N)$ is the constant from the Burkholder-Davis-Gundy
  inequality \\ $\| M_T^{\ast} \|_{L^N} \leqslant C_{\tmop{BDG}} (N) \| \langle M
  \rangle_T \|_{L^{N / 2}}^{1 / 2}$ for continuous local martingales $M$,
  which was implicitly used in the proof of Lemma \ref{Itotricklemma}, and we
  may restrict to local martingales given as stochastic integrals with respect
  to Brownian motion. In that case
  \[ C_{\tmop{BDG}} (N) \lesssim \sqrt{N} . \]
  Indeed, in Section 3 of {\cite{Burgess76}} it is shown that the optimal
  choice of $C_{\tmop{BDG}} (2 n)$ for stochastic integrals with respect to
  Brownian motion is given by the largest positive zero $x_0$ of the $2 n$-th
  Hermite polynomial. In Chapter VI, 6.2. of {\cite{Szegoe75}} the bound $|
  x_0 | \lesssim \sqrt{2 n}$ is shown.
  
  We now use the embedding $B^0_{s, 1} \hookrightarrow L^s$, which is uniform
  in $s$ as can be seen by an expansion in Littlewood-Paley blocks, and then
  Besov embedding (which is also uniform in the parameters, see Theorem 21.23
  in {\cite{vanZuijlen22}}) with $0 \leqslant 1 - d \left( \frac{1}{r} -
  \frac{1}{N} \right)$ to obtain
  \begin{eqnarray*}
    \| \nabla (1 - \mathLaplace)^{- 1} | a |^2 \|_{L^2_T L^N} & \lesssim & \|
    \nabla (1 - \mathLaplace)^{- 1} | a |^2 \|_{L^2_T B^0_{N, 1}}\\
    & \lesssim & \| \nabla (1 - \mathLaplace)^{- 1} | a |^2 \|_{L^2_T B^1_{r,
    1}}\\
    & \lesssim & \| | a |^2 \|_{L^2_T B^0_{r, 1}}^N .
  \end{eqnarray*}
  The same argument gives
  \[ \| (1 - \mathLaplace)^{- 1} | a |^2 \|_{L^2_T L^{\infty}}^N \lesssim \| |
     a |^2 \|_{L^2_T B^0_{r, 1}}^N, \]
  and therefore for some $C > 0$,
  \begin{eqnarray*}
    c^N \mathbb{E} \left[ \left| \int_0^t | a (s, X_s) |^2 \mathd s \right|^N
    \right] & \lesssim_T & (1 + \| \eta \|_{L^{\infty}}) C^N N^{N / 2} \| | a
    |^2 \|_{L^2_T B^0_{r, 1}}^N .
  \end{eqnarray*}
  If $N \geqslant r$ is odd, we simply bound $\mathbb{E} [| \cdummy |^N]
  \leqslant \mathbb{E} [| \cdummy |^{N + 1}]^{N / N + 1}$ and then apply the
  above estimate which thus also holds true in the odd case by changing the
  value of $C$.
  
  For $N < r$ we simply estimate $\mathbb{E} \left[ \left| \int_0^t | a (s,
  X_s) |^2 \mathd s \right|^N \right] \leqslant \mathbb{E} \left[ \left|
  \int_0^t | a (s, X_s) |^2 \mathd s \right|^M \right]^{N / M}$ where $M
  \geqslant r$, and thus we obtain the same bound up to redefining $C$.
  Revisiting (\ref{taylorsum}), it remains to check that
  \[ \sum_{N = 0}^{\infty} \frac{1}{N!} N^{N / 2} C^N x^N \lesssim \sum_{N =
     0}^{\infty} \frac{1}{\sqrt{N!}} C^N x^N \lesssim e^{c x^2}, \]
  which follows from Stirling's formula $N^N C^N \gtrsim N! \gtrsim N^N C^{-
  N}$ for some $C > 0$ (which may change in each step in the computation
  above). Thus, we arrive at
  \begin{eqnarray*}
    \mathbb{E} \left[ \exp \left( c \int_0^T | a (s, X_s) |^2 \mathd s \right)
    \right] & \lesssim & (1 + \| \eta \|_{L^{\infty}}) \exp (C \| | a |^2
    \|^2_{L^2_T B^0_{r, 1}}) .
  \end{eqnarray*}
  Using Lemma (\ref{estimateforsquarenorm}) we can bound
  \[ \| | a |^2 \|_{L^2_T B^0_{r, 1}} \lesssim \| a \|_{L^4_T B^0_{2 r, 1,
     2}}^2 . \]
  Finally, approximating general $a \in L^4_T B^0_{2 r, 1, 2}$ with elements
  in $C^{\infty}_c ([0, T] \times \mathbb{R}^d)$ and applying the Fatou's
  lemma on the left hand side we can conclude the proof.
\end{proof}

\begin{definition}[Energy solution]
  \label{defenergysolutionsEuclid}Let $b : \mathbb{R}_+ \rightarrow
  \mathcal{S}' (\mathbb{R}^d, \mathbb{R}^d)$ be such that for all $T > 0$ and
  $f \in C^{\infty}_c ([0, T] \times \mathbb{R}^d)$ we have $b \cdummy \nabla
  f \in L^2_T H^{- 1}$. Let $X$ be a stochastic process with values in $C
  (\mathbb{R}_+ ; \mathbb{R}^d)$ such that for all $T > 0$:
  \begin{enumerateroman}
    \item \label{incompressibleEucild}$X$ is {\tmem{incompressible in
    probability}}, i.e. for all $\varepsilon > 0$ and $T > 0$ there exists $M
    = M (\varepsilon, T) > 0$ such that for all $A \in \mathcal{B}
    (\mathbb{R}^d)$ with the Lebesgue measure $\tmop{Leb} (A)$
    \[ \mathbb{P} (X_t \in A) \leqslant \varepsilon + M \cdot \tmop{Leb} (A),
       \qquad t \in [0, T], \]
    and
    \[ \mathbb{P} \left( \sup_{t \leqslant T} \left| \int_0^t f (s, X_s)
       \mathd s \right| > \delta \right) \leqslant \varepsilon +
       \frac{M}{\delta} \| f \|_{L^1_T L^1 (\mathbb{R}^d)}, \]
    for all for all $f \in C^{\infty}_c ([0, T] \times \mathbb{R}^d)$.
    
    \item \label{admissibleEucild}X is {\tmem{admissible in
    probability}}/satisfies an {\tmem{energy estimate in probability}}, i.e.
    for all $\varepsilon > 0$ and $T > 0$ there exists $M = M (\varepsilon, T)
    > 0$ such that for all $f \in C^{\infty}_c ([0, T] \times \mathbb{R}^d)$
    \begin{equation}
      \mathbb{P} \left( \sup_{t \leqslant T} \left| \int_0^t f (s, X_s) \mathd
      s \right| > \delta \right) \leqslant \varepsilon + \frac{M}{\delta} \| f
      \|_{L^2_T H^{- 1} (\mathbb{R}^d)} . \label{ItotrickboundEuclid}
    \end{equation}
    \item \label{martingalepropertyEucild}For any $f \in C^{\infty}_c ([0, T]
    \times \mathbb{R}^d)$, the process
    \[ M^f_t = f (t, X_t) - f (0, X_0) - \int_0^t (\partial_s + \mathLaplace +
       b \cdummy \nabla) f (s, X_s) \mathd s, \qquad t \in [0, T], \]
    is a (continuous) local martingale in the filtration generated by $X$,
    where the integral is defined as
    \[ \int_0^t (\partial_s + \mathLaplace + b \cdummy \nabla) f (s, X_s)
       \mathd s \assign I ((\partial_s + \mathLaplace + b \cdummy \nabla)
       f)_t, \]
    with $I$ the unique continuous extension from $C^{\infty}_c ([0, T] \times
    \mathbb{R}^d)$ to $L^2_T H^{- 1} (\mathbb{R}^d)$ of the map $g \mapsto
    \int_0^{\cdummy} g (s, X_s) \mathd s$ taking values in the continuous
    adapted stochastic processes equipped with the topology of uniform
    convergence on compacts (ucp-topology), and the extension $I$ exists by
    (\ref{ItotrickboundEuclid}).
    
    \item \label{quadraticvariationEuclid}The local martingale $M^f$ from
    \ref{martingalepropertyEucild} has quadratic variation
    \[ \langle M^f \rangle_t = \int_0^t | \nabla f (s, X_s) |^2 \mathd s . \]
  \end{enumerateroman}
  Then $X$ is called an {\tmem{energy solution}} to the SDE $\mathd X_t = b
  (t, X_t) \mathd t + \sqrt{2} \mathd W_t$.
\end{definition}

Note that \ref{quadraticvariation} implies that the local martingale from
\ref{martingaleproperty} is a proper martingale.

\begin{theorem}[Existence of energy solutions]
  \label{tightnesstorusBesovEuclid}Let $\mu \ll \tmop{Leb}$ be a probability
  measure on $\mathbb{R}^d$, and let $b_i : \mathbb{R}_+ \rightarrow
  \mathcal{S}' (\mathbb{R}^d, \mathbb{R}^d), i \in \{ 1, 2 \}$ with $\nabla
  \cdummy b_1 \equiv 0$ be such that
  \begin{enumerateroman}
    \item \label{exis2Eucild}$b_1 = b (A), A \in L^q_{\tmop{loc}} L^p$ or $b_1
    \in L^q_{\tmop{loc}} B^{- 1}_{p, 2}$, for some $p \left( \frac{1}{2} -
    \frac{1}{q} \right) > 1$, where $p \in [2, \infty)$ and $q \in [2,
    \infty]$;
    
    \item \label{exis3Eucild}$b_2 \in L^4_{\tmop{loc}} B^0_{2 r, 1, 2} $ for
    some $r \in [d, \infty]$.
  \end{enumerateroman}
  Let $b = b_1 + b_2$ and let $X^n : \mathbb{R}_+ \rightarrow \mathbb{R}^d$
  solve
  \[ \mathd X_t^n = b^n (t, X_t^n) \mathd t + \sqrt{2} \mathd B_t, \]
  with $X^n_0 \sim \mu$, where we recall that $b^n = b \ast \rho^n$. Then,
  $(X^n)_{n \in \mathbb{N}}$ is tight on $C (\mathbb{R}_+, \mathbb{R}^d)$. Any
  limit point is an energy solution to $\mathd X_t = b (t, X_t) \mathd t +
  \sqrt{2} \mathd W_t$.
\end{theorem}

\begin{proof}
  The proof follows closely that of Theorem \ref{tightnesstorusBesov}. Again
  we first assume that $b_2 \equiv 0$ and that $\eta = \frac{\mathd \mu}{
  \tmop{dLeb}} \in L^{\infty} (\mathbb{R}^d)$. For the case $b_1 = b (A)$ with
  $A \in L^q_T L^p (\mathbb{R}^d)$ we use the same bound based on the It{\^o}
  trick as in Theorem~\ref{tightnesstorusBesov}:
  \begin{eqnarray*}
    \mathbb{E} \left[ \left| \int_s^t b^{n, i} (r, X_r^n) \mathd r \right|^p
    \right] & =&\mathbb{E} \left[ \left| \int_s^t \mathLaplace \mathLaplace^{-
    1} b^{n, i} (r, X_r^n) \mathd r \right|^p \right]\\
    & \lesssim &\| \eta \|_{L^{\infty} (\mathbb{R}^d)} | t - s |^{p \left(
    \frac{1}{2} - \frac{1}{q} \right)} \| \mathLaplace^{- 1} \nabla (\nabla
    \cdummy A_i (b^n)) \|^p_{L^q L^p (\mathbb{R}^d)}\\
    & \lesssim & \| \eta \|_{L^{\infty} (\mathbb{R}^d)} | t - s |^{p \left(
    \frac{1}{2} - \frac{1}{q} \right)} \| A (b^n) \|^p_{L^q L^p
    (\mathbb{R}^d)},
  \end{eqnarray*}
  using again the Calderon-Zygmund property of $f \mapsto \nabla \cdummy
  \mathLaplace^{- 1} (\nabla \cdummy f)$. For $b_1 \in L^q_T B^{- 1}_{p, 2}$,
  which now does not imply $A (b) \in L^q_T L^p$, we write $b = (1 - \Delta)
  (1 - \Delta)^{- 1} b$, and then we apply Lemma \ref{ItotricklemmaEuclid} to
  bound
  \begin{eqnarray*}
    \mathbb{E} \left[ \left| \int_s^t b^n (r, X_r^n) \mathd r \right|^p
    \right] & \lesssim &\mathbb{E} \left[ \left| \int_s^t (1 - \mathLaplace)^{-
    1} b^n (r, X_r^n) \mathd r \right|^p \right]\\
    & &+ \| \eta \|_{L^{\infty} (\mathbb{R}^d)} | t - s |^{p \left(
    \frac{1}{2} - \frac{1}{q} \right)} \| \nabla (1 - \mathLaplace)^{- 1} b^n
    \|^p_{L^q_T L^p (\mathbb{R}^d)}\\
    & \lesssim &| t - s |^{p \left( 1 - \frac{1}{q} \right)} \left( \int_s^t
    \mathbb{E} [| (1 - \mathLaplace)^{- 1} b^n (r, X_r^n) |^p]^{q / p} \mathd
    r \right)^{p / q}\\
    & & + \| \eta \|_{L^{\infty} (\mathbb{R}^d)} | t - s |^{p \left(
    \frac{1}{2} - \frac{1}{q} \right)} \| b^n \|^p_{L^q_T B_{p, 2}^{- 1}
    (\mathbb{R}^d)}\\
    & \lesssim &| t - s |^{p \left( 1 - \frac{1}{q} \right)} \| \eta
    \|_{L^{\infty} (\mathbb{R}^d)} \| b^n \|^p_{L^q_T B_{p, 2}^{- 2}
    (\mathbb{R}^d)}\\
    & &+ \| \eta \|_{L^{\infty} (\mathbb{R}^d)} | t - s |^{p \left(
    \frac{1}{2} - \frac{1}{q} \right)} \| b^n \|^p_{L^q_T B_{p, 2}^{- 1}
    (\mathbb{R}^d)}\\
    & \lesssim_T & \| \eta \|_{L^{\infty} (\mathbb{R}^d)} | t - s |^{p \left(
    \frac{1}{2} - \frac{1}{q} \right)} \| b^n \|^p_{L^q_T B_{p, 2}^{- 1}
    (\mathbb{R}^d)},
  \end{eqnarray*}
  where we used that the density $\eta^n$ of $X^n$ solves the Kolmogorov
  forward equation $\partial_t \eta^n = \Delta \eta^n - \nabla \cdummy (b^n
  \eta^n) = \Delta \eta^n - b^n \cdummy \nabla \eta^n$ which satisfies the
  maximum principle, so that $\| \eta^n (t) \|_{L^{\infty} (\mathbb{R}^d)}
  \leqslant \| \eta (0) \|_{L^{\infty} (\mathbb{R})}$. Since $p (\frac{1}{2} -
  \frac{1}{q}) > 1$, this proves tightness by Kolmogorov's continuity
  criterion.
  
  If $b_2 \neq 0$ and $\eta \in L^1 (\mathbb{R}^d)$, then we let $\varepsilon
  > 0$ and $K$ be large enough so that $\int \eta \1_{\{ \eta > K \}}
  \leqslant \varepsilon$ and we set $\tilde{\eta} = \tilde{\eta}_{\varepsilon}
  = \frac{\eta \wedge K}{\int (\eta \wedge K)}$. Note that
  \[ 1 \geqslant \int (\eta \wedge K) \geqslant \int \eta \1_{\{ \eta
     \leqslant K \}} = 1 - \int \eta \1_{\{ \eta > K \}} \geqslant 1 -
     \varepsilon, \]
  and thus $\eta \wedge K \leqslant \tilde{\eta} \leqslant \frac{K}{1 -
  \varepsilon} \backassign M$. Thus, we obtain for any $A \in \mathcal{B} (C
  ([0, T], \mathbb{R}^d))$
  \begin{eqnarray*}
    \mathbb{P} (X^n \in A) & \leqslant & \varepsilon + \int_{\mathbb{R}^d}
    (\eta (y_0) \wedge K) \mathbb{P}_{y_0} (X^n \in A) \mathd y_0\\
    & \leqslant & \varepsilon + \int_{\mathbb{R}^d} \tilde{\eta} (y_0)
    \mathbb{P}_{y_0} (X^n \in A) \mathd y_0\\
    & = & \varepsilon +\mathbb{E}_{\tilde{\eta}} \left[ \1_{\{ X_n \in A \}}
    \mathcal{E} \left( \int_0^{\cdummy} \sqrt{2} b_2^n (r, \tilde{X}_r^n)
    \cdot \mathd B_r \right)_T \right]\\
    & \leqslant & \varepsilon +\mathbb{P}_{\tilde{\eta}} (X_n \in
    A)^{\frac{1}{2}} \mathbb{E}_{\tilde{\eta}} \left[ \mathcal{E} \left(
    \int_0^{\cdummy} \sqrt{2} b_2^n (r, \tilde{X}_r^n) \cdot \mathd B_r
    \right)_T^2 \right]^{\frac{1}{2}},
  \end{eqnarray*}
  where $\mathbb{E}_{\tilde{\eta}}$ is the expectation with initial condition
  $\tilde{\eta} \in L^{\infty}$ and $\tilde{X}^n$ is the process with drift
  $b^n_2 \equiv 0$. By the Novikov bound and since $\| \tilde{\eta}
  \|_{L^{\infty}} \leqslant M$, we obtain up to a redefinition of $M = M
  (\varepsilon, T)$ by Young's inequality for products
  \[ \mathbb{P} (X^n \in A) \leqslant \varepsilon + M\mathbb{P}_{\tilde{\eta}}
     (X_n \in A) . \]
  Now we we obtain tightness of $(X^n)_{n \in \mathbb{N}}$ in $C
  (\mathbb{R}_+, \mathbb{R}^d)$ by the same argument as in
  Theorem~\ref{tightnesstorusBesov}, based on the energy estimate under
  $\mathbb{P}_{\tilde{\eta}}$.
  
  Incompressibility and energy estimate for the limit points also follow by
  the same argument as in Theorem~\ref{tightnesstorusBesov}, relying on the
  bound $\| \tilde{\eta} \|_{L^{\infty}} \leqslant M$. The crucial point for
  showing that any limit point satisfies the martingale problem property
  \ref{martingalepropertyEucild} is the convergence of $\mathcal{L}^n f =
  (\partial_r + \mathLaplace + b^n \cdummy \nabla) f$ to $\mathcal{L} f$ in
  $L^2_T H^{- 1}$ for $f \in C^{\infty}_c (\mathbb{R}  \times \mathbb{R}^d)$.
  If $b_1 = b (A)$ with $A \in L^q_T L^p$, we estimate (using
  (\ref{identityA}) below),
  \begin{eqnarray}
    \| \nabla \cdummy ((A - A^n) \cdummy \nabla f) \|_{L^2_T H^{- 1}} &
    \lesssim & \| (A - A^n) \cdummy \nabla f \|_{L^2_T L^2} \nonumber\\
    & \lesssim & \| A - A^n \|_{L^2_T L^p} \| \nabla f \|_{L^{\infty}_T
    L^{p'}},  \label{convcompinf}
  \end{eqnarray}
  where $\frac{1}{p} + \frac{1}{p'} = \frac{1}{2}$ and we possibly use less
  time integrability than available to obtain the convergence also if $q =
  \infty$. If $b_1 \in L^q_T B^{- 1}_{p, 2}$, we bound
  \[ \| (b_1 - b_1^n) \cdummy \nabla f \|_{L^2_T H^{- 1}} \simeq \| (b_1 -
     b_1^n) \cdummy \nabla f \|_{L^2_T B^{- 1}_{2, 2}} \lesssim \| b_1 - b_1^n
     \|_{L^2_T B^{- 1}_{p, 2}} \| \nabla f \|_{L^{\infty}_T B^{1 +
     \varepsilon}_{p', 2}}, \]
  where again $\frac{1}{p} + \frac{1}{p'} = \frac{1}{2}$. To estimate the
  $b_2$-term we use that $b_2 \in L^4_T L^{2 r}$ and then
  \[ \| (b_2 - b_2^n) \cdummy \nabla f \|_{L^2_T H^{- 1}} \lesssim \| (b_2 -
     b_2^n) f \|_{L^2_T L^2} \lesssim \| b_2 - b_2^n \|_{L^4_T L^{2 r}} \| f
     \|_{L^4_T L^{\frac{2 r}{r - 1}}} . \]
  Finally, the statement on the quadratic variation can be deduced in the same
  way as in the proof of Theorem~\ref{tightnesstorusBesov}.
\end{proof}

\section{Uniqueness of energy solutions}

To obtain uniqueness, we analyze the Kolmogorov backward equation (KBE) of
(\ref{MainSDE}) given by
\begin{equation}
  \partial_t u = \mathLaplace u + b \cdummy \nabla u . \label{MainKBE}
\end{equation}
There are two ways to overcome the singularity of $b_1$ in equation
(\ref{MainKBE}): We can formally ``pull out'' the gradient from $u$ or the
divergence from $A$, resulting in
\begin{equation}
  \partial_t u = \mathLaplace u + \nabla \cdummy (b_1 u) + b_2 \cdummy \nabla
  u, \label{MainKBEa}
\end{equation}
or
\begin{equation}
  \partial_t u = \mathLaplace u + \nabla \cdummy (A \cdummy \nabla u) + b_2
  \cdummy \nabla u . \label{MainKBEb}
\end{equation}
The first case is obvious and for the second case we note that due to the
antisymmetry of $A_{i j}$ we have $\sum_{i j} A_{i j} \partial_{i j} u = 0$
and thus,
\begin{equation}
  \nabla \cdummy (A \cdummy \nabla u) = \sum_i \partial_i \left( \sum_j A_{i
  j} \partial_j u \right) = \sum_{i j} \partial_i A_{i j} \partial_j u =
  \sum_j \partial_j u \sum_i \partial_i A_{i j} = b_1 \cdummy \nabla u .
  \label{identityA}
\end{equation}
\subsection{Uniqueness under the assumptions of Theorem
\ref{WellposednessBesov}}\label{Besovcase}

Before proving uniqueness, we start with some a priori estimates.

\begin{lemma}
  \label{Laplacebound}Let $b \in L^r (M), r \in (d, \infty]$. Then, for every
  $\varepsilon > 0$ there exists $\lambda_0 \geqslant 1$ such that for all
  $\lambda > \lambda_0$,
  \[ \| (\lambda - \mathLaplace)^{- 1 / 2} (b \cdummy \nabla u) \|_{L^2 (M)}
     \lesssim \varepsilon \| b \|_{L^r (M)} \| (\lambda - \mathLaplace)^{1 /
     2} u \|_{L^2 (M)} . \]
  For $b \in L^d (M)$ with $d \geqslant 3$, we can still bound
  \[ \| (1 - \mathLaplace)^{- 1 / 2} (b \cdummy \nabla u) \|_{L^2 (M)}
     \lesssim \| b \|_{L^d (M)} \| (1 - \mathLaplace)^{1 / 2} u \|_{L^2 (M)} .
  \]
\end{lemma}

\begin{proof}
  Let $\delta \in [0, 1 / 2]$ if $d \geqslant 3$ and $\delta \in (0, 1 / 2]$
  if $d = 2$ and let $s = d \left( \frac{d}{2} + 1 - 2 \delta \right)^{- 1} =
  \frac{2 d}{d + 2 - 4 \delta} \in (1, 2]$. Using Besov embedding and the fact
  that ${\| \cdummy \|  }_{B^0_{s, 2}} \lesssim {\| \cdummy \|  }_{L^s}$ since
  $s \in (1, 2]$ (see Theorem~2.40 in {\cite{BahouriCheminDanchin11}}), we
  have
  \begin{eqnarray}
    \| (\lambda - \mathLaplace)^{- 1 / 2} (b \cdummy \nabla u) \|_{L^2} &
    \leqslant & \lambda^{- \delta / 2} \| (1 - \mathLaplace)^{- 1 / 2 +
    \delta} (b \cdummy \nabla u) \|_{L^2} \nonumber\\
    & \simeq & \lambda^{- \delta / 2} \| b \cdummy \nabla u \|_{B^{- 1 + 2
    \delta}_{2, 2}} \nonumber\\
    & \lesssim & \lambda^{- \delta / 2} \| b \cdummy \nabla u \|_{B^0_{s, 2}}
    \nonumber\\
    & \lesssim & \lambda^{- \delta / 2} \| b \cdummy \nabla u \|_{L^s}
    \nonumber\\
    & \overset{\tiny{\frac{1}{t} + \frac{1}{2} = \frac{1}{s}}}{\leqslant} &
    \lambda^{- \delta / 2} \| b \|_{L^t} \| \nabla u \|_{L^2} \nonumber\\
    & \leqslant & \lambda^{- \delta / 2} \| b \|_{L^{d / (1 - 2 \delta)}} \|
    (1 - \mathLaplace)^{1 / 2} u \|_{L^2} \nonumber\\
    & \leqslant & \lambda^{- \delta / 2} \| b \|_{L^{d / (1 - 2 \delta)}} \|
    (\lambda - \mathLaplace)^{1 / 2} u \|_{L^2} .  \label{boundforpertterm}
  \end{eqnarray}
  The claim follows by choosing $\delta (r) \in [0, 1 / 2]$ (resp. $(0, 1 /
  2]$) such that $d / (1 - 2 \delta) = r$.
\end{proof}

The following simple a priori estimate is crucial, because it is uniform in
the divergence free drift $b_1$.

\begin{lemma}[Maximum principle \& energy estimate]
  \label{aprioriPDE}Let $T > 0$, let $b_i \in C ([0, T] ; \mathcal{S} (M))$,
  $i \in \{ 1, 2 \}$, and let $u_0 \in \mathcal{S}$ and let $u$ be the unique
  classical solution to
  \[ \partial_t u = \Delta u + b \cdummy \nabla u, \qquad u (0) = u_0, \]
  where $b = b_1 + b_2$ with $\nabla \cdummy b_1 \equiv 0$. Let for $p \in [2,
  \infty]$
  \[ K = \| b_2 \|^2_{L^2_T L^p} . \]
  Then $u$ satisfies
  \begin{eqnarray*}
    \| u \|_{C_T C_b} & \leqslant & \| u_0 \|_{L^{\infty}},\\
    \| u \|_{C_T L^2}^2 & \leqslant & e^K (\| u_0 \|_{L^2}^2 + K \| u_0
    \|_{L^{\infty}}^2),\\
    \| \nabla u \|_{L^2_T L^2}^2 & \leqslant & (1 + K e^K) \| u_0 \|_{L^2}^2 +
    K (K e^K + 1) \| u_0 \|_{L^{\infty}}^2 .
  \end{eqnarray*}
\end{lemma}

\begin{proof}
  The first estimate follows directly from the maximum principle, so we focus
  on the second and third bound. The classical solution satisfies $u \in C^{1,
  \infty} ([0, T] \times M)$ with $u (t, \cdummy) \in \mathcal{S}$ for all $t
  \in [0, T]$, because $u$ and its derivatives inherit their decay at infinity
  from $u_0$. Therefore, we are allowed to use integration by parts in the
  following computation:
  \begin{eqnarray*}
    \frac{1}{2} \partial_t \| u (t) \|_{L^2}^2 & = & \int u (t) (\Delta u (t)
    + b (t) \cdummy \nabla u (t))\\
    & = & - \int | \nabla u (t) |^2 + \int b (t) \cdummy \frac{1}{2} \nabla
    (u^2 (t))\\
    & = & - \int | \nabla u (t) |^2 + \frac{1}{2} \int b_2 \cdummy \nabla
    (u^2 (t))\\
    & \leqslant & - \int | \nabla u (t) |^2 + \| b_2 \|_{L^p} \cdot
    \frac{1}{2} \| \nabla (u^2) (t) \|_{L^{\frac{p}{p - 1}}} .
  \end{eqnarray*}
  To proceed, we bound with H{\"o}lder's inequality with $\left( \frac{p}{p -
  1} \right)^{- 1} = \frac{1}{2} + \left( \frac{2 p}{p - 2} \right)^{- 1}$,
  \begin{eqnarray*}
    \frac{1}{2} \| \nabla (u^2) \|_{L^{\frac{p}{p - 1}}} & = & \| u \nabla u
    \|_{L^{\frac{p}{p - 1}}}\\
    & \leqslant & \| u \|_{L^{\frac{2 p}{p - 2}}} \| \nabla u \|_{L^2} .
  \end{eqnarray*}
  For $p = 2$ we get $\| u \|_{L^{\infty}} \| \nabla u \|_{L^2}$ on the right
  hand side. For $p > 2$ we write $\frac{2 p}{p - 2} = 2 + \frac{4}{p - 2}$
  and bound
  \[ \left( \int | u |^{2 + \frac{4}{p - 2}} \right)^{\frac{p - 2}{2 p}}
     \leqslant \left( \| u \|_{L^{\infty}}^{\frac{4}{p - 2}} \int | u |^2
     \right)^{\frac{p - 2}{2 p}} = \| u \|_{L^{\infty}}^{\frac{2}{p}} \| u
     \|_{L^2}^{1 - \frac{2}{p}} . \]
  So, finally Young's inequality for products together with the maximum
  principle yields
  \begin{eqnarray*}
    \frac{1}{2} \partial_t \| u (t) \|_{L^2}^2 & \leqslant & - \| \nabla u (t)
    \|_{L^2}^2 + \| b_2 \|_{L^p} \| u (t) \|_{L^{\infty}}^{\frac{2}{p}} \| u
    (t) \|_{L^2}^{1 - \frac{2}{p}} \| \nabla u (t) \|_{L^2}\\
    & \leqslant & - \frac{1}{2} \| \nabla u (t) \|_{L^2}^2 + \frac{1}{2} \|
    b_2 \|_{L^p}^2 \| u_0 \|_{L^{\infty}}^{\frac{4}{p}} \| u (t) \|_{L^2}^{2 -
    \frac{4}{p}} .
  \end{eqnarray*}
  For $p = 2$ resp. $p = \infty$ only $\| u \|_{L^{\infty}}$ resp. $\| u
  \|_{L^2}$ remain. Otherwise, we apply Young's inequality for products and in
  either case we end up with
  \begin{equation}
    \label{eq:energy-estimate-pr1} \partial_t \| u (t) \|_{L^2}^2 \leqslant -
    \| \nabla u (t) \|_{L^2}^2 + \| b_2 \|_{L^p}^2 (\| u_0 \|_{L^{\infty}}^2 +
    \| u (t) \|_{L^2}^2) .
  \end{equation}
  We first ignore the negative term on the right hand side and apply
  Gronwall's inequality:
  \begin{eqnarray*}
    \| u (t) \|_{L^2}^2 & \leqslant & \left( \| u_0 \|_{L^2}^2 + \| u_0
    \|_{L^{\infty}}^2 \int_0^t \| b_2 (s) \|_{L^p}^2 \mathd s \right)
    e^{\int_0^t \| b_2 (r) \|_{L^p}^2 \mathd r} .
  \end{eqnarray*}
  Taking the supremum in time, we obtain
  \[ \| u \|_{C_T L^2}^2 \leqslant e^K (\| u_0 \|_{L^2}^2 + K \| u_0
     \|_{L^{\infty}}^2) . \]
  Now we plug this back into \eqref{eq:energy-estimate-pr1}:
  \begin{eqnarray*}
    \| \nabla u \|_{L^2_T L^2}^2 & \leqslant & \int_0^T \| b_2 (t) \|_{L^p}^2
    (\| u_0 \|_{L^{\infty}}^2 + \| u (t) \|_{L^2}^2) \mathd t + \| u_0
    \|_{L^2}^2\\
    & \leqslant & K (\| u_0 \|_{L^{\infty}}^2 + e^K (\| u_0 \|_{L^2}^2 + K \|
    u_0 \|_{L^{\infty}}^2)) + \| u_0 \|_{L^2}^2\\
    & = & (1 + K e^K) \| u_0 \|_{L^2}^2 + K (K e^K + 1) \| u_0
    \|_{L^{\infty}}^2 .
  \end{eqnarray*}
  
\end{proof}

\begin{proposition}
  \label{PDEex}Let $b_1$ with $\nabla \cdummy b_1 \equiv 0$ satisfy assumption
  \ref{supercriticalbesovunique} of Theorem \ref{WellposednessBesov}, let $b_2
  \in L^{\infty}_T L^r (M)$ for some $r \in (d, \infty]$ or $r = d$ if $d
  \geqslant 3$, and let $u_0 \in \mathcal{S}$. Then, with $b = b_1 + b_2$,
  there exists $u \in C (\mathbb{R}_+ ; L^2)$ with $u (0) = u_0$, such that
  for all $T > 0$ we have $u \in L^2_T H^1 \cap L^{\infty}_T L^{\infty}$,
  $\partial_t u \in L^2_T H^{- 1}$ exists and $u$ is a weak solution in the
  distributional sense in time and space to
  \[ \partial_t u = \mathLaplace u + b \cdummy \nabla u, \]
  meaning
  \[ \langle \partial_t u, f \rangle = \langle \mathLaplace u + b \cdummy
     \nabla u, f \rangle, \qquad f \in C^{\infty}_c (\mathbb{R}_+ \times M),
  \]
  where we define
  \[ b \cdummy \nabla u \assign \nabla \cdummy (A \cdummy \nabla u) + b_2
     \cdummy \nabla u \]
  for $b_1 = b_1 (A)$ with $A \in L^{\infty}_{\tmop{loc}} L^{\infty} (M)$, and
  $b_2 \cdummy \nabla u \in L^2_{\tmop{loc}} L^{\frac{2 r}{2 + r}}$ is defined
  as the usual product, and
  \[ b \cdummy \nabla u \assign \nabla \cdummy (b u) + b_2 \cdummy \nabla u
     \in L^2_T H^{- 1}, \]
  for $b_1 \in L^2_{\tmop{loc}} L^{\infty}$
\end{proposition}

\begin{proof}
  We consider the sequence of classical solutions $(u^n)_{n \in \mathbb{N}}$
  to
  \[ \partial_t u^n = \mathLaplace u^n + b^n \cdummy \nabla u^n, \qquad u^n
     (0) = u_0 . \]
  By the Banach-Alaoglu theorem for separable Banach spaces and by Lemma
  \ref{aprioriPDE} there exists a weakly converging subsequence in $L^2_T H^1$
  for any $T > 0$, and by a diagonal sequence argument we can extract a
  subsequence such that $(u^{n_k} |_{[0, T]})_k$ converges weakly in $L^2_T
  H^1$ for any $T > 0$. Furthermore,
  \[ \int_0^T \| \partial_t u^n (t) \|_{H^{- 1}}^2 \mathd t \lesssim \int_0^T
     \| \mathLaplace u^n (t) \|_{H^{- 1}}^2 \mathd t + \int_0^T \| b^n (t)
     \cdummy \nabla u^n (t) \|^2_{H^{- 1}} \mathd t \lesssim \| u^n \|_{L^2_T
     H^1 \cap L^{\infty}_T L^{\infty}}^2, \]
  by Lemma \ref{Laplacebound} and the fact that
  \[ \| \nabla \cdummy (A^n \cdummy \nabla u^n) \|^2_{H^{- 1}} \lesssim \| A^n
     \|^2_{L^{\infty}_T L^{\infty}} \| \nabla u^n \|_{L^2}^2, \qquad \| \nabla
     \cdummy (b_1^n u^n) \|^2_{H^{- 1}} \lesssim \| b_1^n \|_{L^2}^2 \| u^n
     \|^2_{L^{\infty}_T L^{\infty}}, \]
  respectively.
  
  Hence, we can choose another subsequence such that $(\partial_t u^n |_{[0,
  T]})$ converges weakly in $L^2_T H^{- 1}$ to some $v$ and of course
  $\partial_t u = v$ in the sense of distributions in time and space. We do
  not denote the subsequence explicitly. Applying the embedding
  \[ L^2_T H^1 \cap H^1_T H^{- 1} \hookrightarrow C_T L^2, \]
  from Theorem 3 in {\textsection} 5.9.2 of Part II in {\cite{Evans98}} (which
  translates to the periodic case as well), we get $u|_{[0, T]} \in C_T L^2$.
  To see that $u \in L^{\infty} L^{\infty}$ we write for $f \in C^{\infty}_c
  (\mathbb{R}_+ \times M)$, testing in time and space,
  \[ | \langle u, f \rangle | = \lim_n | \langle u^n, f \rangle | \leqslant
     \sup_n \| u^n \|_{L^{\infty}_T L^{\infty}} \| f \|_{L^1 (\mathbb{R}_+
     \times M)}, \]
  implying that $u \in (L^1 (\mathbb{R}_+ \times M))^{\ast} = L^{\infty}
  (\mathbb{R}_+ \times M)$. For $f \in C^{\infty}_c (\mathbb{R}_+ \times M)$
  we have
  \begin{eqnarray*}
    \langle \partial_t u, f \rangle & = & \lim_n \langle \partial_t u^n, f
    \rangle\\
    & = &\lim_n \langle \mathLaplace u^n + b^n \cdummy \nabla u^n, f \rangle\\
    & = &\langle \mathLaplace u, f \rangle - \lim_n \{ \langle u^n, b_1^n
    \cdummy \nabla f - b_1 \cdummy \nabla f \rangle + \langle u^n - u, b_1
    \cdummy \nabla f \rangle + \langle u, b_1 \cdummy \nabla f \rangle \}\\
    & &- \lim_n \{ \langle u^n, \nabla \cdummy (b_2^n f) - \nabla \cdummy (b_2
    f) \rangle + \langle u^n - u, \nabla \cdummy (b_2 f) \rangle + \langle u,
    \nabla \cdummy (b_2 f) \rangle \} .
  \end{eqnarray*}
  We know from the proof of Theorems \ref{tightnesstorusBesov} and
  \ref{tightnesstorusBesovEuclid} that $b_1^n \cdummy \nabla f \rightarrow b_1
  \cdummy \nabla f$ strongly in $L^2_T H^{- 1}$ (note that (\ref{convcompinf})
  holds and vanishes in the limit also for $p = \infty$, even if this case was
  not considered in the proof), implying
  \[ | \langle u^n, b_1^n \cdummy \nabla f - b_1 \cdummy \nabla f \rangle |
     \leqslant \| u^n \|_{L^2_T H^1} \| b_1^n \cdummy \nabla f - b_1 \cdummy
     \nabla f \|_{L^2_T H^{- 1}} \rightarrow 0 . \]
  Furthermore, by weak convergence in $L^2_T H^1$, we have
  \[ \lim_n \langle u^n - u, b_1 \cdummy \nabla f \rangle = 0 . \]
  To see convergence for the $b_2$-term, we bound
  \[ | \langle u^n, \nabla \cdummy (b_2^n f) - \nabla \cdummy (b_2 f) \rangle
     | \lesssim \| u^n \|_{L^2_T H^1} \| b_2^n f - b_2 f \|_{L^2_T L^2}
     \lesssim \| (b_2^n - b_2) 1_{\tmop{supp} f} \|_{L^2_T L^2} \| f
     \|_{L^{\infty}_T L ^{\infty}}, \]
  which converges to zero. Furthermore, by the same bound, we have $\nabla
  \cdummy (b_2 f) \in L^2_T H^{- 1}$ so that
  \[ \langle u^n - u, \nabla \cdummy (b_2 f) \rangle \rightarrow 0 . \]
  We obtain
  \[ \langle \partial_t u, f \rangle = \langle \mathLaplace u + b \cdummy
     \nabla u, f \rangle . \]
\end{proof}

We are now ready to prove the uniqueness part of
Theorem~\ref{WellposednessBesov}.

\begin{proof}[Proof of Theorem \ref{WellposednessBesov}.]
  Existence of an energy solution $X$ on $M$ follows from
  Theorems~\ref{tightnesstorusBesov} and \ref{tightnesstorusBesovEuclid},
  respectively, and it remains to show uniqueness. Let $T > 0$, and let $u$
  with $u (T) = u_T \in \mathcal{S}$ be a weak solution to
  \[ \partial_t u + \mathLaplace u + b \cdummy \nabla u = 0, \qquad \text{on }
     [0, T] \times M, \]
  as in Proposition \ref{PDEex} (modulo time reversal). In particular, $u \in
  C_T L^2 \cap L^2_T H^1 \cap L^{\infty}_T L^{\infty}$, and $\partial_t u \in
  L^2_T H^{- 1}$. We extend $u$ to all times by setting $u (t) = u (0)$ for $t
  \leqslant 0$ and $u (t) = u (T)$ for $t \geqslant T$. Then we consider the
  convolution $\rho_{\kappa} \ast u \in C^{\infty} (\mathbb{R}_+ ; \mathcal{S}
  (M))$, where $\rho_{\kappa}$ is a positive mollifier in time and space. It
  easily follows from an approximation with compactly supported smooth
  functions that
  \[ M^{\rho_{\kappa} \ast u}_t = \rho_{\kappa} \ast u (t, X_t) -
     \rho_{\kappa} \ast u (0, X_0) - \int_0^t (\partial_t + \mathLaplace + b
     \cdummy \nabla) \rho_{\kappa} \ast u (r, X_r) \mathd r, \qquad t \in [0,
     T], \]
  is a martingale in the filtration generated by $X$, with quadratic variation
  \[ \langle M^{\rho_{\kappa} \ast u} \rangle_t = \int_0^t | \nabla
     \rho_{\kappa} \ast u (s, X_s) |^2 \mathd s. \]
  By the energy estimate we have
  \[ \mathbb{P} \left( \sup_{t \leqslant T} \left| \int_0^t | \nabla
     \rho_{\kappa} \ast u (s, X_s) |^2 \mathd s - \int_0^t | \nabla u (s, X_s)
     |^2 \mathd s \right| > \delta \right) \leqslant \varepsilon +
     \frac{M}{\delta} \| | \nabla \rho_{\kappa} \ast u |^2 - | \nabla u |^2
     \|_{L^1_T L^1}, \]
  and since $u \in L^2_T H^1$, we obtain that $\langle M^{\rho_{\kappa} \ast
  u} \rangle$ converges in ucp to $\int_0^{\nocomma \cdummy} | \nabla u (s,
  X_s) |^2 \mathd s$. In particular, $M^{\rho_{\kappa} \ast u}$ converges in
  distribution to a continuous local martingale $M^u$ by Theorem~VI.4.13 and
  Proposition~IX.1.17 in {\cite{Jacod2003}}. But by incompressibility and
  since $u \in C (\mathbb{R}_+ ; L^2)$ we obtain that $\rho_{\kappa} \ast u
  (t, X_t) - \rho_{\kappa} \ast u (0, X_0)$ converges in probability to $u (t,
  X_t) - u (0, X_0)$. Convergence in probability of the integral term to 0
  follows from the energy estimate if we can show that
  \begin{equation}
    (\partial_t + \mathLaplace + b \cdummy \nabla) (\rho_{\kappa} \ast u)
    \rightarrow (\partial_t + \mathLaplace + b \cdummy \nabla) u = 0, \qquad
    \tmop{in} L^2_T H^{- 1} . \label{convergenceinL2Hminus}
  \end{equation}
  The convergence $(\partial_t + \mathLaplace) \rho_{\kappa} \ast u
  \rightarrow (\partial_t + \mathLaplace) u$ is clear since $\partial_t u \in
  L^2_T H^{- 1}$ and $u \in L^2_T H^1$. Now we treat the drift term. Recall
  the definition of $b \cdummy \nabla u$ from e.g. Proposition \ref{PDEex}.
  For $b_1 \in L^2_T L^2$ we bound
  \begin{eqnarray*}
    \| b_1 \cdummy \nabla (\rho_{\kappa} \ast u - u) \|_{L^2_T H^{- 1}} & = \|
    \nabla \cdummy (b_1 \cdummy (\rho_{\kappa} \ast u - u)) \|_{L^2_T H^{-
    1}}\\
    & \lesssim \| b_1 \cdummy (\rho_{\kappa} \ast u - u) \|_{L^2 ([0, T]
    \times M)},
  \end{eqnarray*}
  which vanishes for $\kappa \rightarrow 0$ by the dominated convergence
  theorem, as $u \in L^{\infty} ([0, T] \times M)$. For $A \in L^{\infty}_T
  L^{\infty}$ we estimate
  \[ \| \nabla \cdummy (A \cdummy \nabla (\rho_{\kappa} \ast u - u)) \|_{L^2_T
     H^{- 1}} \lesssim \| A \|_{L_T^{\infty} L^{\infty}} \| \rho_{\kappa} \ast
     u - u \|_{L^2_T H^1} \rightarrow 0. \]
  By Lemma~\ref{Laplacebound}, the $b_2$-term is controlled by
  \[ \| b_2 \cdummy \nabla (\rho_{\kappa} \ast u - u) \|_{L^2_T H^{- 1}}
     \lesssim \| b_2 \|_{L^{\infty}_T L^r} \| \rho_{\kappa} \ast u - u
     \|_{L^2_T H^1} \rightarrow 0. \]
  Hence, we have (\ref{convergenceinL2Hminus}), and therefore $M^u_t = u (t,
  X_t) - u (0, X_0)$, $t \in [0, T]$, is a (continuous) local martingale.
  Since $u \in L^{\infty}_T L^{\infty}$ and since $\tmop{law} (X_t) \ll
  \tmop{Leb}$ by the incompressibility estimate, this local martingale is
  bounded and thus a true martingale. We obtain
  \[ \mathbb{E} [u_T (X_T)] =\mathbb{E} [u (0, X_0)], \]
  and this yields uniqueness of the one-dimensional distributions of $X$. The
  uniqueness of the finite-dimensional distributions and the Markov property
  then follow iteratively by a standard argument, see
  {\cite{GubinelliPerkowski20}}, Theorem~4.8 or {\cite{Ethier1986}},
  Theorem~4.4.2.
\end{proof}

\subsection{Uniqueness under the assumptions of Theorem
\ref{Wellposednesslocaldiverging}}\label{assumptionsonAcase}

In the following, $b$ is independent of time. Recall that
\[ b^n = b \ast \rho^n \]
for a mollifier $\rho^n$, so that $b^n \in C^{\infty}_b (M)$. Therefore, the
operator $\mathcal{L}^n \in L (H^1, H^{- 1})$ with $\mathcal{L}^n =
\mathLaplace + b^n \cdummy \nabla$ is well-defined.

We assume throughout this section that $b = b_1 + b_2$ with $b_i \in
\mathcal{S}' (M)$ such that $b_1 = b_1 (A)$, where $A \in L^p$ for $p \in [2,
\infty]$ and such that $b_2 \in L^r (M)$ for some $r \geqslant d$. In that
case we define the continuous operator
\[ \mathcal{L}: H^1 \rightarrow \mathcal{S}' (M), \qquad \mathcal{L}u \assign
   \mathLaplace u + \nabla \cdummy (A \cdummy \nabla u) + b_2 \nocomma \cdummy
   \nabla u. \]
One can show that $\mathcal{L}$ maps $H^1$ boundedly to some Besov space. But
for general $u \in H^1$ we do not know if $\mathcal{L}u \in L^2$. This
situation is similar to the infinite-dimensional setting in
{\cite{GubinelliPerkowski20}}, where a domain for the generator has been
constructed using a controlled ansatz. In the next proposition we identify a
domain $\mathcal{D}$ which is indeed mapped to $L^2$. Later, we address the
uniqueness of the domain under additional conditions.

We are unsure if parts of Propositions \ref{existencesemigroup} and
\ref{uniquenessabstract} might follow from general functional analytic theory,
like sesquilinear forms on Hilbert spaces. At least in the standard literature
(cf {\cite{SimonReed80}}, {\cite{Kato80}}) ``sectoriality'' of the associated
form is assumed implying in our case that $\langle u, b \cdummy \nabla u
\rangle$ for complex-valued $u$ can be controlled by $\| u \|_{H^1}^2$, which
is not the case in our setting, exactly due to supercriticality. In some
sense, the proof of proposition \ref{existencesemigroup} is the
operator-theoretic version of the proof of proposition \ref{PDEex}, where now
we solve the resolvent equation for $\mathcal{L}$ instead of the Cauchy
problem, and we do this in way which is ``uniform in the data'', via a
diagonal sequence argument, so that we obtain a domain for $\mathcal{L}$.

\begin{proposition}
  \label{existencesemigroup}Let $b_i \in \mathcal{S}' (M)$, $i \in \{ 1, 2 \}$
  with $b_1 = b_1 (A)$, where $A \in L^p$ for $p \in [2, \infty]$ and $b_2 \in
  L^r (M)$ for $r \in (d, \infty]$ or, if $d \geqslant 3$, $b_2 \in L^d (M)$
  and $\| b \|_{L^d (M)} \leqslant K_d$ for a certain $K_d > 0$ depending only
  on $d$. Let $b = b_1 + b_2$. Then there is a $\lambda_0 > 0$ such that the
  resolvent $(\lambda - \mathcal{L}^n)^{- 1} \in L (H^{- 1}, H^1)$ exists for
  all $\lambda > \lambda_0$ and for all $n \in \mathbb{N}$, and
  \begin{equation}
    \sup_n \| (\lambda - \mathcal{L}^n)^{- 1} \|_{L (H^{- 1}, H^1)} < \infty .
    \label{compactargument}
  \end{equation}
  Moreover, for any subsequence $(n_k)_{k \in \mathbb{N}} \subset \mathbb{N}$
  and $\lambda > \lambda_0$ there exists a further subsequence $(n_{k_l})_{l
  \in \mathbb{N}}$ and a strongly continuous $L^2$-semigroup $P$ with the
  following properties
  \begin{enumerateroman}
    \item The generator $(\mathcal{D}, \mathcal{L})$ of $P$ has domain
    $\mathcal{D} \subset H^1 (M)$ and is given by $\mathcal{L} u \assign
    \mathLaplace u + \nabla \cdummy (A \cdummy \nabla u) + b_2 \nocomma
    \cdummy \nabla u$, for $u \in \mathcal{D}$.
    
    \item $(\lambda -\mathcal{L})^{- 1} \in L (L^2 (M))$ exists, extends to an
    element of $L (H^{- 1}, H^1)$ and $\left( \lambda - \mathcal{L}^{n_{k_l}}
    \right)^{- 1} u^{\sharp} \rightarrow (\lambda -\mathcal{L})^{- 1}
    u^{\sharp}$ weakly in $H^1 (M)$ for any $u^{\sharp} \in H^{- 1} (M)$.
  \end{enumerateroman}
\end{proposition}

\begin{proof}
  We first show that for large enough $\lambda > 0$ the image $(\lambda -
  \mathcal{L}^n) H^1 \subset H^{- 1}$ is dense in $H^{- 1}$ for all $n \in
  \mathbb{N}$. Note that for this claim we may equip $H^{- 1}$ with the
  equivalent norm $\| (\lambda - \Delta)^{- 1 / 2} \cdummy \|_{L^2}$. Let then
  $v \in H^{- 1}$ be such that $\langle (\lambda - \Delta)^{- 1} (\lambda -
  \mathcal{L}^n) u, v \rangle_{L^2} = 0$ for all $u \in H^1$. Setting $w = u =
  (\lambda - \mathLaplace)^{- 1} v \in H^1$ and using the fact that $\nabla
  \cdummy b_1^n = 0$, we obtain
  \[ \| (\lambda - \mathLaplace)^{1 / 2} w \|_{L^2}^2 - \langle b^n_2 \cdummy
     \nabla w, w \rangle_{L^2} = \langle (\lambda - \mathLaplace) w, w
     \rangle_{L^2} - \langle b^n \cdummy \nabla w, w \rangle_{L^2} = 0 . \]
  Next, we estimate
  \[ | \langle b^n_2 \cdummy \nabla w, w \rangle_{L^2} | = \langle (\lambda -
     \mathLaplace)^{- 1 / 2} b^n_2 \cdummy \nabla w, (\lambda -
     \mathLaplace)^{1 / 2} w \rangle_{L^2} \leqslant C \| (\lambda -
     \mathLaplace)^{1 / 2} w \|_{L^2}^2, \]
  where by Lemma \ref{Laplacebound} we can take $C < 1$ if $\lambda > 0$ is
  big enough or $\| b_2 \|_{L^d}$ is small enough, and in that case we get $w
  = 0$ and thus $v = 0$. Therefore, $(\lambda - \mathcal{L}^n) H^1$ is dense
  in $H^{- 1}$.
  
  Let now $u^{\sharp} \in (\lambda - \mathcal{L}^n) H^1 \subset H^{- 1}$.
  Testing the equation $(\lambda - \mathcal{L}^n) u = u^{\sharp}$ in $L^2$
  against $u \in H^1$, we obtain by a similar argument as above
  \[ (1 - C) \| (\lambda - \mathLaplace)^{1 / 2} u \|^2_{L^2} \leqslant \|
     (\lambda - \mathLaplace)^{1 / 2} u \|_{L^2} \| (\lambda -
     \mathLaplace)^{- 1 / 2} u^{\sharp} \|_{L^2}, \]
  i.e.
  \begin{equation}
    \| u \|_{H^1} \lesssim \| u^{\sharp} \|_{H^{- 1}},
    \label{uniforminvertbound}
  \end{equation}
  This shows that $(\lambda - \mathcal{L}^n)$ is injective on $H^1$, and that
  its image is closed.
  
  Since we already saw that the image of $(\lambda - \mathcal{L}^n)$ is dense,
  it must be equal to $H^{- 1}$ and thus $(\lambda - \mathcal{L}^n)^{- 1} \in
  L (H^1, H^{- 1})$ exists with uniformly bounded operator norm. Let now
  $(e_{\ell})_{\ell \in \mathbb{N}}$ be a complete orthonormal system in $H^{-
  1}$. By the Banach-Alaoglu theorem and (\ref{compactargument}) we get a
  subsequence $n_{k_1}$ such that $(\lambda - \mathcal{L}^{n_k})^{- 1} e_1$
  converges weakly in $H^1$. By a diagonal sequence argument we then find a
  subsequence $(n_k)_{k \in \mathbb{N}}$ such that
  \begin{equation}
    \mathcal{K}e_{\ell} \assign \lim_k  (\lambda - \mathcal{L}^{n_k})^{- 1}
    e_{\ell}, \qquad \| \mathcal{K}e_{\ell} \|_{H^1} \lesssim \| e_{\ell}
    \|_{H^{- 1}} \label{preboundK}
  \end{equation}
  exists weakly in $H^1$ for all $\ell$. Using linearity and the bound in
  (\ref{preboundK}), we get a unique extension of $\mathcal{K}$ to an operator
  $\mathcal{K} \in L (H^{- 1}, H^1)$. By an $\varepsilon / 3$ argument we can
  verify that $(\lambda - \mathcal{L}^{n_k})^{- 1} u^{\sharp} \rightarrow
  \mathcal{K} u^{\sharp} \backassign u$ weakly in $H^1$ for each $u^{\sharp}
  \in H^{- 1}$. Furthermore, we can pass to the limit in the equation
  \[ \lambda u^{n_k} = \mathLaplace u^{n_k} + b^{n_k} \cdummy \nabla u^{n_k} +
     u^{\sharp}, \]
  where $u^{n_k} = (\lambda - \mathcal{L}^{n_k})^{- 1} u^{\sharp}$, in the
  sense of distributions. Indeed, we observe that $\mathcal{G} u = \nabla
  \cdummy (A \cdummy \nabla u) + b_2  \cdummy \nabla u$ is a bounded operator
  from $H^1$ to some space $B^s_{p, q}$ and hence for $f \in C^{\infty}_c (M)$
  and with $\mathcal{G}_{n_k} = b^{n_k} \cdummy \nabla u^{n_k}$,
  \begin{eqnarray*}
    \langle \mathcal{G} u - \mathcal{G}_{n_k} u^{n_k}, f \rangle & = \langle
    \mathcal{G} (u - u^{n_k}), f \rangle + \langle (\mathcal{G} -
    \mathcal{G}_{n_k}) u^{n_k}, f \rangle,
  \end{eqnarray*}
  converges to $0$, since $\langle \mathcal{G} \cdummy, f \rangle$ is a weakly
  continuous linear functional on $H^1$ and
  \begin{eqnarray*}
    | \langle (\mathcal{G} - \mathcal{G}_{n_k}) u^{n_k}, f \rangle | &
    \leqslant &| \langle \nabla \cdummy ((A - A_{n_k}) \cdummy \nabla u^{n_k}),
    f \rangle | + | \langle (b_2  - b_2^{n_k}) \cdummy \nabla u^{n_k}, f
    \rangle |\\
    & \leqslant &\| \nabla u^{n_k} \|_{L^2} (\| \nabla f \|_{L^{\infty}} \| (A
    - A_{n_k}) 1_{\tmop{supp} f} \|_{L^2} + \| f \|_{L^{\infty}} \| (b_2  -
    b_2^{n_k}) 1_{\tmop{supp} f} \|_{L^2}),
  \end{eqnarray*}
  and $\| \nabla u^{n_k} \|_{L^2} \lesssim \| u^{n_k} \|_{H^1} \lesssim 1$.
  Thus, we can apply the dominated convergence theorem and obtain
  \[ \lambda u = \mathLaplace u + \nabla \cdummy (A \cdummy \nabla u) + b_2 
     \cdummy \nabla u + u^{\sharp}, \]
  meaning
  \begin{equation}
    (\lambda - \mathcal{L}) u = u^{\sharp} . \label{surjective}
  \end{equation}
  In particular, setting
  \[ \mathcal{D} \assign \mathcal{K} L^2, \]
  we obtain an operator $(\mathcal{D}, \mathcal{L})$ on $L^2$. To show that
  this operator generates a strongly continuous semigroup on $L^2$, we apply
  the theorem of Lumer-Phillips for reflexive Banach spaces (see Corollary
  II.3.20 in {\cite{EngelNagel00}}) to $\mathcal{L}- a$, where $a \in [0,
  \lambda)$. Since $L^2$ is a reflexive Banach space and $(\lambda - a) -
  (\mathcal{L} - a) = \lambda -\mathcal{L}$ is surjective due to
  (\ref{surjective}) and $\lambda - a > 0$, it remains to identify $a \in [0,
  \lambda)$ such that $\langle u, (\mathcal{L} - a) u \rangle \leqslant 0$ for
  all $u \in \mathcal{D}$. Since any $u \in \mathcal{D}$ solves
  (\ref{surjective}) for some $u^{\sharp} \in L^2$, we obtain by weak
  convergence in $H^1$
  \begin{eqnarray*}
    \langle u, (\mathcal{L} - a) u \rangle & = & - \langle u, u^{\sharp} \rangle
    + (\lambda - a) \| u \|_{L^2}^2\\
    & =& - \lim_k \langle u^{n_k}, u^{\sharp} \rangle + (\lambda - a) \| u
    \|_{L^2}^2\\
    & =& - \lim_k \langle u^{n_k}, (\lambda - \mathcal{L}^{n_k}) u^{n_k}
    \rangle + (\lambda - a) \| u \|_{L^2}^2\\
    & =& - \lim_k \| (\lambda - \mathLaplace)^{1 / 2} u^{n_k} \|_{L^2}^2 -
    \langle u^{n_k}, b_2^{n_k} \cdummy \nabla u^{n_k} \rangle + (\lambda - a)
    \| u \|_{L^2}^2\\
    & \leqslant& - \lim_k \| (\lambda - \mathLaplace)^{1 / 2} u^{n_k}
    \|^2_{L^2} - \langle u^{n_k}, b_2^{n_k} \cdummy \nabla u^{n_k} \rangle +
    \lambda^{- 1} (\lambda - a) \| (\lambda - \mathLaplace)^{1 / 2} u
    \|^2_{L^2}\\
    & \leqslant & \liminf_k (- 1 + C + \lambda^{- 1} (\lambda - a)) \| (\lambda
    - \mathLaplace)^{1 / 2} u^{n_k} \|^2_{L^2},
  \end{eqnarray*}
  where $C < 1$ comes from Lemma \ref{Laplacebound} as before. Now the claim
  follows by choosing $a < \lambda$ close enough to $\lambda$. We obtain that
  $(\mathcal{D}, (\mathcal{L} - a))$ generates a strongly continuous semigroup
  on $L^2$ and thus the same is true for $(\mathcal{D}, \mathcal{L})$.
\end{proof}

\begin{remark}
  \label{remarkstrongconvergence}For $M =\mathbb{T}^d$ we can use the compact
  embedding $H^1 \hookrightarrow L^2$ to obtain the strong convergence
  $u^{n_k} \rightarrow u$ in $L^2$. Therefore, also
  \[ \mathcal{L}^{n_k} u^{n_k} = - u^{\sharp} + \lambda u^{n_k} \rightarrow -
     u^{\sharp} + \lambda u = \mathcal{L} u \]
  in $L^2$, and this simultaneous convergence $(u^{n_k}, \mathcal{L}^{n_k}
  u^{n_k}) \rightarrow (u, \mathcal{L}u)$ for any $u \in \mathcal{D}$ is
  equivalent to the convergence of semigroups
  \[ P^{n_k}_t f \rightarrow P_t f, \]
  in $L^2 (\mathbb{T}^d)$ for any $f \in L^2 (\mathbb{T}^d)$ and $t \geqslant
  0$. On $M =\mathbb{R}^d$ we would have to work with a suitable weighted
  $L^2$ space.
\end{remark}

Let $X$ be a continuous incompressible process (see \ref{defenergysolutions}
and \ref{defenergysolutionsEuclid} for the definition) and let $g \in L^2
(\tmop{Leb})$. Then,
\[ \int_0^t | g | (X_s) \mathd s < \infty, \]
almost surely. Indeed, for $f \in C^{\infty}_c (M)$ we have
\[ \mathbb{P} \left( \int_0^t | f | (X_s) > N \right) \leqslant \mathbb{P}
   \left( \int_0^t | f | (X_s) > N, \sup_{0 \leqslant s \leqslant t} | X_s |
   \leqslant L \right) +\mathbb{P} (\sup_{0 \leqslant s \leqslant t} | X_s | >
   L), \]
and for given $\varepsilon > 0$ we can bound the second term by
$\frac{\varepsilon}{2}$ by fixing $L (\varepsilon)$ large enough. Then we
bound with incompressibility of $X$,
\[ \mathbb{P} \left( \int_0^t | f | (X_s) > N, \sup_{0 \leqslant s \leqslant
   t} | X_s | \leqslant L \right) \lesssim_t \frac{\varepsilon}{2} + \frac{M
   (\varepsilon)}{N} \| 1_{\bar{B}_L (0)} | f | \|_{L^1} \lesssim
   \frac{\varepsilon}{2} + L (\varepsilon)^{d / 2} \frac{M (\varepsilon)}{N}
   \| f \|_{L^2}, \]
meaning that, redefining $M (\varepsilon)$, we arrive at
\[ \mathbb{P} \left( \int_0^t | f | (X_s) > N \right) \lesssim \varepsilon +
   \frac{M (\varepsilon)}{N} \| f \|_{L^2} . \]
Letting $C^{\infty}_c \ni f^n \rightarrow g$ in $L^2$ and applying the
Portmanteau theorem we arrive at the desired result. Thus the following
definition makes sense.

\begin{definition}[Martingale problem]
  \label{defmartingalesol}Let $A \in L^2_{\tmop{loc}} (M ; \mathbb{R}^{d
  \times d})$ be antisymmetric $\tmop{Leb}_M$-a.e. and let $b_2 : M
  \rightarrow \mathbb{R}^d$ be measurable. Let furthermore $\mu$ be a
  probability measure on $M$ with $\mu \ll \tmop{Leb}$ and let $X$ be a
  continuous $M$-valued process which is incompressible in probability. Let
  $\mathcal{D} \subset H^1$ be such that $\mathcal{L}u \in L^2$ for all $u \in
  \mathcal{D}$, where $\mathcal{L} u = \mathLaplace u + \nabla \cdummy (A
  \cdummy \nabla u) + b_2 \cdummy \nabla u$. We call $X$ a {\tmem{solution to
  the martingale problem for $(\mathcal{D}, \mathcal{L})$}} with initial
  distribution $\mu$, if $X (0) \sim \mu$ and for all $u \in \mathcal{D}$ the
  process
  \[ u (X_t) - u (X_0) - \int_0^t \mathcal{L} u (X_s) \mathd s, \qquad t
     \geqslant 0, \]
  is a continuous local martingale in the filtration generated by $X$.
\end{definition}

For the next lemma, we define for $\mathcal{D}$ as in
Proposition~\ref{existencesemigroup}:
\[ \mathcal{D}_{\infty} \assign \{ f \in \mathcal{D} \cap L^{\infty} :
   \mathcal{L}f \in L^2 \cap L^{\infty} \} . \]
Also, recall that an $L^2$-semigroup $(P_t)_{t \geqslant 0}$ is called
{\tmem{Markovian}} if $P_t f \geqslant 0$ for all $f \geqslant 0$ and $P_t f
\leqslant 1$ for all $f \leqslant 1$, where all inequalities have to be
interpreted in an almost everywhere sense. This implies that $\| P_t u
\|_{L^{\infty}} \leqslant \| u \|_{L^{\infty}}$ for all $u \in L^2 \cap
L^{\infty}, t \geqslant 0$, writing $\frac{\pm u}{\| u \|_{\infty}} \leqslant
1$.

\begin{lemma}
  If $(\mathcal{D}, \mathcal{L})$ is the generator of a strongly continuous,
  Markovian $L^2$-semigroup $(P_t)_{t \geqslant 0}$, then any solution to the
  martingale problem for $(\mathcal{D}, \mathcal{L})$ is a Markov process and
  its law is uniquely determined by its initial distribution.
\end{lemma}

\begin{proof}
  We first show a useful property for $\mathcal{D}_{\infty}$ which is that
  for any $u \in L^2 \cap L^{\infty}$ there exists a sequence $(u^m) $ in
  $\mathcal{D}_{\infty}$ such that $\sup_m \| u^m \|_{L^{\infty}} < \infty$
  and $u^m \rightarrow u$ in $L^2$. Indeed we define for $u \in L^2 \cap
  L^{\infty}$ the approximation $u^m = m \int_0^{1 / m} P_t u \mathrm{d} t$ and have
  that $u^m \in \mathcal{D}$ and $u^m \rightarrow u$ in $L^2$ (see e.g. the
  proof of {\cite[Theorem VII.4.6.]{Werner2018}}). Moreover
  \[ \sup_m \| u^m \|_{L^{\infty}} \leqslant \| u \|_{L^{\infty}} . \]
  It remains to check that $u^m \in \mathcal{D}_{\infty}$ which reduces to
  showing that $\mathcal{L} u^m \in L^{\infty}$ and indeed we have
  $\mathcal{L} u^m = m (P_{1 / m} u - u) \in L^{\infty}$. Let now $X$ be a
  solution to the martingale problem and let $u \in \mathcal{D}_{\infty}$.
  Then,
  \begin{equation}
    u (X_t) - u (X_0) - \int_0^t \mathcal{L} u (X_s) \mathd s,
  \end{equation}
  is not only a local but a proper continuous martingale since $u, \mathcal{L}
  u \in L^{\infty}$. Now we can perform the same time-discrete approximation
  as in Lemma A.3 of {\cite{GubinelliPerkowski20}} to conclude that for any $u
  \in \mathcal{D}_{\infty} (\mathcal{L})$
  \begin{equation}
    u (X_t) - P_t u (X_0), \label{finalmartingale}
  \end{equation}
  is a continuous martingale. Indeed, defining $u (t) = P_{T - t} u$ we can
  write with $t_k^n = t \frac{k}{n}, k \in \mathbb{N}_0, n \in \mathbb{N}$,
  \begin{eqnarray*}
    u (t, X_t) - u (0, X_0) & = &\sum_{k = 0}^{\infty} u \left( t^n_{k + 1}
    {\wedge t, X_{t^n_{k + 1} \wedge t}}  \right) - u \left( t^n_k {\wedge t,
    X_{t^n_{k + 1} \wedge t}}  \right)\\
    & &+ u \left( t^n_k {\wedge t, X_{t^n_{k + 1} \wedge t}}  \right) - u
    \left( t^n_k {\wedge t, X_{t^n_k \wedge t}}  \right)\\
    & =& \sum_{k = 0}^{\infty} \int_{t^n_k \wedge t}^{t^n_{k + 1} \wedge t}
    \partial_s u (s, X_{t^n_{k + 1} \wedge t}) \mathd s + \int_{t^n_k \wedge
    t}^{t^n_{k + 1} \wedge t} \mathcal{L} u (t^n_k \wedge t, X_s) \mathd s\\
    & &+ M^n_t \\
    & = &\sum_{k = 0}^{\infty} \int_{t^n_k \wedge t}^{t^n_{k + 1} \wedge t} v
    (t^n_k \wedge t, X_s) - v (s, X_{t^n_{k + 1} \wedge t}) \mathd s + M^n_t,
  \end{eqnarray*}
  where $M^n_t = \sum_{k = 0}^{\infty} M^{u (t^n_k \wedge t)}_{t^n_{k + 1}
  \wedge t} - M^{u (t^n_k \wedge t)}_{t^n_k \wedge t}$, ${M^f} $ is the
  continuous martingale from (\ref{martingaletobeshown}), which is proper as
  $P_t \mathcal{D}_{\infty} \subset \mathcal{D}_{\infty}$ since $\mathcal{L}
  P_t = P_t \mathcal{L}$, and we have put $v = \mathcal{L} u = P_{T - \cdummy}
  \mathcal{L} u \in L^{\infty}_T L^{\infty} \cap C_T L^2$. By the dominated
  convergence theorem it follows that the sum converges to $0$ in $L^1
  (\mathbb{P})$ if we can show for fixed $s \leqslant t$ that
  \[ v ([s]^n \wedge t, X_s) - v \left( s, X_{\left( [s]^n + \frac{1}{n}
     \right) \wedge t} \right) \rightarrow 0, \]
  in probability, where $[s]^n = \max \{ t_k^n : t_k^n \leqslant s \}
  \leqslant t$. We write
  \begin{eqnarray*}
    v ([s]^n, X_s) - v \left( s, X_{\left( [s]^n + \frac{1}{n} \right) \wedge
    t} \right) & = & v ([s]^n, X_s) - v (s, X_s)
    + v (s, X_s) - v \left( s, X_{\left( [s]^n + \frac{1}{n} \right) \wedge
    t} \right),
  \end{eqnarray*}
  and observe that the first difference vanishes in probability by
  incompressibility of $X$ and since $v \in C_T L^2$. The second difference
  vanishes choosing $f \in C (M)$ such that $\tmop{Leb} (| f - v(s) | > \delta /
  3) < \varepsilon/M_\varepsilon$ and then writing
  \begin{eqnarray*}
    &&\mathbb{P} \left( \left| v (s, X_s) - v \left( s, X_{\left( [s]^n +
    \frac{1}{n} \right) \wedge t} \right) \right| > \delta \right)\\
   & \leqslant & \mathbb{P} (| v (s, X_s) - f ( X_s) | > \delta / 3)
    +\mathbb{P} \left( \left| f ( X_s) - f \left(  X_{\left( [s]^n +
    \frac{1}{n} \right) \wedge t} \right) \right| > \delta / 3 \right)\\
    &&+\mathbb{P} \left( \left| f \left(  X_{\left( [s]^n + \frac{1}{n}
    \right) \wedge t} \right) - v \left(  X_{\left( [s]^n + \frac{1}{n}
    \right) \wedge t} \right) \right| > \delta / 3 \right)\\
    &\leqslant & 2 (\varepsilon + M_{\varepsilon} \frac{\varepsilon}{M_\varepsilon}) +\mathbb{P} \left( \left| f ( X_s) - f \left( 
    X_{\left( [s]^n + \frac{1}{n} \right) \wedge t} \right) \right| > \delta /
    3 \right),
  \end{eqnarray*}
  where we used the incompressibility of $X$ in the second inequality, we get
  that the last term vanishes as $n \rightarrow \infty$ by continuity of $f$
  and of the trajectories of $X$. Hence, the process (\ref{finalmartingale})
  is indeed a proper continuous martingale and in particular
  \[ \mathbb{E} [u (X_t)] =\mathbb{E} [P_t u (X_0)] . \]
  We can extend this identity by approximating arbitrary $u \in L^{\infty}
  \cap L^2$ using the approximation property for $\mathcal{D}_{\infty}
  (\mathcal{L})$ from the beginning of this proof and obtain uniqueness of
  one-dimensional distributions. The proof of uniqueness of finite-dimensional
  distributions and of the Markov property then follows the same iterative
  argument as in {\cite{GubinelliPerkowski20}} or the proof of Theorem
  \ref{WellposednessBesov}.
\end{proof}

\tmcolor{red}{}Similarly to {\cite{GubinelliPerkowski20}} we need that
$C_c^{\infty} (M)$ is a core for $\mathcal{L}$ w.r.t. $H^{- 1}$ in order to
link energy solutions to the martingale problem for $(\mathcal{D},
\mathcal{L})$. This is part of the next proposition.

\begin{proposition}
  \label{uniquenessabstract}Let $b_i \in \mathcal{S}' (M)$ and $b = b_1 + b_2$
  be as in Proposition~\ref{existencesemigroup}, let $\lambda > \lambda_0 > 0$
  and let the domain $\mathcal{D}$ as well as the semigroup $P$ be as in that
  proposition.
  \begin{enumerateroman}
    \item \label{uniquenessdomain}Suppose that $(\lambda -\mathcal{L})$ is
    injective on $H^1$, i.e. for any $u \in H^1$
    \begin{equation}
      (\lambda - \mathcal{L}) u = 0 \qquad \Rightarrow \qquad u = 0.
      \label{injective}
    \end{equation}
    Then
    \[ \mathcal{D} =\mathcal{D}_{\max} (\mathcal{L}) \assign \{ h \in H^1 :
       \mathcal{L} h \in L^2 \}, \]
    where $\mathcal{D}_{\max} (\mathcal{L})$ is the maximal domain of
    $\mathcal{L}$. In particular, $P$ is unique. Moreover $P$ is Markovian.
    
    \item Additionally, suppose that
    \[ (\lambda - \mathcal{L}) C_c^{\infty} (M), \]
    is dense in $H^{- 1} (M)$. Then, any energy solution $X$ to
    (\ref{MainSDE}) is a solution to the martingale problem for
    $(\mathcal{D}_{\max} (\mathcal{L}), \mathcal{L})$. In particular, $X$ is a
    Markov process and its law is uniquely determined by its initial
    distribution. 
  \end{enumerateroman}
\end{proposition}

\begin{proof}
  
  \begin{enumerateroman}
    \item Suppose $(\lambda -\mathcal{L})$ is injective. Clearly, $\mathcal{D}
    \subset \mathcal{D}_{\max} (\mathcal{L})$. To prove the opposite
    inclusion, let $h \in H^1$ be such that $\mathcal{L} h \in L^2$ and set
    $u^{\sharp} \assign (\lambda - \mathcal{L}) h$, so that $u^{\sharp} \in
    L^2$. Then
    \[ (\lambda - \mathcal{L}) ((\lambda - \mathcal{L})^{- 1} u^{\sharp} - h)
       = u^{\sharp} - u^{\sharp} = 0, \]
    and thus $(\lambda - \mathcal{L})^{- 1} u^{\sharp} = h$. Since $(\lambda -
    \mathcal{L})^{- 1} u^{\sharp} \in \mathcal{D}$, we get $h \in
    \mathcal{D}$. In order to see that $P_t$ is Markovian we use the fact that
    for any $u^{\sharp} \in L^2$ with $u^{\sharp} \leqslant 1$ almost
    everywhere,
    \[ (\lambda - \mathcal{L})^{- 1} u^{\sharp} \leqslant \frac{1}{\lambda},
    \]
    which follows from the weak convergence $(\lambda - \mathcal{L}^n)^{- 1}
    u^{\sharp} \rightarrow (\lambda - \mathcal{L})^{- 1} u^{\sharp}$ in $H^1$
    (without denoting the subsequence), noting that for any bounded $A \in
    \mathcal{B} (M)$,
    \[ \langle 1_A, (\lambda - \mathcal{L})^{- 1} u^{\sharp} \rangle = \lim_n
       \langle 1_A, (\lambda - \mathcal{L}^n)^{- 1} u^{\sharp} \rangle =
       \lim_n \langle 1_A, \int_0^{\infty} e^{- \lambda t} P_t^n u^{\sharp}
       \rangle \leqslant \frac{1}{\lambda} \tmop{Leb} (A), \]
    implying $(\lambda - \mathcal{L})^{- 1} u^{\sharp} \leqslant
    \frac{1}{\lambda}$ and we argue similarly for the implication $u^{\sharp}
    \geqslant 0 \Rightarrow (\lambda - \mathcal{L})^{- 1} u^{\sharp} \geqslant
    0$. Using the approximation
    \begin{equation}
      P_t u = \lim_m \left( \frac{m}{t} \left( \frac{m}{t} - \mathcal{L}
      \right)^{- 1} \right)^m u, \label{Hille}
    \end{equation}
    (see e.g. {\cite[IX-1.2]{Kato80}}) we get that $P$ is indeed Markovian.
    
    \item Let now $X$ be an energy solution to (\ref{MainSDE}) and suppose
    $(\lambda - \mathcal{L}) C_c^{\infty} (M)$ is dense in $H^{- 1}$. Note
    that indeed $\mathcal{L}u \in H^{- 1}$ for $u \in C_c^{\infty}$ (if $M
    =\mathbb{R}^d$) or $u \in C^{\infty}$ (if $M =\mathbb{T}^d$),
    respectively, since by H{\"o}lder's inequality
    \begin{eqnarray*}
      \| \nabla \cdummy (A \cdummy \nabla u) \|_{H^{- 1}} & \lesssim & \| A
      \cdummy \nabla u \|_{L^2} \lesssim \| A \|_{L^p} \| \nabla u \|_{{L^q}
      },\\
      \| b_2 \cdummy \nabla u \|_{H^{- 1}} & \lesssim & \| b_2 \|_{L^r} \|
      \nabla u \|_{L^s}
    \end{eqnarray*}
    for some $r, s \in [1, \infty]$. Let $u \in \mathcal{D}_{\max}
    (\mathcal{L})$. We show that
    \begin{equation}
      u (X_t) - u (X_0) - \int_0^t \mathcal{L} u (X_s) \mathd s,
      \label{martingaletobeshown}
    \end{equation}
    is a continuous local martingale in the filtration generated by $X$.
    Indeed, let $u^m \in C_c^{\infty}$ be a sequence such that $(\lambda -
    \mathcal{L}) u^m \rightarrow (\lambda - \mathcal{L}) u$ in $H^{- 1}$. It
    follows that $u^m = (\lambda - \mathcal{L})^{- 1} (\lambda - \mathcal{L})
    u^m \rightarrow (\lambda - \mathcal{L})^{- 1} (\lambda - \mathcal{L}) u =
    u$ in $H^1$. Then, $\int_0^{\cdummy} \mathcal{L} u^m (X_s) \mathd s \rightarrow
    \int_0^{\cdummy} \mathcal{L} u (X_s) \mathd s$ in ucp by the energy estimate
    and $u^m (X_t) \rightarrow u (X_t), u^m (X_0) \rightarrow u (X_0)$ in
    probability by incompressibility so that we can argue as as in the proof
    of Theorem \ref{WellposednessBesov} that \ref{martingaletobeshown} is a
    continuous local martingale.
  \end{enumerateroman}
\end{proof}

\begin{remark}
  We can use Proposition \ref{existencesemigroup} to construct (possibly
  non-unique and non-continuous) Markov processes for the limiting equation
  (\ref{MainSDE}). This involves an (almost) classical construction of a
  process from finite-dimensional distributions, described by the semigroup
  $P$, via Kolmogorovs extension theorem and an application of Proposition
  1.2.3 from {\cite{BakryGentilLedoux14}}. Proposition
  \ref{existencesemigroup} then yields the subsequential convergence of the
  finite-dimensional distributions of $X^n$ to those of such a process, see
  also Remark \ref{remarkstrongconvergence}. Then \ref{uniquenessdomain} from
  Proposition~\ref{uniquenessabstract} gives uniqueness of the limit, even if
  we cannot prove tightness and thus the existence of energy solutions is not
  known. For example, in the context of
  Theorem~\ref{Wellposednesslocaldiverging} $A \in L^2$ is allowed, and
  general critical $b_2 \in L^d$ with a small enough prefactor, instead of
  additionally requiring $b_2 \in B^0_{2 r, 1, 2}$.
\end{remark}

\begin{lemma}
  \label{lemmainjective}Let $A \in L^p (M ; \mathbb{R}^{d \times d})$ for $p
  \in [2, \infty]$ be antisymmetric $\tmop{Leb} - a.e.$ and assume that there
  exists a sequence $(g_n)_{n \in \mathbb{N}} \subset L^{\infty} (M)$ such
  that $\nabla g_n \in L^2 (M) \cap L^{\frac{2 p}{p - 2}} (M)$ with the
  following properties
  \begin{enumerateroman}
    \item $\nabla g_n \rightarrow 0$ weakly in $L^2$, $\int g_n h \rightarrow
    \int h$ for all $h \in L^1$.\label{lemmaassumptioni}
    
    \item $g_n A \in L^{\infty}, n \in \mathbb{N}.$\label{lemmaassumptionii}
    
    \item $A \cdummy \nabla g_n \rightarrow 0$ weakly in
    $L^2$.\label{lemmaassumptioniii}
  \end{enumerateroman}
  Let furthermore $b_2 \in L^r$ for some $r > d$ or, in $d \geqslant 3$, $b_2
  \in L^d$ and $\| b_2 \|_{L^d} \leqslant K_d$ for a certain constant $K_d >
  0$ depending only on $d$, and set $\mathcal{A} f = b_2 \cdummy \nabla f$ or
  $\mathcal{A} = \nabla \cdummy (b_2 f)$. Then there exists $\lambda_0 > 0$
  such that for all $\lambda > \lambda_0$, $u = 0$ is the only solution $u \in
  H^1$ to
  \begin{equation}
    \lambda u - \mathLaplace u - \nabla \cdummy (A \cdummy \nabla u) -
    \mathcal{A} u = 0 . \label{zeroeq}
  \end{equation}
\end{lemma}

\begin{proof}
  The idea is to test the equation against $u$ and to integrate by parts. Due
  to the low regularity of $u$ and $A$, this requires several approximations,
  which is where the sequence $(g_n)_{n \in \mathbb{N}}$ comes in.
  
  We first consider the case $b_2 \in L^r$, $r > d$. Let $u \in H^1$ be such
  that $\lambda u - \mathLaplace u - \nabla \cdummy (A \cdummy \nabla u) -
  \mathcal{A} u = 0$. We will truncate $u$ by applying a nonlinear function $f
  \in C^{\infty}_c (\mathbb{R}, \mathbb{R})$ with $f (0) = 0$ to $u$. Note
  that
  \[ \nabla (f \circ u) = (f' \circ u) \nabla u \in L^2 \qquad \text{and}
     \qquad f \circ u \in \bigcap_{p \in [2, \infty]} L^p . \]
  The second claim is clear because $f$ is bounded and $| f \circ u |
  \leqslant \| \nabla f \|_{\infty} | u |$ and thus $f \circ u \in L^2 \cap
  L^{\infty} \supset L^p$, and for the first claim we note that for all
  $\varphi \in C^{\infty}_c$ with $u_m = \rho^m \ast u$,
  \[ \langle \varphi, \nabla (f \circ u) \rangle = - \langle \nabla \varphi, f
     \circ u \rangle = - \lim_m \langle \nabla \varphi, f \circ u_m \rangle =
     \lim_m \langle \varphi, (f' \circ u_m) \nabla u_m \rangle = \langle
     \varphi, (f' \circ u) \nabla u \rangle, \]
  where we used the dominated convergence theorem in the second and third
  step.
  
  Consider now the function
  \[ F = \nabla u + A \cdummy \nabla u, \]
  which satisfies
  \begin{equation}
    \nabla \cdummy F = \lambda u - \mathcal{A} u \in H^{- 1} .
    \label{divergenceFinHminus}
  \end{equation}
  Indeed, $b_2 \cdummy \nabla u \in H^{- 1}$ by Lemma \ref{Laplacebound}, and
  also
  \begin{equation}
    | \langle \nabla \cdummy (b_2 u), v \rangle | = | \langle u, b_2 \cdummy
    \nabla v \rangle | \lesssim \| u \|_{H^1} \| v \|_{H^1} .
    \label{dualargument}
  \end{equation}
  Let now $(g_n)_{n \in \mathbb{N}}$ be as in the assumption of the lemma and
  fix $n \in \mathbb{N}$. We set $g_{n, \ell} = \rho^{\ell} \ast g_n$ and
  consider $C^{\infty}_c \ni u_m \rightarrow u$ in $H^1$. Note that, since $f
  (0) = 0$, also $f \circ u_m \in C^{\infty}_c$. We obtain
  \begin{eqnarray*}
    - \langle f \circ u, F \cdummy \nabla g_n \rangle & =& - \lim_m \langle f
    \circ u_m, F \cdummy \nabla g_n \rangle\\
    & =& - \lim_m \lim_{\ell} \langle f \circ u_m, F \cdummy \nabla g_{n,
    \ell} \rangle\\
    & =& \lim_m \lim_{\ell} \langle \nabla \cdummy ((f \circ u_m) F), g_{n,
    \ell} \rangle\\
    & =& \lim_m \lim_{\ell} \langle \nabla \cdummy F, (f \circ u_m) g_{n,
    \ell} \rangle + \langle \nabla (f \circ u_m) \cdummy F, g_{n, \ell}
    \rangle\\
    & =& \lim_m \langle \nabla \cdummy F, (f \circ u_m) g_n \rangle + \langle
    \nabla (f \circ u_m) \cdummy F, g_n \rangle\\
    & =& \langle \nabla \cdummy F, (f \circ u) g_n \rangle_{H^{- 1}, H^1} +
    \langle \nabla (f \circ u), F g_n \rangle .
  \end{eqnarray*}
  In the first equality we used that $F \nabla g_n \in L^1$ and the dominated
  convergence theorem. Since $(f \circ u_m) F \in L^{\frac{2 p}{p + 2}}$, the
  same reasoning yields the second equality. In the third and fourth equality
  we integrated by parts and applied the Leibniz rule for products of smooth
  functions and distributions. In the fifth and sixth equality we are allowed
  to pass to the limit in the first term due to convergence in $H^1 = W^{1,
  2}$ in the second variable, e.g. writing
  \begin{equation}
    \nabla ((f \circ u) g_n) = \nabla (f \circ u) g_n + (f \circ u) \nabla g_n
    \in L^2, \label{productforHone}
  \end{equation}
  and the fact that $\nabla \cdummy F \in H^{- 1}$ by
  (\ref{divergenceFinHminus}). In the sixth equality we used that $F g_n \in
  L^2$ by \ref{lemmaassumptionii}. In total, we obtain
  \begin{eqnarray}
    - \langle (f \circ u), F \cdummy \nabla g_n \rangle & = &\langle \nabla
    \cdummy F, (f \circ u) g_n \rangle + \langle \nabla (f \circ u), F g_n
    \rangle \nonumber\\
    & = &\langle (\lambda u - \mathcal{A} u), (f \circ u) g_n \rangle +
    \langle (f' \circ u) \nabla u, (\nabla u + A \nabla u) g_n \rangle
    \nonumber\\
    & = &\langle (\lambda u - \mathcal{A} u), (f \circ u) g_n \rangle +
    \langle (f' \circ u) | \nabla u |^2, g_n \rangle,  \label{eqapproxF}
  \end{eqnarray}
  applying the pointwise anti-symmetry of $A$ in the third equality. For the
  left hand side we get furthermore
  \[ \lim_n \langle (f \circ u) F, \nabla g_n \rangle = \lim_n \int (f \circ
     u) \nabla u^T \cdummy \nabla g_n + \int (f \circ u) \nabla u^T \cdummy A
     \cdummy \nabla g_n = 0, \]
  by \ref{lemmaassumptioni} and \ref{lemmaassumptioniii}. Plugging this into
  (\ref{eqapproxF}), we obtain

  \[ \langle (\lambda u - \mathcal{A} u), (f \circ u) \rangle_{H^{- 1}, H^1} +
     \int (f' \circ u) | \nabla u |^2 = 0, \]
  where we applied \ref{lemmaassumptioni} for taking the limit $n \rightarrow
  \infty$ in the second term and for the first term we used
  \ref{lemmaassumptioni} and (\ref{productforHone}) to obtain the weak
  convergence $(f \circ u) g_n \rightarrow f \circ u$ in $H^1$. Now, let
  $(f_n) \subset C^{\infty}_c (\mathbb{R}, \mathbb{R})$ be such that $\| f'_n
  \|_{\infty} \leqslant 1$ and $f_n (x) = x$ for $| x | < n$. In particular
  $f_n' (x) = 1$ for $| x | < n$ and $\lim_n (f'_n \circ u) = 1$, pointwise.
  We obtain
  \[ \int | \nabla u |^2 + \lambda | u |^2 \mathd x - \langle u, \mathcal{A} u
     \rangle_{H^1, H^{- 1}} = \lim_{n \rightarrow \infty} \int_{\mathbb{T}^2}
     (\nabla u)^2 (f'_n \circ u) + (f_n \circ u) (\lambda u - \mathcal{A} u)
     \mathd x = 0, \]
  where we have used that $f_n \circ u \rightarrow u $ in $H^1$. Hence,
  \[ 0 = \| (\lambda - \mathLaplace)^{1 / 2} u \|_{L^2}^2 - \langle u,
     \mathcal{A} u \rangle_{H^1, H^{- 1}}, \]
  where
  \begin{equation}
    \langle u, \mathcal{A} u \rangle_{H^1, H^{- 1}} \lesssim \delta (\lambda)
    \| b_2 \|_{L^r} \| (\lambda - \mathLaplace)^{1 / 2} u \|_{L^2}^2,
    \label{laplaceboundpropositoin}
  \end{equation}
  again by Lemma \ref{Laplacebound} (and the same reasoning as in
  (\ref{dualargument})), implying $u = 0$ if $\lambda > 0$ is large enough.
  
  For $b_2 \in L^d$, $d \geqslant 3$, the argument follows along the same
  lines, except that we only get
  \[ \langle u, \mathcal{A} u \rangle_{H^1, H^{- 1}} \lesssim \| b_2 \|_{L^r}
     \| (1 - \mathLaplace)^{1 / 2} u \|_{L^2}^2, \]
  instead of (\ref{laplaceboundpropositoin}), so that we can only conclude $u
  = 0$ if $\| b_2 \|_{L^d}$ is small enough, i.e. smaller than a certain
  constant which only depends on $d$.
\end{proof}

\begin{remark}
  In a late stage of this work  we noticed that
  diffusion operators with distributional divergence-free drift, given by the
  divergence of some antisymmetric matrix field $A$ of low regularity have
  recently gained considerable interest in the PDE community, see e.g.
  \cite{SereginSilvestreSverakZlatos2010, MazjaVerbitsky06, QianXi18} and the foundational work
  {\cite{Osada1987}}. To the best of our knowledge all of these works do not
  go beyond the critical regime of $A \in \tmop{BMO}$, i.e. $A$ is of bounded
  mean oscillation, which implies $A \in L^p$ for all $p < \infty$ and thus is far from the regularities we allow in Lemma~\ref{lemmainjective}.
  
  The closest argument to our proof we found
  in the literature is {\cite{Oelschlaeger88}}, where the author shows a CLT
  for a non-singular SDE driven by a random, divergence-free, critically
  correlated field. In the proof the author shows uniqueness of solutions
  to the resolvent equation in some abstract Hilbert space, using a similar approximation as in Lemma~\ref{lemmainjective}. Specifically, he applies some bounded $f$ to the solution and
  multiplies with functions $g_n$ which remove the unbounded parts of $A$.
  However he can choose a ``naive'' cutoff $g_n (\omega) = 1_{\{ (1 + |
  \mathbf{x} |) \| A \|_{\infty} \lesssim n \}} (\omega)$, due to the specific
  setting, see pages 1100-1105 in {\cite{Oelschlaeger88}}.
\end{remark}

With help of the previous two results, we are ready to prove Theorem
\ref{Wellposednesslocaldiverging}.

\begin{proof}[Proof of Theorem \ref{Wellposednesslocaldiverging}.]
  Theorem \ref{tightnesstorusBesov} yields existence of an energy solution and
  the convergence statement while Propositions \ref{existencesemigroup} and
  \ref{uniquenessabstract} together with Lemma \ref{lemmainjective} yield
  uniqueness once we show that $(\lambda - \mathcal{L}) C_c^{\infty}$ is dense
  in $H^{- 1}$ for $\lambda > 0$ large enough (or $\| b_2 \|_{L^d}$ small
  enough). Let $v \in H^{- 1}$ be such that
  \[ \langle (\lambda - \mathcal{L}) u, v \rangle_{H^{- 1}} = 0, \qquad
     \forall u \in C_c^{\infty} . \]
  Then, writing $w = (1 - \mathLaplace)^{- 1} v \in H^1$, we have for all $u
  \in C_c^{\infty}$,
  \begin{eqnarray*}
    \langle (\lambda - \mathcal{L}) u, v \rangle_{H^{- 1}} & =& \langle \lambda
    u - \mathLaplace u - \nabla \cdummy (A \cdummy \nabla u) - b_2  \cdummy
    \nabla u, (1 - \mathLaplace)^{- 1} v \rangle\\
    & =& \langle u, \lambda w - \mathLaplace w + \nabla \cdummy (A \cdummy
    \nabla w) + \nabla \cdummy (b_2  w) \rangle,
  \end{eqnarray*}
  where we are allowed to integrate by parts because $A \cdummy \nabla w \in
  L^1$ and $u \in C^{\infty}_c$. This means that $\lambda w - \mathLaplace w +
  \nabla \cdummy (A \cdummy \nabla w) + \nabla \cdummy (b_2  w) = 0$, which
  implies $w = 0$ by Lemma \ref{lemmainjective}.
\end{proof}

Finally, we prove Lemma~\ref{uniquenesscompactset} which gives a sufficient
condition for the existence of $(g_n)_{n \in \mathbb{N}}$ as in
Theorem~\ref{Wellposednesslocaldiverging}.

\begin{proof}[Proof of Lemma \ref{uniquenesscompactset}.]
  Let $K$ be a compact set such that for $B^{\varepsilon} = \{ x : d (x, K)
  \leqslant \varepsilon \}$:
  \[ \sup_{\varepsilon \in (0, 1)} \varepsilon^{- 2} \tmop{Leb}
     (B^{\varepsilon}), \qquad \sup_{\varepsilon \in (0, 1)} \varepsilon^{- 2}
     \int_{B^{\varepsilon}} | A |^2 < \infty, \qquad A \1_{B_{\varepsilon}^c}
     \in L^{\infty} (M) \text{ for all } \varepsilon \in (0, 1) . \]
  For $\varepsilon \in (0, 1)$, let $g_{\varepsilon} \in C^{\infty} (M,
  \mathbb{R})$ be such that $\sup_{\varepsilon > 0} \| g_{\varepsilon}
  \|_{\infty} < \infty$ and for some $C > 0$,
  \begin{equation}
    g_{\varepsilon} (x) = \left\{\begin{array}{ll}
      1, & d (x, K) > \varepsilon,\\
      0, & d (x, K) < \varepsilon / 2,
    \end{array}\right. \qquad \text{and} \qquad \| \nabla g_{\varepsilon}
    \|_{\infty} \leqslant C \varepsilon^{- 1} . \label{gdef}
  \end{equation}
  The existence of such $(g_{\varepsilon})_{\varepsilon \in (0, 1)}$ is shown
  in Lemma~$\ref{existenceg}$ in the appendix. Given $(g_{\varepsilon})$, we
  verify the conditions of Lemma \ref{lemmainjective}. Clearly,
  $g_{\varepsilon} \in L^{\infty}$ and $\nabla g_{\varepsilon} \in L^2 \cap
  L^{\infty}$ and also $A g_{\varepsilon} \in L^{\infty}$ (so, Condition ii.
  in Lemma~\ref{lemmainjective} is satisfied) as well as $A \cdummy \nabla
  g_{\varepsilon} \in L^2$. Note that $\tmop{Leb} (B^{\varepsilon}) \lesssim
  \varepsilon^2$ implies $\tmop{Leb} (K) = 0$. For $h \in L^2 (M ;
  \mathbb{R}^d)$ we have
  \begin{eqnarray*}
    \int \nabla g_{\varepsilon} (x) \cdummy h (x) \mathd x & \leqslant & \left(
    \int_{B^{\varepsilon}} | \nabla g_{\varepsilon} (x) |^2 \mathd x
    \right)^{1 / 2} \left( \int_{B^{\varepsilon}} | h |^2 (x) \mathd x
    \right)^{1 / 2}\\
    & \lesssim & (\sup_{\delta} \delta^{- 2} \tmop{Leb} (B^{\delta}))^{1 / 2}
    \left( \int_{B^{\varepsilon}} | h |^2 (x) \mathd x \right)^{1 / 2}\\
    && \rightarrow 0,
  \end{eqnarray*}
  as $\varepsilon \rightarrow 0$, so that $\nabla g_{\varepsilon} \rightarrow
  0$ weakly in $L^2$. Similarly, we have for $h \in L^1$ by the dominated
  convergence theorem
  \[ \lim_{\varepsilon \rightarrow 0} \int g_{\varepsilon} h = \int h, \]
  i.e. $g_{\varepsilon} \rightarrow 1$ in the weak-$\ast$ topology in
  $L^{\infty}$. This proves Condition~i. in Lemma~\ref{lemmainjective}. To see
  Condition iii., let $h \in L^2 (\mathbb{M}; \mathbb{R}^d)$ and note that
  \[ \left| \int h^T \cdummy A \nabla g_{\varepsilon} \right|^2 = \left|
     \int_{B^{\varepsilon}} h^T \cdummy A \nabla g_{\varepsilon} \right|^2
     \lesssim \int_{B^{\varepsilon}} | A \nabla g_{\varepsilon} |^2  \| h
     1_{B^{\varepsilon}} \|_{L^2}^2 \lesssim \sup_{\delta} \delta^{- 2}
     \int_{B^{\delta}} | A |^2  \| h 1_{B^{\varepsilon}} \|_{L^2}^2
     \rightarrow 0 \]
  by the dominated convergence theorem.
\end{proof}

\appendix\section{Auxiliary results}\label{appendixA}

For $p \in [1, \infty]$ we define the space
\[ B^0_{p, 1, 2} = \left\{ u \in B^0_{p, 1} : \| u \|_{B^0_{p, 1, 2}}^2 =
   \sum_j \| \Pi_{\geqslant j} u \|_{B^0_{p, 1}}^2 < \infty \right\}, \]
where $\Pi_{\geqslant j} u = \sum_{i \geqslant j} \mathLaplace_j u$. One can
check that $(B^0_{p, 1, 2}, \| \cdummy \|_{B^0_{p, 1, 2}})$ is a Banach space
and
\[ B^0_{p, 1, 2} \longhookrightarrow B^0_{p, 1} \longhookrightarrow L^p . \]
\begin{lemma}
  \label{estimateforsquarenorm}For $p \in [1, \infty]$, let $b \in L^4_T
  B^0_{2 p, 1, 2}$. Then, $| b |^2 \in L^2_T B^0_{p, 1}$ and
  \[ \| | b |^2 \|_{L^2_T B^0_{p, 1}} \lesssim \| b \|^2_{L^4_T B^0_{2 p, 1,
     2}} . \]
\end{lemma}

\begin{proof}
  We estimate with the paraproduct $\olessthan$ and the resonant product
  $\odot$
  \begin{eqnarray*}
    \| | b |^2 \|_{B^0_{p, 1}} & \leqslant & 2 \| b \olessthan b \|_{B^0_{p, 1}}
    + \| b \odot b \|_{B^0_{p, 1}}\\
    & \lesssim & \| b \|_{L^{2 p}} \| b \|_{B^0_{2 p, 1}} + \| b \odot b
    \|_{B^0_{p, 1}}\\
    & \lesssim & \| b \|_{B^0_{2 p, 1, 2}}^2 + \| b \odot b \|_{B^0_{p, 1}},
  \end{eqnarray*}
  where we used Theorem 27.5 in {\cite{vanZuijlen22}} in the second step.
  Furthermore,
  \[ \mathLaplace_j (b \odot b) = \sum_{| i - i' | \leqslant 1} \mathLaplace_j
     (\mathLaplace_i b \mathLaplace_{i'} b) = \mathLaplace_j (\Pi_{\geqslant j
     - K} b \odot \Pi_{\geqslant j - K} b), \]
  for a certain $K \in \mathbb{N}$ which depends only on the annulus
  $\mathcal{A}$. Then, by Theorem 27.10 in {\cite{vanZuijlen22}},
  \[ \| \mathLaplace_j (b \odot b) \|_{L^p} \leqslant \| \Pi_{\geqslant j - K}
     b \odot \Pi_{\geqslant j - K} b \|_{B^0_{p, \infty}} \lesssim \|
     \Pi_{\geqslant j - K} b \|_{B^0_{2 p, 1}}^2, \]
  meaning that
  \[ \| b \odot b \|_{B^0_{p, 1}} \lesssim \sum_j \| \Pi_{\geqslant j} b
     \|_{B^0_{2 p, 1}}^2 = \| b \|_{B^0_{2 p, 1, 2}}^2 . \]
  Now the claim follows by integrating in time.
\end{proof}

\begin{lemma}
  \label{uniformboundinapproxnew}Let $b \in L^q_T B^0_{p, 1, 2}$ with $p, q
  \in [1, \infty]$. Then,
  \[ \sup_n \| \rho^n \ast b \|_{L^q_T B^0_{p, 1, 2}} < \infty, \]
  and for $p, q < \infty$ we have $\rho^n \ast b \rightarrow b$ in $L^q_T
  B^0_{p, 1, 2}$.
\end{lemma}

\begin{proof}
  We set $b^n = \rho^n \ast b$, where we recall that $\rho^n = n^{d + 1} \rho
  (n \cdummy)$ and $\rho (t, x) = \rho_t (t) \rho_x (x)$, where $\rho_t,
  \rho_x$ are positive mollifiers in space and time, and we extend $b (t, x) =
  0$ for $t \in [0, T]^c$. Then
  \[ \rho^n \ast b = \rho^n_t \ast_t (\rho^n_x \ast_x b) = \rho^n_t \ast_t
     (\rho^n_x \ast_x b - b) + (\rho^n_t \ast_t b - b) + b. \]
  Applying Minkowski's inequality we obtain
  \[ \| \rho^n_t \ast_t (\rho^n_x \ast_x b - b) (s) \|_{B^0_{p, 1, 2}}
     \leqslant \int \rho^n_t (r) \| (\rho^n_x \ast_x b - b) (s - r)
     \|_{B^0_{p, 1, 2}} \mathd r = \rho^n_t \ast \| \rho^n_x \ast_x b - b
     \|_{B^0_{p, 1, 2}} (s) . \]
  Thus, since we extended $b$ by $0$ outside of $[0, T]$,
  \[ \| \rho^n_t \ast_t (\rho^n_x \ast_x b - b) \|_{L^q_T B^0_{p, 1, 2}}
     \leqslant \| \rho^n_x \ast_x b - b \|_{L^q_T B^0_{p, 1, 2}} . \]
  It follows from Young's inequality that the right hand side is uniformly
  bounded in $n$. Moreover, writing out the definition of the $B^0_{p, 1, 2}$
  norm, we see that for $p, q < \infty$ the right hand side converges to zero
  by convergence of mollifications in $L^p$-spaces, Young's inequality, and
  the dominated convergence theorem. For the next term,
  \[ \rho^n_t \ast_t b (s) - b (s) = \int \rho^n_t (s - r) (b (r) - b (s))
     \mathd r, \]
  we obtain again with Minkowski's inequality,
  \begin{equation}
    \| \rho^n_t \ast_t b - b \|_{B^0_{p, 1, 2}} (s) \leqslant \int \rho^n_t (s
    - r) \| b (r) - b (s) \|_{B^0_{p, 1, 2}} \leqslant \rho^n_t \ast_t \| b
    \|_{B^0_{p, 1, 2}} (s) + \| b \|_{B^0_{p, 1, 2}} (s), \label{A2xest}
  \end{equation}
  and thus
  \[ \| \rho^n_t \ast_t b - b \|_{L^q_T B^0_{p, 1, 2}} \lesssim \| b \|_{L^q_T
     B^0_{p, 1, 2}} . \]
  The convergence to zero of the left hand side in (\ref{A2xest}) for $p, q <
  \infty$ follows from an approximation of $b$ with elements in $C_T B^0_{p,
  1, 2}$.
\end{proof}

\begin{lemma}
  \label{existenceg}The sequence $g_{\varepsilon}$ from Proposition
  \ref{uniquenesscompactset} exists.
\end{lemma}

\begin{proof}
  Let $g \in C^{\infty} (\mathbb{R}_+ ; \mathbb{R}_+)$ be such that
  \[ g (x) = \left\{\begin{array}{ll}
       0, & x \in \left[ 0, \frac{1}{2} + \frac{1}{8} \right],\\
       1, & x \in \left[ 1 - \frac{1}{8}, \infty \right) .
     \end{array}\right. \]
  For $\delta = \delta (\varepsilon) > 0$ and a positive mollifier $\rho :
  \mathbb{R}^d \rightarrow \mathbb{R}$ we set
  \[ g_{\varepsilon, \delta} (x) = g (\varepsilon^{- 1} \rho_{\delta} \ast d
     (\cdummy, K) (x)), \]
  where $\rho_{\delta} = \delta^{- d} \rho (\delta^{- 1} \cdummy)$. Now note
  that
  \[ | \nabla \rho_{\delta} \ast d (\cdummy, K) (x) | \leqslant 1 . \]
  Indeed, it follows from the triangle inequality that
  \[ | d (x, K) - d (y, K) | \leqslant | x - y |, \qquad x, y \in M, \]
  and therefore
  \[ | \rho_{\delta} \ast d (\cdummy, K) (x) - \rho_{\delta} \ast d (\cdummy,
     K) (y) | \leqslant \int \mathd z \rho_{\delta} (z) | d (x - z, K) - d (y
     - z, K) | \leqslant | x - y | . \]
  This means that
  \[ | \nabla g_{\varepsilon, \delta} (x) | = \varepsilon^{- 1} | g'
     (\varepsilon^{- 1} \rho_{\delta} \ast d (\cdummy, K) (x)) \nabla
     \rho_{\delta} \ast d (\cdummy, K) (x) | \lesssim \varepsilon^{- 1}, \]
  uniformly in $\delta$. Now we can choose $\delta (\varepsilon)$ as we want.
  Since $\rho_{\delta} \ast d (\cdummy, K)$ converges to $d (\cdummy, K)$
  uniformly by Lipschitz continuity of $d (\cdummy, K)$, there exists $\delta
  (\varepsilon) > 0$ such that
  \[ \| \rho_{\delta (\varepsilon)} \ast d (\cdummy, K) - d (\cdummy, K)
     \|_{\infty} < \frac{\varepsilon}{8} . \]
  For this $\delta (\varepsilon)$, we set
  \[ g_{\varepsilon} = g_{\varepsilon, \delta (\varepsilon)} . \]
  Then for $x$ such that $d (x, K) < \varepsilon / 2$ it holds
  \[ | \rho_{\delta (\varepsilon)} \ast d (\cdummy, K) (x) | \leqslant |
     \rho_{\delta (\varepsilon)} \ast d (\cdummy, K) (x) - d (x, K) | + d (x,
     K) < \varepsilon \left( \frac{1}{2} + \frac{1}{8} \right), \]
  i.e. $g_{\varepsilon} (x) = 0$. Similarly, we obtain $g_{\varepsilon} (x) =
  1$ for $d (x, K) > \varepsilon$.
\end{proof}

\begin{example}[Counterexample to Morrey regularity]
  \label{counterexampleMorrey}The second assumption in (\ref{divcond}) can be
  interpreted as a {\tmem{local}} form of Morrey regularity on the set $K$.
  The Morrey space $L^{2, 2}$ (see {\cite{ChiarenzaFrasca1990}} for a
  definition) is the set of all $f \in L^2_{\tmop{loc}} (\mathbb{R}^d)$ such
  that
  \[ \sup_{x \in \mathbb{R}^d, \varepsilon > 0} \varepsilon^{- 2}
     \int_{B_{\varepsilon} (x)} | f |^2 < \infty . \]
  Here we construct an $A \in L^2$ which fulfills the assumptions
  (\ref{divcond}) and (\ref{condfinite}) but is not in $L^{2, 2}$ (hence the
  term ``local''). We consider $A = | A | B$, where $B$ is an arbitrary
  antisymmetric matrix independent of $x$, and
  \[ | A | = \sum_n \alpha_n \sqrt{\rho^{\varepsilon_n}} (\cdummy - 2^{- n}
     v), \]
  where $\rho$ is a compactly supported, positive mollifier with $\tmop{supp}
  \rho \subset B_1 (0)$, $\rho^{\varepsilon} = \varepsilon^{- d} \rho
  (\varepsilon^{- 1} \cdummy)$, $v \in \mathbb{R}^d, | v | = 1$ is some
  vector, and $\sum_n | \alpha_n |^2 < \infty$. We take $\varepsilon_n
  \leqslant 2^{- n - 3}$ which means that the supports of the
  $\rho_{\varepsilon_n} (\cdummy - 2^{- n} v)$ do not intersect so that indeed
  $A \in L^2$. Note also that $A$ is compactly supported. Furthermore, we set
  $K = \{ 0 \}$. Now,
  \[ \int_{B_{\varepsilon} (0)} | A |^2 \leqslant \sum_{n \geqslant \left( -
     \ln \frac{8}{7} - \ln \varepsilon \right) / \ln 2 } | \alpha_n |^2, \]
  and for $\alpha_n$ converging sufficiently fast to zero we obtain
  $\int_{B_{\varepsilon} (0)} | A |^2 \lesssim \varepsilon^2$, meaning that
  $A$ fulfills (\ref{divcond}) and (\ref{condfinite}). On the other hand
  \[ \sup_{\delta} \delta^{- 2} \int_{B_{\delta} (2^{- n} v)} | A |^2
     \geqslant \varepsilon_n^{- 2} | \alpha_n |^2 \int_{B_{\varepsilon_n} (0)}
     \rho_{\varepsilon_n} = \varepsilon_n^{- 2} | \alpha_n |^2, \]
  and by letting $\varepsilon_n$ converge to zero faster than $\alpha_n$ we
  can make the right hand side unbounded in $n$. The same argument also shows
  that for any $\lambda > 0$ we can find $A$ satisfying (\ref{divcond}) and
  (\ref{condfinite}) but such that $A \nin L^{\lambda, 2}$.
\end{example}

\begin{lemma}[Regularity of the log-regularized Gaussian free field]\label{lem:GFF-reg}
Let $\xi$ be a massless Gaussian free field on $\mathbb T^2$ and let $\alpha>0$. Then almost surely
\[
	\log (1 - \mathLaplace)^{- \alpha} \xi \in B^0_{\infty, \infty},
\]
and consequently we have for all $\alpha>1$
\[
	\log (1 - \mathLaplace)^{- \alpha} \xi \in B^0_{\infty, 1}.
\]
\end{lemma}

\begin{proof}
	 In a first step we show that
  $\log (1 - \mathLaplace)^{- \alpha} \xi \in B^0_{\infty, \infty}$ whenever
  $\alpha > 0$. By the Borel-Cantelli lemma, it suffices to show that for some
  $C > 0$,
  \begin{equation}
    \sum_j \mathbb{P} (\| W_j \|_{L^{\infty}} > C) < \infty,
    \label{sumforborelcantelli}
  \end{equation}
  where (as in {\cite{Veraar11}})
  \[ W_j = \sum_{k \in \mathbb{Z}^2 \setminus \{ 0 \}} \varphi_j (k) (1 + 4
     \pi^2 | k |^2)^{- 1 / 2} \log (1 + 4 \pi^2 | k |^2)^{- \alpha} \gamma_k
     e_k, \]
  $(\varphi_j)_{j \geqslant - 1}$ is the dyadic partition of unity used to
  define our Littlewood-Paley blocks, $(e_k)_{k \in \mathbb{Z}^2}$ are the
  Fourier monomials, and $(\gamma_k)_{k \in \mathbb{Z}^2}$ are i.i.d. standard
  (complex) Gaussians. We estimate, using $\sum_{i = 1}^n | x_i | \leqslant
  \sqrt{n} \left( \sum_{i = 1}^n | x_i |^2 \right)^{1 / 2}$, for an annulus
  $\mathcal{A}$ such that $\tmop{supp} (\varphi_j) \subset 2^j \mathcal{A}$
  for $j \geqslant 0$, and for a constant $\delta > 0$ whose value can change
  in every line
  \begin{eqnarray*}
    \mathbb{P} (\| W_j \|_{L^{\infty}} > C) & \leqslant & \mathbb{P} \left(
    \sum_{k \in 2^j \mathcal{A}} \varphi_j (k) (1 + 4 \pi^2 | k |^2)^{- 1 / 2}
    \log (1 + 4 \pi^2 | k |^2)^{- \alpha} | \gamma_k | > C \right)\\
    & \leqslant & \mathbb{P} \left( \sum_{k \in 2^j \mathcal{A}} (1 + 4 \pi^2 |
    k |^2)^{- 1} \log (1 + 4 \pi^2 | k |^2)^{- 2 \alpha} | \gamma_k |^2 >
    \delta 2^{- 2 j} C^2 \right)\\
    & \leqslant & \mathbb{P} \left( \sum_{k \in 2^j \mathcal{A}} | \gamma_k |^2
    > \delta j^{2 \alpha} C^2 \right) .
  \end{eqnarray*}
  Let $| \mathcal{A} |$ be the number of integer points in $\mathcal{A}$.
  Since $\sum_{k \in 2^j \mathcal{A}} | \gamma_k |^2 \sim \chi^2 (|
  \mathcal{A} | 2^{2 j})$ has a chi-square distribution, we can use the tail
  bound bound $\mathbb{P} (Z \geqslant t 2 n) \leq e^{- t n / 10}$ from
  {\cite{dkoutsou2018}} for $Z \sim \chi^2 (n)$ with $t = \delta j^{2 \alpha}
  C^2 | \mathcal{A} |^{- 1} 2^{- 2 j - 1}$ and $n = | \mathcal{A} | 2^{2 j}$
  and we obtain
  \[ \mathbb{P} (\| W_j \|_{L^{\infty}} > C) \leqslant e^{- \delta j^{2
     \alpha} C^2 / 20} . \]
  Hence the sum (\ref{sumforborelcantelli}) is finite and we conclude that
  $\log (1 - \mathLaplace)^{- \alpha} \xi \in B^0_{\infty, \infty}$ for any
  $\alpha > 0$. Finally, estimating for $\beta > \frac{1}{r}$
  \begin{eqnarray*}
    \| u \|_{B^0_{\infty, r}}^r & = & \sum_j \| \mathLaplace_j u
    \|_{L^{\infty}}^r\\
    & \leqslant & \| \mathLaplace_{- 1} u \|_{L^{\infty}}^r + \|
    \mathLaplace_0 u \|_{L^{\infty}}^r + \sup_{j \geqslant 1} j^{\beta r} \|
    \mathLaplace_j u \|_{L^{\infty}}^r \sum_{j \geqslant 1} j^{- r \beta} \\
    & \lesssim & \| \log (1 - \mathLaplace)^{\beta} u \|_{B^0_{\infty,
    \infty}}^r,
  \end{eqnarray*}
  we conclude that $\log (1 - \mathLaplace)^{- \alpha} \xi \in B^0_{\infty,
  r}$ for any $r > \frac{1}{\alpha}$, $r \geqslant 1$.	
\end{proof}

\paragraph{Acknowledgements}We would like to thank Xiaohao Ji, Ana Djurdjevac
and Immanuel Zachhuber for fruitful and continued discussions. Further
gratitude is owed to the organizers of the Summer School on PDEs and
Randomness 2023 at the Max Planck Institute for Mathematics in the Sciences in
Leipzig for their hospitality and for giving us the opportunity to present
this (and related) work. A special thanks goes to Massimiliano Gubinelli for
his feedback and for his idea of using the exponential martingale inequality
for the estimate in Lemma \ref{stochexponential}, which simplified the proof
significantly. Finally, we thank Zimo Hao  and  Xicheng Zhang for sending us
their work {\cite{HaoZhang23}} before its publication on arXiv.

We are grateful for funding by DFG through EXC 2046, Berlin Mathematical
School, through IRTG 2544 ``Stochastic Analysis in Interaction''. We also are grateful for seed support for DFG CRC/TRR 388 ``Rough Analysis, Stochastic Dynamics and Related Topics''.

\begin{small}

\end{small}


\begin{thebibliography}{BFGM19}
  \bibitem[ABK24]{Armstrong2024}Scott Armstrong, Ahmed Bou-Rabee, and  Tuomo
  Kuusi. {\newblock}Superdiffusive central limit theorem for a Brownian
  particle in a critically-correlated incompressible random drift.
  {\newblock}\tmtextit{ArXiv preprint arXiv:2404.01115}, 2024.{\newblock}

  \bibitem[ALL23]{AnzelettiLeLing23}Lukas Anzeletti, Khoa L{\^e}, and 
  Chengcheng Ling. {\newblock}Path-by-path uniqueness for stochastic
  differential equations under Krylov-R{\"o}ckner condition.
  {\newblock}\url{https://arxiv.org/pdf/2304.06802}, 2023.
  {\newblock}Preprint.{\newblock}
  
  \bibitem[BC01]{Bass2001}Richard~F.~Bass  and  Zhen-Qing Chen.
  {\newblock}Stochastic differential equations for Dirichlet processes.
  {\newblock}\tmtextit{Probab. Theory Related Fields}, 121(3):422--446,
  2001.{\newblock}
  
  \bibitem[BCD11]{BahouriCheminDanchin11}Hajer Bahouri, Jean-Yves Chemin, and 
  Rapha{\"e}l Danchin. {\newblock}\tmtextit{Fourier analysis and nonlinear
  partial differential equations},  volume  343  of \tmtextit{Grundlehren der
  Mathematischen Wissenschaften}. {\newblock}Springer, 2011.{\newblock}
  
  \bibitem[BFGM19]{BeckFlandoliGubinelliMaurelli14}Lisa Beck, Franco Flandoli,
  Massimiliano Gubinelli, and  Mario Maurelli. {\newblock}Stochastic ODEs and
  stochastic linear PDEs with critical drift: regularity, duality and
  uniqueness. {\newblock}\tmtextit{Electronic Journal of Probability},
  24(136):1--72, 2019.{\newblock}
  
  \bibitem[BGL13]{BakryGentilLedoux14}Dominique Bakry, Ivan Gentil, and 
  Michel Ledoux. {\newblock}\tmtextit{Analysis and Geometry of Markov
  Diffusion} \tmtextit{Operators},  volume  348  of \tmtextit{Grundlehren der
  mathematischen Wissenschaften}. {\newblock}Springer Cham, 1  edition,
  2013.{\newblock}
  
  \bibitem[CC18]{CannizzaroChouk18}Giuseppe Cannizzaro  and  Khalil Chouk.
  {\newblock}Multidimensional SDEs with singular drift and universal
  construction of the polymer measure with white noise potential.
  {\newblock}\tmtextit{Annals of Probability}, 46(3):1710--1763,
  2018.{\newblock}
  
  \bibitem[CET23]{CannizzaroErhardToninelli21}Giuseppe Cannizzaro, Dirk
  Erhard, and  Fabio Toninelli. {\newblock}Weak coupling limit of the
  anisotropic KPZ equation. {\newblock}\tmtextit{Duke Math. J.},
  172(16):3013--3104, 2023.{\newblock}
  
  \bibitem[CF90]{ChiarenzaFrasca1990}Filippo Chiarenza  and  Michele Frasca.
  {\newblock}A Remark on a Paper by C. Fefferman.
  {\newblock}\tmtextit{Proceedings of the American Mathematical Society},
  108(2):407--409, February 1990.{\newblock}
  
  \bibitem[CGT23]{CannizzaroGubinelliToninelli23}Giuseppe Cannizzaro,
  Massimiliano Gubinelli, and  Fabio Toninelli. {\newblock}Gaussian
  Fluctuations for the stochastic Burgers equation in dimension $d{\geq}2$.
  {\newblock}\url{https://arxiv.org/pdf/2304.05730.pdf}, 2023.
  {\newblock}Preprint.{\newblock}
  
  \bibitem[CHT21]{CannizzaroHaunschmidSibitzToninelli21}Giuseppe Cannizzaro,
  Levi Haunschmid-Sibitz, and  Fabio Toninelli. {\newblock}Sqrt(log
  t)-superdiffusivity for a Brownian particle in the curl of the 2d GFF.
  {\newblock}\tmtextit{Annals of Probability}, 50(6):2475--2498,
  2021.{\newblock}
  
  \bibitem[CMOW22]{ChatzigeorgiouMorfeOttoWang23}Georgiana Chatzigeorgiou,
  Peter Morfe, Felix Otto, and  Lihan Wang. {\newblock}The Gaussian free-field
  as a stream function: asymptotics of effective diffusivity in infra-red
  cut-off. {\newblock}\url{https://arxiv.org/pdf/2212.14244.pdf}, 2022.
  {\newblock}Preprint.{\newblock}
  
  \bibitem[CQ02]{Coutin2002}Laure Coutin  and  Zhongmin Qian.
  {\newblock}Stochastic analysis, rough path analysis and fractional Brownian
  motions. {\newblock}\tmtextit{Probab. Theory Related Fields},
  122(1):108--140, 2002.{\newblock}
  
  \bibitem[CZZ21]{ChenZhangZhao21}Zhen-Qing Chen, Xicheng Zhang, and  Guohuan
  Zhao. {\newblock}Supercritical SDEs driven by multiplicative stable-like
  L\'{}evy processes. {\newblock}\tmtextit{Transactions of the American
  Mathematical Society}, 374(11):7621--7655, 2021.{\newblock}
  
  \bibitem[Dav76]{Burgess76}Burgess Davis. {\newblock}On the Lp norms of
  stochastic integrals and other martingales. {\newblock}\tmtextit{Duke
  Mathematical Journal}, 43(4):697--704, 1976.{\newblock}
  
  \bibitem[DD16]{DelarueDiel16}Fran{\c c}ois Delarue  and  Roland Diel.
  {\newblock}Rough paths and 1d SDE with a time dependent distributional
  drift: application to polymers. {\newblock}\tmtextit{Probability Theory and
  Related Fields}, 165(1-2):1--63, 2016.{\newblock}
  
  \bibitem[dko18]{dkoutsou2018}Chi-squared distribution tail bound.
  {\newblock}\url{https://math.stackexchange.com/questions/2864188/chi-squared-distribution-tail-bound},
  2018. {\newblock}QA website.{\newblock}
  
  \bibitem[DL89]{DiPerna1989}R.~J.~DiPerna  and  P.-L.~Lions.
  {\newblock}Ordinary differential equations, transport theory and Sobolev
  spaces. {\newblock}\tmtextit{Invent. Math.}, 98(3):511--547,
  1989.{\newblock}
  
  \bibitem[EK86]{Ethier1986}Stewart~N.~Ethier  and  Thomas~G.~Kurtz.
  {\newblock}\tmtextit{Markov processes: Characterization and} \tmtextit{convergence}.
  {\newblock}John Wiley \& Sons, 1986.{\newblock}
  
  \bibitem[EN00]{EngelNagel00}Klaus-Jochen Engel  and  Rainer Nagel.
  {\newblock}\tmtextit{One-parameter semigroups for linear evolution}  \tmtextit{equations
  (Graduate Texts in Mathematics 194)}. {\newblock}Springer, 2000.{\newblock}
  
  \bibitem[Eva98]{Evans98}Lawrence~C.~Evans. {\newblock}\tmtextit{Partial
  Differential Equations},  volume~19  of \tmtextit{Graduate Studies in
  Mathematics}. {\newblock}American Mathematical Society, 1  edition,
  1998.{\newblock}
  
  \bibitem[FGP10]{Flandoli2010}F.~Flandoli, M.~Gubinelli, and  E.~Priola.
  {\newblock}Well-posedness of the transport equation by stochastic
  perturbation. {\newblock}\tmtextit{Invent. Math.}, 180(1):1--53,
  2010.{\newblock}
  
  \bibitem[Fig08]{Figalli2008}Alessio Figalli. {\newblock}Existence and
  uniqueness of martingale solutions for SDEs with rough or degenerate
  coefficients. {\newblock}\tmtextit{J. Funct. Anal.}, 254(1):109--153,
  2008.{\newblock}
  
  \bibitem[FIR17]{FlandoliIssoglioRusso17}Franco Flandoli, Elena Issoglio, and
  Francesco Russo. {\newblock}Multidimensional stochastic differential
  equations with distributional drift. {\newblock}\tmtextit{Transactions of
  the American Mathematical Society}, 369:1665--1688, 2017.{\newblock}
  
  \bibitem[Fla11]{Flandoli2011}Franco Flandoli. {\newblock}\tmtextit{Random
  Perturbation of PDEs and Fluid Dynamic Models:}\\ \tmtextit{{\'E}cole d'{\'e}t{\'e} de
  Probabilit{\'e}s de Saint-Flour XL--2010},  volume  2015.
  {\newblock}Springer Science \& Business Media, 2011.{\newblock}
  
  \bibitem[FRW03]{Flandoli2003}Franco Flandoli, Francesco Russo, and  Jochen
  Wolf. {\newblock}Some SDEs with distributional drift. I. General calculus.
  {\newblock}\tmtextit{Osaka J. Math.}, 40(2):493--542, 2003.{\newblock}
  
  \bibitem[FW22]{FeltesWeber22}Guilherme~L.~Feltes  and  Hendrik Weber.
  {\newblock}Brownian particle in the curl of 2-d stochastic heat equations.
  {\newblock}\url{https://arxiv.org/pdf/2211.02194.pdf}, 2022.
  {\newblock}Preprint.{\newblock}
  
  \bibitem[GIP15]{Gubinelli2015Paracontrolled}Massimiliano Gubinelli, Peter
  Imkeller, and  Nicolas Perkowski. {\newblock}Paracontrolled distributions
  and singular PDEs. {\newblock}\tmtextit{Forum of Mathematics, Pi}, 3(e6),
  2015.{\newblock}
  
  \bibitem[GJ13]{GubinelliJara13}Massimiliano Gubinelli  and  Milton Jara.
  {\newblock}Regularization by noise and stochastic Burgers equations.
  {\newblock}\tmtextit{Stochastic Partial Differential Equations: Analysis and
  Computations}, 1(2):325--350, 2013.{\newblock}
  
  \bibitem[GJ14]{GoncalvesJara14}Patricia Gon{\c c}alves  and  Milton Jara.
  {\newblock}Nonlinear Fluctuations of Weakly Asymmetric Interacting Particle
  Systems. {\newblock}\tmtextit{Archive for Rational Mechanics and Analysis},
  212:597--644, 2014.{\newblock}
  
  \bibitem[GP18]{GubinelliPerkowski18}Massimiliano Gubinelli  and  Nicolas
  Perkowski. {\newblock}Energy solutions of KPZ are unique.
  {\newblock}\tmtextit{Journal of the American Mathematical Society},
  31(2):427--471, 2018.{\newblock}
  
  \bibitem[GP20]{GubinelliPerkowski20}Massimiliano Gubinelli  and  Nicolas
  Perkowski. {\newblock}The infinitesimal generator of the stochastic Burgers
  equation. {\newblock}\tmtextit{Probability Theory and Related Fields},
  178:1067--1124, 2020.{\newblock}
  
  \bibitem[GP23]{GraefnerPerkowski23}Lukas Gr{\"a}fner  and  Nicolas
  Perkowski. {\newblock}Energy solutions and generators of singular SPDEs.
  {\newblock}\url{https://files-www.mis.mpg.de/mpi-typo3/events-files/slides_342_5751.pdf},
  May 2023. {\newblock}Summer School on PDEs and Randomness, lecture
  notes.{\newblock}
  
  \bibitem[GP24+]{GraefnerPerkowski23plus}Lukas Gr{\"a}fner  and  Nicolas
  Perkowski. {\newblock}Weak well-posedness of energy solutions to some
  critical singular SPDEs. {\newblock}2024+. {\newblock}Preprint.{\newblock}
  
  \bibitem[GT20]{GubinelliTurra20}Massimiliano Gubinelli  and  Mattia Turra.
  {\newblock}Hyperviscous stochastic Navier--Stokes equations with white noise
  invariant measure. {\newblock}\tmtextit{Stochastics and
  Dynamics\href{https://www.worldscientific.com/toc/sd/20/06}{}}, 20(6),
  2020.{\newblock}
  
  \bibitem[Hai14]{Hairer2014}Martin Hairer. {\newblock}A theory of regularity
  structures. {\newblock}\tmtextit{Invent. Math.}, 198(2):269--504,
  2014.{\newblock}
  
  \bibitem[Hai24]{Hairer2024}Martin Hairer. {\newblock}Renormalisation in the
  presence of variance blowup. {\newblock}\tmtextit{ArXiv:2401.10868},
  2024.{\newblock}
  
  \bibitem[HP86]{HaussmannPardoux86}Ulrich Haussmann  and  {\'E}tienne
  Pardoux. {\newblock}Time reversal of diffusions. {\newblock}\tmtextit{The
  Annals of Probability}, 14(4):1188--1205, 1986.{\newblock}
  
  \bibitem[HZ23]{HaoZhang23}Zimo Hao  and  Xicheng Zhang. {\newblock}SDES WITH
  SUPERCRITICAL DISTRIBUTIONAL DRIFTS.
  {\newblock}\url{https://arxiv.org/pdf/2312.11145}, 2023.
  {\newblock}Preprint.{\newblock}
  
  \bibitem[JS03]{Jacod2003}Jean Jacod  and  Albert~N.~Shiryaev.
  {\newblock}\tmtextit{Limit theorems for stochastic processes}.
  {\newblock}Springer, 2nd  edition, 2003.{\newblock}
  
  \bibitem[Kat95]{Kato80}Tosio Kato. {\newblock}\tmtextit{Perturbation Theory
  for Linear Operators (Grundlehren der}\\ \tmtextit{mathematischen Wissenschaften)}, 
  volume  132. {\newblock}Springer-Verlag, 1995. {\newblock}Reprint of the
  1980 Edition.{\newblock}
  
  \bibitem[KLvR19]{Konarovskyi2019}Vitalii Konarovskyi, Tobias Lehmann, and 
  Max-K.~von~Renesse. {\newblock}Dean-Kawasaki dynamics: ill-posedness vs.
  triviality. {\newblock}\tmtextit{Electron. Commun. Probab.}, 24:0,
  2019.{\newblock}
  
  \bibitem[Kol37]{Kolmogorov37}Andrey Kolmogorov. {\newblock}Zur Umkehrbarkeit
  der statistischen Naturgesetze. {\newblock}\tmtextit{Mathematische Annalen},
  113:766--772, 1937.{\newblock}
  
  \bibitem[KP22]{KrempPerkowski22}Helena Kremp  and  Nicolas Perkowski.
  {\newblock}Multidimensional SDE with distributional drift and L{\'e}vy
  noise. {\newblock}\tmtextit{Bernoulli}, 28(3):1757--1783, 2022.{\newblock}
  
  \bibitem[KP23]{Kremp2023}Helena Kremp  and  Nicolas Perkowski.
  {\newblock}Rough weak solutions for singular L{\'e}vy SDEs.
  {\newblock}\tmtextit{ArXiv:2309.15460}, 2023.{\newblock}
  
  \bibitem[KR05]{Krylov2005}N.~V.~Krylov  and  M.~R{\"o}ckner.
  {\newblock}Strong solutions of stochastic equations with singular time
  dependent drift. {\newblock}\tmtextit{Probab. Theory Related Fields},
  131(2):154--196, 2005.{\newblock}
  
  \bibitem[Kry23]{Krylov2023}N.~V.~Krylov. {\newblock}On strong solutions of
  It{\^o}'s equations with $D \sigma$ and $b$ in Morrey classes containing
  $L_d$. {\newblock}\tmtextit{Ann. Probab.}, 51(5):1729--1751,
  2023.{\newblock}
  
  \bibitem[Lee65]{Leeuw65}Karel~de Leeuw. {\newblock}On L\tmrsub{$\mathrm{p}$}
  Multipliers. {\newblock}\tmtextit{Annals of Mathematics}, 81(2):364--79,
  1965.{\newblock}
  
  \bibitem[LL19]{LeBris2019}Claude Le Bris  and  Pierre-Louis Lions.
  {\newblock}\tmtextit{Parabolic equations with irregular data and} \tmtextit{related
  issues---applications to stochastic differential equations},  volume~4  of
  \tmtextit{De Gruyter Series in Applied and Numerical Mathematics}.
  {\newblock}De Gruyter, Berlin, [2019] {\textcopyright}2019.{\newblock}
  
  \bibitem[LST22]{Lee2022}Haesung Lee, Wilhelm Stannat, and  Gerald Trutnau.
  {\newblock}\tmtextit{Analytic theory of It{\^o}-stochastic differential}
  \tmtextit{equations with non-smooth coefficients}. {\newblock}SpringerBriefs in
  Probability and Mathematical Statistics. Springer, Singapore,
  2022.{\newblock}
  
  \bibitem[LZ21]{LuoZhu21}Dejun Luo  and  Rongchan Zhu. {\newblock}Stochastic
  mSQG equations with multiplicative transport noises: White noise solutions
  and scaling limit. {\newblock}\tmtextit{Stochastic Processes and their
  Applications}, 140:236--286, 2021.{\newblock}
  
  \bibitem[Mat94]{Mathieu1994}Pierre Mathieu. {\newblock}Zero white noise
  limit through Dirichlet forms, with application to diffusions in a random
  medium. {\newblock}\tmtextit{Probab. Theory Related Fields}, 99(4):549--580,
  1994.{\newblock}
  
  \bibitem[MS18]{Modena2018}Stefano Modena  and  L{\'a}szl{\'o}
  Sz{\'e}kelyhidi, Jr. {\newblock}Non-uniqueness for the transport equation
  with Sobolev vector fields. {\newblock}\tmtextit{Ann. PDE}, 4(2):0,
  2018.{\newblock}
  
  \bibitem[MV06]{MazjaVerbitsky06}Vladimir Mazja  and  Igor Verbitsky.
  {\newblock}Form boundedness of the general second-order differential
  operator. {\newblock}\tmtextit{Communications on Pure and Applied
  Mathematics}, 59(9):1286--1329, 2006.{\newblock}
  
  \bibitem[Oel88]{Oelschlaeger88}Karl Oelschl{\"a}ger.
  {\newblock}Homogenization of a diffusion process in divergence-free random
  field. {\newblock}\tmtextit{The Annals of Probability}, 16(3):1084--1126,
  1988.{\newblock}
  
  \bibitem[Osa87]{Osada1987}Hirofumi Osada. {\newblock}Diffusion processes
  with generators of generalized divergence form. {\newblock}\tmtextit{J.
  Math. Kyoto Univ.}, 27(4):597--619, 1987.{\newblock}
  
  \bibitem[Per02]{Perkins2002}Edwin Perkins. {\newblock}Dawson-Watanabe
  superprocesses and measure-valued diffusions. {\newblock}In
  \tmtextit{Lectures on probability theory and statistics (Saint-Flour,
  1999)},  volume  1781  of \tmtextit{Lecture Notes in Math.},  pages 
  125--324. Springer, Berlin, 2002.{\newblock}
  
  \bibitem[QX18]{QianXi18}Zhongmin Qian  and  Guangyu Xi. {\newblock}Parabolic
  equations with singular divergence-free drift vector fields.
  {\newblock}\tmtextit{Journal of the London Mathematical Society},
  100(1):17--40, 2018.{\newblock}
  
  \bibitem[RS80]{SimonReed80}Michael Reed  and  Barry Simon.
  {\newblock}\tmtextit{I: Functional Analysis},  volume~1  of
  \tmtextit{Methods of modern Mathematical Physics}. {\newblock}Academic
  Press, 1  edition, 1980.{\newblock}
  
  \bibitem[RZ23]{RoecknerZhao23}Michael R{\"o}ckner  and  Guohuan Zhao.
  {\newblock}SDEs with critical time dependent drifts: Weak solutions.
  {\newblock}\tmtextit{Bernoulli}, 29(1):757--784, February 2023.{\newblock}
  
  \bibitem[SS{\v S}Z12]{SereginSilvestreSverakZlatos2010}Gregory Seregin, Luis
  Silvestre, Vladimir {\v S}ver{\'a}k, and  Andrej Zlato{\v s}. {\newblock}On
  divergence-free drifts. {\newblock}\tmtextit{Journal of Differential
  Equations}, 252(1):505--540, 2012.{\newblock}
  
  \bibitem[Sta99]{Stannat1999}Wilhelm Stannat. {\newblock}The theory of
  generalized Dirichlet forms and its applications in analysis and
  stochastics. {\newblock}\tmtextit{Mem. Amer. Math. Soc.}, 142(678):0,
  1999.{\newblock}
  
  \bibitem[Sze75]{Szegoe75}Gabor Szeg{\"o}. {\newblock}\tmtextit{Orthogonal
  Polynomials},  volume~23  of \tmtextit{Colloquium Publications}.
  {\newblock}American Mathematical Society, 4  edition, 1975.{\newblock}
  
  \bibitem[Tri83]{Triebel83}Hans Triebel. {\newblock}\tmtextit{Theory of
  function spaces}. {\newblock}Birkhauser, 1983.{\newblock}
  
  \bibitem[Tri06]{Triebel2006}Hans Triebel. {\newblock}\tmtextit{Theory of
  function spaces. III},  volume  100  of \tmtextit{Monographs in
  Mathematics}. {\newblock}Birkh{\"a}user Verlag, Basel, 2006.{\newblock}
  
  \bibitem[TV12]{Toth2012}B{\'a}lint T{\'o}th  and  Benedek Valk{\'o}.
  {\newblock}Superdiffusive bounds on self-repellent Brownian polymers and
  diffusion in the curl of the Gaussian free field in $d = 2$.
  {\newblock}\tmtextit{J. Stat. Phys.}, 147:113--131, 2012.{\newblock}
  
  \bibitem[Ver81]{Veretennikov1981}Alexander~Yu.~Veretennikov. {\newblock}On
  strong solution and explicit formulas for solutions of stochastic integral
  equations. {\newblock}\tmtextit{Math. USSR Sb.}, 39:387--403,
  1981.{\newblock}
  
  \bibitem[Ver11]{Veraar11}Mark~C.~Veraar. {\newblock}Regularity of Gaussian
  white noise on the d-dimensional torus. {\newblock}\tmtextit{Banach Center
  Publications}, 95(1):385--398, 2011.{\newblock}
  
  \bibitem[Wer18]{Werner2018}Dirk Werner.
  {\newblock}\tmtextit{Funktionalanalysis}. {\newblock}Springer-Lehrbuch.
  Springer Spektrum Berlin, Heidelberg, 8  edition, 2018.{\newblock}
  
  \bibitem[YY24]{Yang2024}Huanyu Yang  and  Zhilin Yang. {\newblock}Weak
  coupling limit of a Brownian particle in the curl of the 2d GFF.
  {\newblock}\tmtextit{ArXiv:2405.05778}, 2024.{\newblock}
  
  \bibitem[Zui22]{vanZuijlen22}Willem~van Zuijlen. {\newblock}Theory of
  function spaces.
  {\newblock}\url{https://www.wias-berlin.de/people/vanzuijlen/teaching.html},
  2022. {\newblock}Lecture notes.{\newblock}
  
  \bibitem[Zvo74]{Zvonkin1974}A.~K.~Zvonkin. {\newblock}A transformation of
  the phase space of a diffusion process that will remove the drift.
  {\newblock}\tmtextit{Mat. Sb. (N.S.)}, 93(135):129--149, 1974.{\newblock}
  
  \bibitem[ZZ21]{ZhangZhao21}Xicheng Zhang  and  Guohuan Zhao.
  {\newblock}Stochastic Lagrangian Path for Leray's Solutions of 3D
  Navier--Stokes Equations. {\newblock}\tmtextit{Communications in
  Mathematical Physics}, 381(2):1--35, 2021.{\newblock}
\end{thebibliography}
\end{document}